\documentclass[11pt]{article}
\usepackage{latexsym}
\usepackage{xypic}
\usepackage{setspace}
\usepackage[all]{xy}
\usepackage{graphicx}

\usepackage{amsfonts}
\usepackage{amsthm}
\usepackage{amsmath}
\usepackage{amssymb}
\usepackage{amscd}
\usepackage{verbatim}
\usepackage{multicol}
\usepackage{tabularx}

\newcommand{\spann}{\text{span}}
\newcommand{\nrm}{\ \triangleleft \ }
\newcommand{\Ext}{\operatorname{Ext}^1}
\newcommand{\nin}{\ \not\in \ }

\def\sn{-\hspace{-.17cm}-}

\newtheorem{theorem}{Theorem}[section]
\newtheorem{lemma}[theorem]{Lemma}
\newtheorem{proposition}[theorem]{Proposition}
\newtheorem{corollary}[theorem]{Corollary}
\newtheorem{problem}[theorem]{Problem}
\theoremstyle{definition}
\newtheorem{definition}[theorem]{Definition}

\newtheorem{example}[theorem]{Example}

\newtheorem{remark}[theorem]{Remark}

\newcommand\xx{{\bf x}}
\newcommand{\HH}{{\mathbb H}}

\newcommand{\NN}{{\mathbb{N}}}
\newcommand{\ZZ}{{\mathbb{Z}}}
\newcommand{\RR}{{\mathbb{R}}}
\newcommand{\Tr}{\text{Tr}}

\newcommand{\Image}{\text{Im}}
\newcommand{\Ker}{\text{Ker\,}}

\newcommand{\Hom}{\text{Hom}}

\newcommand{\Id}{\text{Id}}

\newcommand{\Ind}{\text{Ind}}
\newcommand{\Rep}{\text{Rep}}

\newcommand{\tr}{\text{Tr}\,}

\newcommand{\eps}{\varepsilon}
\newcommand{\BA}{{\Bbb A}}
\newcommand{\BQ}{{\Bbb Q}}
\newcommand{\BZ}{{\Bbb Z}}

\providecommand{\abs}[1]{\lvert#1\rvert}

\newcommand{\ov}{\overline}

\newcommand{\ot}{\otimes}

\newcommand{\ben}{\begin{enumerate}}
\newcommand{\een}{\end{enumerate}}

\newcommand{\Rad}{{\text{Rad}}}

\newcommand{\mB}{{\mathcal B}}

\newcommand{\CC}{{\mathbb{C}}}

\newcommand{\cA}{{\mathcal A}}

\newcommand{\BC}{{\mathbb C}}

\newcommand{\End}{\operatorname{End}}
\newcommand{\Mat}{\operatorname{Mat}}

\begin{document}

\title{Introduction to representation theory}

\author{Pavel Etingof, Oleg Golberg, Sebastian Hensel,\\
Tiankai Liu, Alex Schwendner, Dmitry Vaintrob, and Elena Yudovina}

\maketitle

\tableofcontents

\vskip .1in
\centerline{\bf INTRODUCTION}

Very roughly speaking, representation theory studies symmetry in linear
spaces. It is a beautiful mathematical subject which has many
applications, ranging from number theory and
combinatorics to geometry, probability theory,
quantum mechanics and quantum field theory.

Representation theory was born in 1896 in the work of
the German mathematician F. G. Frobenius. This work
was triggered by a letter to Frobenius by R. Dedekind.
In this letter Dedekind made the following observation:
take the multiplication table of a finite group $G$ and turn
it into a matrix $X_G$ by replacing every entry $g$ of this table by a
variable $x_g$. Then the determinant of $X_G$
factors into a product of irreducible polynomials in $\lbrace{x_g\rbrace}$,
each of which occurs with multiplicity equal to its degree.
Dedekind checked this surprising fact in a few special cases, but
could not prove it in general. So he gave this problem to
Frobenius. In order to find a solution of this problem
(which we will explain below), Frobenius created representation
theory of finite groups. \footnote{For more on the history of
representation theory, see \cite{Cu}.}

The present lecture notes arose from a representation theory
course given by the first author to the remaining six authors
in March 2004 within the framework of the Clay Mathematics
Institute Research Academy for high school students,
and its extended version given by the first
author to MIT undergraduate math students in the Fall of 2008.
The lectures are supplemented by many problems and exercises,
which contain a lot of additional material; the more difficult
exercises are provided with hints.

The notes cover a number of standard topics in representation theory
of groups, Lie algebras, and quivers. We mostly follow
\cite{FH}, with the exception of the sections
discussing quivers, which follow \cite{BGP}.
We also recommend the comprehensive textbook \cite{CR}.
The notes should be accessible to students with a strong
background in linear algebra and a basic knowledge of abstract algebra.

{\bf Acknowledgements.} The authors are grateful to the Clay
Mathematics Institute for hosting the first version of this
course. The first author is very indebted to Victor Ostrik
for helping him prepare this course, and thanks Josh
Nichols-Barrer and Thomas Lam for helping run the course
in 2004 and for useful comments. He is also very grateful 
to Darij Grinberg for very careful reading of the text,
for many useful comments and corrections, and for suggesting 
the Exercises in Sections 1.10, 2.3, 3.5, 4.9, 4.26, and 6.8.

\newpage \section{Basic notions of representation theory}

\subsection{What is representation theory?}

In technical terms, representation theory studies representations
of associative algebras. Its general content can be very briefly
summarized as follows.

An {\bf associative algebra} over a field $k$ is a vector space $A$
over $k$ equipped with an associative bilinear multiplication
$a,b\mapsto ab$, $a,b\in A$. We will always consider associative algebras
with unit, i.e., with an element $1$
such that $1\cdot a=a\cdot 1=a$ for all $a\in A$.
A basic example of an associative algebra is the algebra ${\rm
End}V$ of linear operators from a vector space $V$ to itself.
Other important examples include algebras defined by generators
and relations, such as group algebras and universal enveloping algebras of
Lie algebras.

A {\bf representation} of an associative algebra $A$ (also called a
left $A$-module) is a vector space
$V$ equipped with a homomorphism $\rho: A\to {\rm End}V$, i.e., a
linear map preserving the multiplication and unit.

A {\bf subrepresentation} of a representation $V$ is a subspace $U\subset
V$ which is invariant under all operators $\rho(a)$, $a\in A$.
Also, if $V_1,V_2$ are two representations of $A$ then the {\bf direct sum}
$V_1\oplus V_2$ has an obvious structure of a representation of $A$.

A nonzero representation $V$ of $A$ is said to be {\bf irreducible} if its only
subrepresentations are $0$ and $V$ itself, and
{\bf indecomposable} if it cannot be written as a direct sum of two
nonzero subrepresentations. Obviously, irreducible implies
indecomposable, but not vice versa.

Typical problems of representation theory are as follows:

1. Classify irreducible representations of a given algebra $A$.

2. Classify indecomposable representations of $A$.

3. Do 1 and 2 restricting to finite dimensional representations.

As mentioned above, the algebra $A$ is often given to us by
generators and relations. For example, the universal enveloping
algebra $U$ of the Lie algebra ${\mathfrak{sl}}(2)$ is generated by $h,e,f$
with defining relations
\begin{equation}\label{sl2}
he-eh=2e,\quad hf-fh=-2f,\quad ef-fe=h.
\end{equation}
This means that the problem of finding,
say, $N$-dimensional representations of $A$ reduces to solving
a bunch of nonlinear algebraic equations with respect to a bunch
of unknown $N$ by $N$ matrices, for example system (\ref{sl2})
with respect to unknown matrices $h,e,f$.

It is really striking that such,
at first glance hopelessly complicated, systems of
equations can in fact be solved completely by methods of
representation theory! For example, we will prove the following
theorem.

\begin{theorem} Let $k=\Bbb C$ be the field of complex numbers. Then:

(i) The algebra $U$ has exactly one irreducible representation
$V_d$ of each dimension, up to equivalence; this representation is
realized in the space of homogeneous polynomials of two variables $x,y$
of degree $d-1$, and defined by the formulas
$$
\rho(h)=x\frac{\partial}{\partial x}-y\frac{\partial}{\partial y},\quad
\rho(e)=x\frac{\partial}{\partial y},\quad
\rho(f)=y\frac{\partial}{\partial x}.
$$

(ii) Any indecomposable finite dimensional representation of $U$
is irreducible. That is, any finite dimensional representation
of $U$ is a direct sum of irreducible representations.
\end{theorem}

As another example consider the representation
theory of quivers.

A {\bf quiver} is a finite oriented graph $Q$. A {\bf
representation} of $Q$ over a field $k$ is an assignment
of a $k$-vector space $V_i$ to every vertex $i$ of $Q$, and of a
linear operator $A_h: V_i\to V_j$ to every directed edge $h$
going from $i$ to $j$ (loops and multiple
edges are allowed). We will show that a representation
of a quiver $Q$ is the same thing as a representation
of a certain algebra $P_Q$ called the path algebra of $Q$.
Thus one may ask: what are the indecomposable
finite dimensional representations of $Q$?

More specifically, let us say that $Q$ is {\bf of finite type}
if it has finitely many indecomposable representations.

We will prove the following striking theorem, proved by
P. Gabriel about 35 years ago:

\begin{theorem} The finite type property of $Q$ does not depend on
the orientation of edges. The connected graphs that yield quivers
of finite type are given by the following list:
\begin{itemize}

\item $A_n$ : {\large \vspace{-.55cm} $$\stackrel{}{\circ}\hspace{-.2cm}\sn\hspace{-.2cm}\stackrel{}{\circ} \dots \stackrel{\text{}}{\circ}\hspace{-.29cm}\sn\hspace{-.18cm}\stackrel{\text{}}{\circ}$$}

\item $D_n$: {\large \vspace{-.95cm} $$\stackrel{}{\circ}\hspace{-.2cm}\sn\hspace{-.2cm}\stackrel{}{\circ} \dots \stackrel{\text{}}{\circ}\hspace{-.29cm}\sn\hspace{-.23cm}\stackrel{\text{}}{\circ}$$
\vspace{-0.85cm}$$\hspace{1.03cm}|$$
\vspace{-0.9cm}$$\hspace{1.05cm}\circ\text{}$$}
\\
\item $E_6$  : {\large \vspace{-.55cm} $$\stackrel{}{\circ}\hspace{-.2cm}\sn\hspace{-.2cm}\stackrel{}{\circ} \hspace{-.2cm}\sn\hspace{-.2cm} \stackrel{\text{}}{\circ}\hspace{-.2cm}\sn\hspace{-.2cm}\stackrel{}{\circ}\hspace{-.2cm}\sn \hspace{-.2cm}\stackrel{}{\circ}$$
\vspace{-.87cm}$$\hspace{.04cm}|$$
\vspace{-.97cm}$$\hspace{-.02cm}\circ\text{}$$}

\item $E_7$  : {\large \vspace{-.55cm} $$\hspace{.5cm}\stackrel{}{\circ}\hspace{-.2cm}\sn\hspace{-.2cm}\stackrel{}{\circ} \hspace{-.2cm}\sn\hspace{-.2cm} \stackrel{\text{}}{\circ}\hspace{-.2cm}\sn\hspace{-.2cm}\stackrel{}{\circ}\hspace{-.2cm}\sn \hspace{-.2cm}\stackrel{}{\circ}\hspace{-.2cm}\sn\hspace{-.2cm}\stackrel{}{\circ}$$
\vspace{-0.9cm}$$\hspace{1.02cm}|$$
\vspace{-0.9cm}$$\hspace{1.02cm}\circ\text{}$$}

\item $E_8$  : {\large \vspace{-1.cm} $$\hspace{1cm}\stackrel{}{\circ}\hspace{-.2cm}\sn\hspace{-.2cm}\stackrel{}{\circ} \hspace{-.2cm}\sn\hspace{-.2cm} \stackrel{\text{}}{\circ}\hspace{-.2cm}\sn\hspace{-.2cm}\stackrel{}{\circ}\hspace{-.2cm}\sn \hspace{-.2cm}\stackrel{}{\circ}\hspace{-.2cm}\sn\hspace{-.2cm}\stackrel{}{\circ} \hspace{-.2cm}\sn\hspace{-.2cm}\stackrel{}{\circ}$$
\vspace{-1.00cm}$$\hspace{2.05cm}|$$
\vspace{-0.9cm}$$\hspace{2.05cm}\circ\text{}$$}

\end{itemize}

\end{theorem}

The graphs listed in the theorem are called (simply laced) {\bf Dynkin
diagrams}. These graphs arise in a multitude of classification
problems in mathematics, such as classification of simple Lie
algebras, singularities, platonic solids, reflection groups,
etc. In fact, if we needed to make contact with an alien
civilization and show them how sophisticated our
civilization is, perhaps showing them Dynkin diagrams would be
the best choice!

As a final example consider the representation theory of finite
groups, which is one of the most fascinating chapters
of representation theory. In this theory,
one considers representations of the group
algebra $A=\Bbb C[G]$ of a finite group
$G$ -- the algebra with basis $a_g, g\in G$
and multiplication law $a_ga_h=a_{gh}$.
We will show that any finite dimensional
representation of $A$ is a direct sum of irreducible
representations, i.e., the notions of an irreducible and
indecomposable representation are the same for $A$ (Maschke's theorem).
Another striking result discussed below is the Frobenius
divisibility theorem:
the dimension of any irreducible representation of $A$ divides the
order of $G$. Finally, we will show how to use representation
theory of finite groups to prove Burnside's theorem:
any finite group of order $p^aq^b$, where $p,q$ are primes, is
solvable. Note that this theorem does not mention
representations, which are used only in its proof;
a purely group-theoretical proof of this theorem (not using
representations) exists but is much more difficult!

\subsection{Algebras}
Let us now begin a systematic discussion of representation
theory.

Let $k$ be a field. Unless stated otherwise, we will always assume that
$k$ is algebraically closed, i.e., any nonconstant polynomial
with coefficients in $k$ has a root in $k$. The main example
is the field of complex numbers $\Bbb C$, but we will also
consider fields of characteristic $p$, such as the algebraic
closure $\overline{\Bbb F}_p$ of the finite field $\Bbb F_p$
of $p$ elements.

\begin{definition}
An associative algebra over $k$ is a vector space $A$ over $k$
together with a bilinear map $A\times A\to A$,
$(a,b)\mapsto ab$, such that $(ab)c=a(bc)$.
\end{definition}

\begin{definition}
A unit in an associative algebra $A$ is an element $1\in A$
such that $1a=a1=a$.
\end{definition}

\begin{proposition}
If a unit exists, it is unique.
\end{proposition}

\begin{proof}
Let $1,1'$ be two units. Then $1=11'=1'$.
\end{proof}

 From now on, by an algebra $A$ we will mean an associative algebra
with a unit. We will also assume that $A\ne 0$.

\begin{example} Here are some examples of algebras over $k$:

1. $A=k$.

2. $A=k[x_1,...,x_n]$ -- the algebra of polynomials
in variables $x_1,...,x_n$.

3. $A={\rm End}V$ -- the algebra of endomorphisms
of a vector space $V$ over $k$ (i.e., linear maps, or operators, from $V$ to
itself). The multiplication is given by composition of operators.

4. The free algebra $A=k\langle x_1,...,x_n\rangle$. A basis of
this algebra consists of words in letters $x_1,...,x_n$,
and multiplication in this basis is simply concatenation of words.

5. The group algebra $A=k[G]$ of a group $G$. Its basis is
$\lbrace{a_g,g\in G\rbrace}$, with multiplication law
$a_ga_h=a_{gh}$.
\end{example}

\begin{definition} An algebra $A$ is commutative if $ab=ba$ for
all $a,b\in A$.
\end{definition}

For instance, in the above examples, $A$ is commutative in cases
$1$ and $2$, but not commutative in cases 3 (if $\dim V>1$), and 4
(if $n>1$). In case 5, $A$ is commutative if and only if $G$ is
commutative.

\begin{definition}
A homomorphism of algebras $f: A\to B$
is a linear map such that $f(xy)=f(x)f(y)$ for all $x,y\in A$, and
$f(1)=1$.
\end{definition}

\subsection{Representations}

\begin{definition}
A representation of an algebra $A$ (also called a left $A$-module)
is a vector space $V$ together with a homomorphism of algebras
$\rho: A\to {\rm End}V$.
\end{definition}

Similarly, a right $A$-module is a space $V$ equipped
with an antihomomorphism $\rho: A\to {\rm End}V$;
i.e., $\rho$ satisfies $\rho(ab)=\rho(b)\rho(a)$ and
$\rho(1)=1$.

The usual abbreviated notation for $\rho(a)v$ is $av$ for a left
module and $va$ for the right module. Then the property that
$\rho$ is an (anti)homomorphism can be written as a kind of
associativity law: $(ab)v=a(bv)$ for left modules, and
$(va)b=v(ab)$ for right modules.

Here are some examples of representations.

\begin{example} 1. $V=0$.

2. $V=A$, and $\rho: A\to {\rm End}A$
is defined as follows: $\rho(a)$ is the operator of left
multiplication by $a$, so that $\rho(a)b=ab$ (the usual product).
This representation is called the {\it regular} representation of $A$.
Similarly, one can equip $A$ with a structure of a right
$A$-module by setting $\rho(a)b:=ba$.

3. $A=k$. Then a representation of $A$ is simply a vector space
over $k$.

4. $A=k\langle x_1,...,x_n\rangle$.
Then a representation of $A$ is just a vector space $V$ over $k$
with a collection of arbitrary linear operators
$\rho(x_1),...,\rho(x_n): V\to V$ (explain why!).
\end{example}

\begin{definition}
A subrepresentation of a representation $V$ of an algebra $A$
is a subspace $W\subset V$ which is invariant under all the operators
$\rho(a): V\to V$, $a\in A$.
\end{definition}

For instance, $0$ and $V$ are always subrepresentations.

\begin{definition}
A representation $V\ne 0$ of $A$ is irreducible (or simple)
if the only subrepresentations of $V$ are $0$ and $V$.
\end{definition}

\begin{definition} Let $V_1,V_2$ be two representations
of an algebra $A$. A homomorphism (or intertwining operator)
$\phi: V_1\to V_2$ is a linear operator which commutes with the
action of $A$, i.e., $\phi(av)=a\phi(v)$ for any $v\in V_1$.
A homomorphism $\phi$ is said to be an isomorphism of
representations if it is an isomorphism of vector spaces.
The set (space) of all homomorphisms of representations $V_1\to
V_2$ is denoted by $\Hom_A(V_1,V_2)$.
\end{definition}

Note that if a linear operator $\phi: V_1\to V_2$ is an isomorphism of
representations then so is the linear operator
$\phi^{-1}: V_2\to V_1$ (check it!).

Two representations between which there exists an isomorphism
are said to be isomorphic. For practical purposes, two isomorphic
representations may be regarded as ``the same'', although
there could be subtleties related to the fact that an isomorphism
between two representations, when it exists, is not unique.

\begin{definition} Let $V_1,V_2$ be representations of an
algebra $A$. Then the space $V_1\oplus V_2$ has an obvious
structure of a representation of $A$, given by $a(v_1\oplus
v_2)=av_1\oplus av_2$.
\end{definition}

\begin{definition} A nonzero representation
$V$ of an algebra $A$ is said to be indecomposable
if it is not isomorphic to a direct sum of two nonzero
representations.
\end{definition}

It is obvious that an irreducible representation is
indecomposable. On the other hand,
we will see below that the converse statement is
false in general.

One of the main problems of representation theory
is to classify irreducible and indecomposable representations of
a given algebra up to isomorphism. This problem is usually hard
and often can be solved only partially (say, for finite
dimensional representations). Below we will see a number of
examples in which this problem is partially or fully solved for
specific algebras.

We will now prove our first result -- Schur's lemma.
Although it is very easy to prove, it is fundamental in the
whole subject of representation theory.

\begin{proposition}\label{shul} (Schur's lemma)
Let $V_1,V_2$ be representations
of an algebra $A$ over any field $F$ (which need not be
algebraically closed). Let $\phi: V_1\to V_2$ be a nonzero
homomorphism of representations. Then:

(i) If $V_1$ is irreducible, $\phi$ is injective; 

(ii) If $V_2$ is irreducible, $\phi$ is surjective. 

Thus, if both $V_1$ and $V_2$ are irreducible, $\phi$ is an isomorphism.
\end{proposition}

\begin{proof}
(i) The kernel $K$ of $\phi$ is a subrepresentation of $V_1$.
Since $\phi\ne 0$, this subrepresentation cannot
be $V_1$. So by irreducibility of $V_1$ we have $K=0$.

(ii) The image $I$ of $\phi$
is a subrepresentation of $V_2$.
Since $\phi\ne 0$, this subrepresentation cannot
be $0$. So by irreducibility of $V_2$ we have $I=V_2$.
\end{proof}

\begin{corollary}\label{slacf} (Schur's lemma for algebraically closed fields)
Let $V$ be a finite dimensional irreducible representation
of an algebra $A$ over an algebraically closed field $k$,
and $\phi: V\to V$ is an intertwining operator.
Then $\phi=\lambda\cdot {\rm Id}$ for some $\lambda\in k$ (a scalar operator).
\end{corollary}

{\bf Remark.} Note that this Corollary is false over the field of real numbers: 
it suffices to take $A=\Bbb C$ (regarded as an $\Bbb R$-algebra), and $V=A$. 

\begin{proof}
Let $\lambda$ be an eigenvalue of $\phi$ (a root of the
characteristic polynomial of $\phi$). It exists since $k$ is an
algebraically closed field. Then the operator
$\phi-\lambda{\rm Id}$ is an intertwining operator $V\to V$,
which is not an isomorphism (since its determinant is zero).
Thus by Proposition \ref{shul} this operator is zero, hence the result.
\end{proof}

\begin{corollary}
Let $A$ be a commutative algebra.
Then every irreducible finite dimensional representation $V$
of $A$ is 1-dimensional.
\end{corollary}

{\bf Remark.} Note that a 1-dimensional representation of any
algebra is automatically irreducible.

\begin{proof} Let $V$ be irreducible.
For any element $a\in A$, the operator $\rho(a): V\to V$ is an
intertwining operator. Indeed,
$$
\rho(a)\rho(b)v=\rho(ab)v=\rho(ba)v=\rho(b)\rho(a)v
$$
(the second equality is true since the algebra is commutative).
Thus, by Schur's lemma, $\rho(a)$ is a scalar operator
for any $a\in A$. Hence every subspace of $V$ is a
subrepresentation. But $V$ is irreducible, so $0$ and $V$ are the only subspaces of $V$.
This means that $\dim V=1$ (since $V\ne 0$).
\end{proof}

\begin{example} 1. $A=k$. Since representations of $A$ are simply
vector spaces, $V=A$ is the only irreducible
and the only indecomposable representation.

2. $A=k[x]$. Since this algebra is commutative, the irreducible
representations of $A$ are its 1-dimensional representations.
As we discussed above, they are defined by a single operator
$\rho(x)$. In the 1-dimensional case, this is just a number from
$k$. So all the irreducible representations of $A$ are
$V_\lambda=k$, $\lambda\in k$, in which the action of $A$ defined by
$\rho(x)=\lambda$. Clearly, these representations are pairwise
non-isomorphic.

The classification of indecomposable representations
of $k[x]$ is more interesting. To obtain it, recall that any linear
operator on a finite dimensional vector space $V$
can be brought to Jordan normal form.
More specifically, recall that the Jordan block
$J_{\lambda,n}$ is the operator
on $k^n$ which in the standard basis
is given by the formulas $J_{\lambda,n}e_i=\lambda e_i+e_{i-1}$
for $i>1$, and $J_{\lambda,n}e_1=\lambda e_1$.
Then for any linear operator $B: V\to V$
there exists a basis of $V$ such that the matrix of $B$ in this
basis is a direct sum of Jordan blocks.
This implies that all the indecomposable
representations of $A$ are $V_{\lambda,n}=k^n$, $\lambda\in k$,
with $\rho(x)=J_{\lambda,n}$. The fact that these representations
are indecomposable and pairwise non-isomorphic follows
from the Jordan normal form theorem (which in particular says
that the Jordan normal form of an operator is unique
up to permutation of blocks).

This example shows that an indecomposable representation
of an algebra need not be irreducible.

3. The group algebra $A=k[G]$, where $G$ is a group.
A representation of $A$ is the same thing as a representation of
$G$, i.e., a vector space $V$ together with a group homomorphism
$\rho: G\to {\rm Aut}(V)$, whre ${\rm Aut}(V)=GL(V)$ denotes the group 
of invertible linear maps from the space $V$ to itself. 
\end{example}

\begin{problem}\label{1:1}
Let $V$ be a nonzero finite dimensional representation of an algebra
$A$. Show that it has an irreducible subrepresentation.
Then show by example that this does not always hold
for infinite dimensional representations.
\end{problem}

\begin{problem}\label{1:2}
Let $A$ be an algebra over a
field $k$. The center $Z(A)$ of $A$
is the set of all elements $z\in A$ which commute with all
elements of $A$. For example, if $A$ is commutative then
$Z(A)=A$.

(a) Show that if $V$ is an irreducible finite dimensional
representation of $A$ then any element $z\in Z(A)$ acts
in $V$ by multiplication by some scalar $\chi_V(z)$.
Show that $\chi_V: Z(A)\to k$ is a homomorphism.
It is called the {\bf central character} of $V$.

(b) Show that if $V$ is an indecomposable finite dimensional representation
of $A$ then for any $z\in Z(A)$, the operator $\rho(z)$
by which $z$ acts in $V$ has only one eigenvalue $\chi_V(z)$,
equal to the scalar by which $z$ acts on some irreducible
subrepresentation of $V$. Thus $\chi_V: Z(A)\to k$ is a
homomorphism, which is again called the central character of $V$.

(c) Does $\rho(z)$ in (b) have to be a scalar operator?
\end{problem}

\begin{problem}\label{1:3}
Let $A$ be an associative algebra, and $V$ a representation of
$A$. By ${\rm End}_A(V)$ one denotes the algebra of all
homomorphisms of representations $V\to V$. Show that ${\rm End}_A(A)=
A^{op}$, the algebra $A$ with opposite multiplication.
\end{problem}

\begin{problem}\label{1:8}
Prove the following ``Infinite dimensional Schur's lemma''
(due to Dixmier):
Let $A$ be an algebra over $\Bbb C$ and $V$ be an irreducible
representation of $A$ with at most countable basis. Then any homomorphism
of representations $\phi: V\to V$ is a scalar operator.

Hint. By the usual Schur's lemma, the algebra $D:={\rm End}_A(V)$ is
an algebra with division. Show that $D$ is at most countably
dimensional. Suppose $\phi$ is not a scalar,
and consider the subfield $\Bbb C(\phi)\subset D$. Show that
$\Bbb C(\phi)$ is a transcendental extension of $\Bbb C$.
Derive from this that $\Bbb C(\phi)$ is uncountably dimensional
and obtain a contradiction.
\end{problem}

\subsection{Ideals}
A \emph{left ideal} of an algebra $A$ is a subspace $I\subseteq A$ such that $aI
\subseteq I$ for all $a\in A$.
Similarly, a \emph{right ideal} of an algebra $A$ is a subspace
$I\subseteq A$ such that $Ia \subseteq I$ for all $a\in A$.
A \emph{two-sided ideal} is a subspace that is both a left and a
right ideal.

Left ideals are the same as subrepresentations of the regular representation $A$.
Right ideals are the same as subrepresentations of the regular representation of
the opposite algebra $A^\text{op}$.

Below are some examples of ideals:
\begin{itemize}
\item
If $A$ is any algebra, $0$ and $A$ are two-sided ideals.
An algebra $A$ is called \emph{simple} if $0$ and $A$ are its only two-sided
ideals.
\item
If $\phi : A \to B$ is a homomorphism of algebras, then $\ker \phi$ is a two-sided
ideal of $A$.
\item
If $S$ is any subset of an algebra $A$, then the two-sided ideal
\emph{generated} by $S$ is denoted $\langle S\rangle$ and is the span of
elements of the form $asb$, where $a,b \in A$ and $s \in S$. Similarly we can
define $\langle S\rangle_\ell = \text{span}\{as\}$ and $\langle S\rangle_r =
\text{span}\{sb\}$, the left, respectively right, ideal generated
by $S$.
\end{itemize}

\subsection{Quotients}

Let $A$ be an algebra and $I$ a two-sided ideal in $A$.
Then $A/I$ is the set of (additive) cosets of $I$.
Let $\pi : A \to A/I$ be the quotient map.
We can define multiplication in $A/I$ by $ \pi(a)\cdot \pi(b) := \pi(ab). $
This is well defined because
if $\pi(a) = \pi(a')$ then
\[
\pi(a'b) = \pi(ab + (a'-a)b) = \pi(ab)+\pi((a'-a)b) = \pi(ab)
\]
because $(a'-a)b \in Ib \subseteq I = \ker\pi$, as $I$ is a right ideal;
similarly, if $\pi(b) = \pi(b')$ then
\[
\pi(ab') = \pi(ab + a(b'-b)) = \pi(ab)+\pi(a(b'-b)) = \pi(ab)
\]
because $a(b'-b) \in aI \subseteq I = \ker\pi$, as $I$ is also a left ideal.
Thus, $A/I$ is an algebra.

Similarly, if $V$ is a representation of $A$, and $W\subset V$ is a
subrepresentation, then $V/W$ is also a representation.
Indeed, let $\pi : V \to V/W$ be the quotient map, and
set $\rho_{V/W}(a) \pi(x) := \pi (\rho_V(a) x)$.

Above we noted that left ideals of $A$ are subrepresentations of
the regular representation of $A$, and vice versa. Thus, if $I$ is a left ideal in $A$, then
$A/I$ is a representation of $A$.

\begin{problem}\label{1:4}
Let $A=k[x_1,...,x_n]$ and $I\ne A$ be any ideal in $A$ containing
all homogeneous polynomials of degree $\ge N$. Show that
$A/I$ is an indecomposable representation of $A$.
\end{problem}

\begin{problem}\label{1:9}
Let $V\ne 0$ be a representation of $A$. We say that a vector
$v\in V$ is cyclic if it generates $V$, i.e., $Av=V$.
A representation admitting a cyclic vector is said to be cyclic.
Show that

(a) $V$ is irreducible if and only if all nonzero vectors of $V$
are cyclic.

(b) $V$ is cyclic if and only if it is isomorphic to $A/I$, where
$I$ is a left ideal in $A$.

(c) Give an example of an indecomposable representation
which is not cyclic.

Hint. Let $A=\Bbb C[x,y]/I_2$, where $I_2$ is the ideal
spanned by homogeneous polynomials of degree $\ge 2$
(so $A$ has a basis $1,x,y$). Let $V=A^*$ be the space of linear
functionals on $A$, with the action of $A$ given by
$(\rho(a)f)(b)=f(ba)$. Show that $V$ provides such an
example.
\end{problem}

\subsection{Algebras defined by generators and relations}

If $f_1, \dots, f_m$ are elements of the free algebra $k\langle x_1, \dots,
x_n\rangle$, we say that the algebra \linebreak
$A := k\langle x_1, \dots, x_n\rangle /
\langle\{f_1, \dots, f_m\}\rangle$ is \emph{generated by $x_1, \dots, x_n$ with
defining relations $f_1 = 0, \ \dots, \ f_m = 0$}.

\subsection{Examples of algebras}

\begin{enumerate}

\item
The Weyl algebra, $k\langle x, y\rangle/\langle yx-xy-1\rangle$.

\item
The $q$-Weyl algebra,
generated by $x, x^{-1}, y, y^{-1}$ with defining relations
$ yx = qx y$ and $ xx^{-1} = x^{-1} x = y y^{-1} = y^{-1} y = 1 $.
\end{enumerate}

\vskip .05in

{\bf Proposition.} (i) A basis for the Weyl algebra $A$ is $\{x^i
y^j, i,j\ge 0\}$.

(ii) A basis for the q-Weyl algebra $A_q$ is
$\{x^i y^j, i,j\in \Bbb Z\}$.
\vskip .05in

{\it Proof.} (i) First let us show that the elements $x^iy^j$ are a
spanning set for $A$. To do this, note that any word in $x,y$ can be
ordered to have all the $x$ on the left of the $y$, at the cost
of interchanging some $x$ and $y$. Since $yx-xy=1$, this will
lead to error terms, but these terms will be sums of monomials
that have a smaller number of letters $x,y$ than the original
word. Therefore, continuing this process, we can order everything
and represent any word as a linear combination of $x^iy^j$.

The proof that $x^iy^j$ are linearly independent is based on
representation theory. Namely, let $a$ be a variable,
and $E=t^ak[a][t,t^{-1}]$ (here $t^a$ is just a formal symbol,
so really $E=k[a][t,t^{-1}]$). Then $E$
is a representation of $A$
with action given by $xf = tf$ and $yf = \frac{df}{dt}$
(where $\frac{d(t^{a+n})}{dt}:=(a+n)t^{a+n-1}$). Suppose now that
we have a nontrivial linear relation $\sum c_{ij}x^iy^j=0$.
Then the operator
$$
L=\sum c_{ij}t^i\left(\frac{d}{dt}\right)^j
$$
acts by zero in $E$. Let us write $L$ as
$$
L=\sum_{j=0}^rQ_j(t)\left(\frac{d}{dt}\right)^j,
$$
where $Q_r\ne 0$. Then we have
$$
Lt^a=\sum_{j=0}^rQ_j(t)a(a-1)...(a-j+1)t^{a-j}.
$$
This must be zero, so we have
$\sum_{j=0}^rQ_j(t)a(a-1)...(a-j+1)t^{-j}=0$ in
$k[a][t,t^{-1}]$. Taking the leading term in $a$, we get
$Q_r(t)=0$, a contradiction.

(ii) Any word in $x,y,x^{-1},y^{-1}$ can be ordered at the cost
of multiplying it by a power of $q$. This easily implies both the
spanning property and the linear independence.
\vskip .05in

\vskip .05in

{\bf Remark.} The proof of (i) shows that the Weyl algebra $A$
can be viewed as the algebra of polynomial differential
operators in one variable $t$.

\vskip .05in

The proof of (i) also brings up the notion of a faithful
representation.

{\bf Definition.}
A representation $\rho : A \to \text{End } V$ is \emph{faithful} if $\rho$ is
injective.

For example, $k[t]$ is a faithful representation of the Weyl
algebra, if $k$ has characteristic zero (check it!), but not in
characteristic $p$, where $(d/dt)^pQ=0$ for any polynomial $Q$.
However, the representation $E=t^ak[a][t,t^{-1}]$, as we've seen,
is faithful in any characteristic.

\begin{problem}\label{1:5}
Let $A$ be the Weyl algebra, generated
by two elements $x,y$ with the relation
$$
yx-xy-1=0.
$$

(a) If ${\rm char} k=0$, what are the finite dimensional
representations of $A$? What are the two-sided ideals in $A$?

Hint. For the first question, use the fact that for
two square matrices $B,C$, $\Tr(BC)=\Tr(CB)$. For the second
question, show that any nonzero two-sided ideal in $A$ contains a
nonzero polynomial in $x$, and use this to characterize this ideal.

Suppose for the rest of the problem that ${\rm char} k=p$.

(b) What is
the center of $A$?

Hint. Show that $x^p$ and $y^p$ are central elements.

(c) Find all irreducible finite dimensional representations
of $A$.

Hint. Let $V$ be an irreducible finite dimensional representation of
$A$, and $v$ be an eigenvector of $y$ in $V$. Show that
$\lbrace{v,xv,x^2v,...,x^{p-1}v\rbrace}$ is a basis of $V$.
\end{problem}

\begin{problem}\label{1:6}
Let $q$ be a nonzero complex number, and $A$ be the
$q$-Weyl algebra over $\Bbb C$
generated by $x^{\pm 1}$ and $y^{\pm 1}$ with defining
relations $xx^{-1}=x^{-1}x=1,yy^{-1}=y^{-1}y=1$, and
$xy=qyx$.

(a) What is the center of $A$ for different $q$?
If $q$ is not a root of unity, what are the two-sided ideals in
$A$?

(b) For which $q$ does this algebra have finite dimensional
representations?

Hint. Use determinants.

(c) Find all finite dimensional irreducible representations
of $A$ for such $q$.

Hint. This is similar to part (c) of the previous problem.
\end{problem}

\subsection{Quivers}

\begin{definition}
A {\bf quiver} $Q$ is a directed graph, possibly with self-loops
and/or multiple edges between two vertices.
\end{definition}

\begin{example}
\[
\xymatrix@+=1cm{\bullet \ar[r] & \bullet & \bullet \ar[l] \\ & \bullet \ar[u]}
\]
\end{example}

We denote the set of vertices of the quiver $Q$ as $I$, and the set of edges as $E$. For an edge $h \in E$, let $h'$, $h''$ denote
the source and target of $h$, respectively:
\[
\xymatrix@+=1cm{\bullet  \save[]+<0cm,-0.3cm>*{h'}\restore \ar[r]_{h}
& \bullet \save[]+<0cm,-0.3cm>*{h''}\restore}
\]

\begin{definition} \label{quiverRep}
A representation of a quiver $Q$ is an assignment
to each vertex $i \in I$ of a vector space
$V_i$ and to each edge $h \in E$ of a
linear map $x_h : V_{h'} \longrightarrow V_{h''}$.
\end{definition}

It turns out that the theory of representations of quivers is
a part of the theory of representations of algebras in the sense
that for each quiver $Q$, there exists a certain algebra $P_Q$,
called the path algebra of $Q$, such that a representation of the
quiver $Q$ is ``the same'' as a representation of the algebra
$P_Q$. We shall first define the path algebra of a quiver and
then justify our claim that representations of these two objects
are ``the same''.

\begin{definition}
The {\bf path algebra} $P_Q$ of a quiver $Q$ is the algebra
whose basis is formed by oriented paths in $Q$, including the
trivial paths $p_i$, $i\in I$, corresponding to the vertices of
$Q$, and multiplication is concatenation of paths: $ab$
is the path obtained by first tracing $b$ and then $a$.
If two paths cannot be concatenated, the product is defined to be
zero.
\end{definition}

\begin{remark}
It is easy to see that for a finite quiver $\sum\limits_{i \in I} p_i=1$,
so $P_Q$ is an algebra with unit.
\end{remark}

\begin{problem}\label{1:6.5}
Show that the algebra $P_Q$ is
generated by $p_i$ for $i \in I$ and $a_h$ for $h \in E$
with the defining relations:
\begin{enumerate}
\item $p_i^2 = p_i$, $p_ip_j=0$ for $i \neq j$
\item $a_hp_{h'}=a_h$, $a_hp_j=0$ for $j \neq h'$
\item $p_{h''}a_h=a_h$, $p_ia_h=0$ for $i \neq h''$
\end{enumerate}
\end{problem}

We now justify our statement that a representation of a quiver is
the same thing as a representation of the path algebra of a quiver.

Let $\mathbf{V}$ be a representation of the path algebra
$P_Q$. From this representation, we can construct a
representation of $Q$ as follows: let $V_i = p_i \mathbf{V}$, and
for any edge $h$, let
$x_h = a_h |_{p_{h'}\mathbf{V}} : p_{h'}\mathbf{V} \longrightarrow
p_{h''}\mathbf{V}$ be the operator corresponding to the one-edge path
$h$.

Similarly, let $\left(V_i, x_h\right)$
be a representation of a quiver $Q$.
From this representation, we can construct a representation
of the path algebra $P_Q$:
let $\mathbf{V}=\bigoplus_i V_i$,
let $p_i : \mathbf{V} \rightarrow V_i \rightarrow \mathbf{V}$ be the
projection onto $V_i$, and for any path $p=h_1...h_m$
let $a_p=x_{h_1}...x_{h_m}: V_{h_m'}\to V_{h_1''}$
be the composition of the operators
corresponding to the edges occurring in $p$
(and the action of this operator on the other $V_i$ is zero). 

It is clear that the above assignments $\mathbf{V} \mapsto (p_i \mathbf{V})$ and $(V_i) \mapsto \bigoplus_i V_i$ are inverses
of each other. Thus, we have a bijection between isomorphism
classes of representations of the algebra $P_Q$ and of the quiver
$Q$.

\begin{remark}
In practice, it is generally easier to consider a representation
of a quiver as in Definition \ref{quiverRep}. 
\end{remark}

We lastly define several previous concepts in the context of quivers representations.

\begin{definition}
A subrepresentation of a representation $(V_i, x_h)$ of a quiver
$Q$ is a representation $(W_i, x'_h)$ where $W_i \subseteq V_i$
for all $i \in I$ and where $x_h(W_{h'}) \subseteq W_{h''}$ and
$x'_h = x_h |_{W_{h'}} : W_{h'} \longrightarrow W_{h''}$ for all
$h \in E$.
\end{definition}

\begin{definition}
The direct sum of two representations $(V_i, x_h)$ and $(W_i,
y_h)$ is the representation $(V_i \oplus W_i, x_h \oplus y_h)$.
\end{definition}

As with representations of algebras, a nonzero representation $(V_i)$ of
a quiver $Q$ is said to be irreducible if its only
subrepresentations are $(0)$ and $(V_i)$ itself, and
indecomposable if it is not isomorphic to a direct sum of two
nonzero representations.

\begin{definition} \label{quiverRepMor}
Let $(V_i, x_h)$ and $(W_i, y_h)$ be representations of the
quiver $Q$. A homomorphism $\varphi : (V_i) \longrightarrow
(W_i)$ of quiver representations is a collection of maps
$\varphi_i : V_i \longrightarrow W_i$ such that $y_h \circ
\varphi_{h'} = \varphi_{h''}
\circ x_h$ for all $h \in E$.
\end{definition}

\begin{problem}\label{1:7}
Let $A$ be a $\Bbb Z_+$-graded algebra, i.e., $A=\oplus_{n\ge 0}A[n]$,
and $A[n]\cdot A[m]\subset A[n+m]$. If $A[n]$ is finite
dimensional, it is useful to consider the Hilbert series
$h_A(t)=\sum \dim A[n]t^n$ (the generating function of dimensions of $A[n]$).
Often this series converges to a rational function,
and the answer is written in the form of such function.
For example, if $A=k[x]$ and $deg(x^n)=n$ then
$$
h_A(t)=1+t+t^2+...+t^n+...=\frac{1}{1-t}
$$

Find the Hilbert series of:

(a) $A=k[x_1,...,x_m]$ (where the
grading is by degree of polynomials);

(b) $A=k<x_1,...,x_m>$ (the grading is by length of words);

(c) $A$ is the exterior (=Grassmann) algebra
$\wedge_k [x_1,...,x_m]$, generated over some field $k$ by $x_1,...,x_m$ with
the defining relations $x_ix_j+x_jx_i=0$ and $x_i^2=0$ for all $i,j$
(the grading is by degree).

(d) $A$ is the path algebra $P_Q$ of a quiver $Q$ (the grading is
defined by $\deg(p_i)=0$, $\deg(a_h)=1$).

Hint. The closed answer is written in terms of the adjacency
matrix $M_Q$ of $Q$.
\end{problem}

\subsection{Lie algebras}

Let ${\frak g}$ be a vector space over a field $k$, and let
$[\,,] : {\frak g} \times {\frak g} \longrightarrow {\frak g}$ be
a skew-symmetric bilinear map. (That is, $[a,a]=0$, and hence
$[a,b]=-[b,a]$).

\begin{definition}
$\left({\frak g}, [\,,]\right)$ is a {\bf Lie algebra} if $[\,,]$
satisfies the Jacobi identity
\begin{equation} \label{Jacobi}
\bigl[\left[a,b\right],c\bigr]+\bigl[\left[b,c\right],a\bigr]+\bigl[\left[c,a\right],b\bigr]=0.
\end{equation}
\end{definition}

\begin{example} \label{LieEx}
Some examples of Lie algebras are:
\begin{enumerate}
\item Any space ${\frak g}$ with $[\,,]=0$ (abelian Lie algebra).
\item Any associative algebra $A$ with $[a,b]=ab-ba$ \label{uLieEx}.
\item Any subspace $U$ of an associative algebra $A$ such that
$[a,b] \in U$ for all $a,b \in U$.
\item The space ${\rm Der}(A)$ of derivations of an algebra $A$, i.e.
linear maps $D: A\to A$ which satisfy the Leibniz rule:
$$
D(ab)=D(a)b+aD(b).
$$
\end{enumerate}

\begin{remark} Derivations are important because they are
the ``infinitesimal version'' of automorphisms 
(i.e., isomorphisms onto itself). For example,
assume that $g(t)$ is a differentiable family of automorphisms
of a finite dimensional algebra $A$ over $\Bbb R$ or $\Bbb C$
parametrized by $t\in (-\epsilon,\epsilon)$ such that $g(0)={\rm
Id}$. Then $D:=g'(0): A\to A$ is a derivation (check
it!). Conversely, if $D: A\to A$ is a derivation, then
$e^{tD}$ is a 1-parameter family of automorphisms (give a proof!).

This provides a motivation for the notion of a Lie
algebra. Namely, we see that Lie algebras arise as spaces of infinitesimal
automorphisms (=derivations) of associative algebras. In fact,
they similarly arise as spaces of derivations of any kind of linear
algebraic structures, such as Lie algebras, Hopf algebras, etc.,
and for this reason play a very important role in algebra.
\end{remark}

Here are a few more concrete examples of Lie algebras:

\begin{enumerate}
\item $\mathbb{R}^3$ with $[u,v] = u \times v$,
the cross-product of $u$ and $v$.

\item ${\mathfrak{sl}}(n)$, the set of $n \times n$ matrices with trace $0$. \\
  For example, ${\mathfrak{sl}}(2)$ has the basis
  \begin{align*}
    e&=\begin{pmatrix} 0 & 1 \\ 0 & 0  \end{pmatrix}&
    f&=\begin{pmatrix} 0 & 0 \\ 1 & 0  \end{pmatrix}&
    h&=\begin{pmatrix} 1 & 0 \\ 0 & -1 \end{pmatrix}
  \end{align*}
  with relations 
$$
[h,e]=2e,\ [h,f]=-2f,\ [e,f]=h.
$$
\item The Heisenberg Lie algebra ${\mathcal H}$ of
matrices \label{HeisenbergLie}
  $\left(\begin{smallmatrix}
    0 & * & * \\
    0 & 0 & * \\
    0 & 0 & 0
  \end{smallmatrix}\right)$ \\
  It has the basis
  \begin{align*}
    x&=\begin{pmatrix}0&0&0\\0&0&1\\0&0&0\end{pmatrix}&
    y&=\begin{pmatrix}0&1&0\\0&0&0\\0&0&0\end{pmatrix}&
    c&=\begin{pmatrix}0&0&1\\0&0&0\\0&0&0\end{pmatrix}
  \end{align*}
  with relations $[y,x]=c$ and $[y,c]=[x,c]=0$.
\item The algebra ${\rm aff}(1)$ of matrices
  $\left(\begin{smallmatrix}
  * & * \\
  0 & 0
  \end{smallmatrix}\right)$ \\
  Its basis consists of
$X=\left(\begin{smallmatrix} 1 & 0 \\ 0 & 0 \end{smallmatrix}\right)$ and
  $Y=\left(\begin{smallmatrix} 0 & 1 \\ 0 & 0
\end{smallmatrix}\right)$, with $[X,Y]=Y$.
\item ${\mathfrak{so}}(n)$, the space of skew-symmetric $n \times n$
matrices, with $[a,b]=ab-ba$.
\end{enumerate}

{\bf Exercise.} Show that Example 1 is a special case of Example
5 (for $n=3$).

\end{example}

\begin{definition}
Let ${\frak g}_1, {\frak g}_2$ be Lie algebras. A homomorphism
$\varphi : {\frak g}_1 \longrightarrow {\frak g}_2$ of Lie
algebras is a linear map such that
$\varphi([a,b])=[\varphi(a),\varphi(b)]$.
\end{definition}

\begin{definition}
A representation of a Lie algebra ${\frak g}$ is a vector space
$V$ with a homomorphism of Lie algebras
$\rho : {\frak g} \longrightarrow \End V$.
\end{definition}

\begin{example} Some examples of representations of Lie algebras are:
\begin{enumerate}
\item $V=0$.
\item Any vector space $V$ with $\rho=0$ (the trivial representation).
\item The adjoint representation
$V={\frak g}$ with $\rho(a)(b):=[a,b].$ That this is a representation follows from Equation
  (\ref{Jacobi}). Thus, the meaning of the Jacobi identity is 
that it is equivalent to the existence of the adjoint representation.
\end{enumerate}
\end{example}

It turns out that a representation of a Lie algebra ${\frak g}$
is the same thing as a representation of a certain associative algebra
$\mathcal{U}({\frak g})$. Thus, as with quivers, we can view the
theory of representations of Lie algebras as a part of the theory
of representations of associative algebras.

\begin{definition}
Let ${\frak g}$ be a Lie algebra with basis $x_i$ and $[\,,]$
defined by $[x_i,x_j]=\sum_k c_{ij}^k x_k$. The {\bf universal
enveloping algebra} $\mathcal{U}({\frak g})$ is the associative
algebra generated by the $x_i$'s with the defining relations
$x_ix_j-x_jx_i=\sum_k c_{ij}^k x_k$.
\end{definition}

{\bf Remark.} This is not a very good definition since it depends
on the choice of a basis. Later we will give an equivalent
definition which will be basis-independent.

{\bf Exercise.} Explain why a representation of a Lie algebra 
is the same thing as a representation of its universal enveloping algebra. 

\begin{example}
The associative algebra $\mathcal{U}({\mathfrak{sl}}(2))$ is the algebra
generated by $e$, $f$, $h$ with relations
\begin{align*}
he-eh &= 2e &
hf-fh &=-2f &
ef-fe &= h.
\end{align*}
\end{example}
\begin{example}
The algebra $\mathcal{U}({\mathcal H})$, where ${\mathcal H}$
is the Heisenberg Lie algebra, is
the algebra generated by $x$, $y$, $c$ with the relations
\begin{align*}
yx-xy &= c &
yc-cy &= 0 &
xc-cx &= 0.
\end{align*}
Note that the Weyl algebra is the quotient of $\mathcal{U}({\mathcal H})$
by the relation $c=1$.
\end{example}

\subsection{Tensor products}\label{tenprod}

In this subsection we recall the notion of tensor product of
vector spaces, which will be extensively used below.

\begin{definition}\label{tenpro} The tensor product $V\otimes W$ of vector
spaces $V$ and $W$ over a field $k$ is the quotient of the
space $V*W$ whose basis is given by formal symbols $v\otimes w$,
$v\in V$, $w\in W$, by the subspace spanned by the
elements
$$
(v_1+v_2)\otimes w-v_1\otimes w-v_2\otimes w,
v\otimes (w_1+w_2)-v\otimes w_1-v\otimes w_2,
av\otimes w-a(v\otimes w),
v\otimes aw-a(v\otimes w),
$$
where $v\in V,w\in W,a\in k$.
\end{definition}

{\bf Exercise.} Show that $V\otimes W$ can be equivalently
defined as the quotient of the free abelian group $V\bullet W$
generated by $v\otimes w$, $v\in V,w\in W$ by the subgroup
generated by
$$
(v_1+v_2)\otimes w-v_1\otimes w-v_2\otimes w,
v\otimes (w_1+w_2)-v\otimes w_1-v\otimes w_2,
av\otimes w-v\otimes aw,
$$
where $v\in V,w\in W,a\in k$.

The elements $v\otimes w\in V\otimes W$, for $v\in V,w\in W$ are called pure tensors.
Note that in general, there are elements of $V\otimes W$ which are not pure tensors.

\vskip .1in

This allows one to define the tensor product of any number of
vector spaces, $V_1\otimes...\otimes V_n$. Note that
this tensor product is associative, in the sense that
$(V_1\otimes V_2)\otimes V_3$ can be naturally identified with
$V_1\otimes (V_2\otimes V_3)$.

In particular, people often consider tensor products of the form
$V^{\otimes n}=V\otimes...\otimes V$ ($n$ times) for a given
vector space $V$, and, more generally,
$E:=V^{\otimes n}\otimes (V^*)^{\otimes m}$.
This space is called the space of tensors of type $(m,n)$ on
$V$. For instance, tensors of type $(0,1)$ are vectors, of type
$(1,0)$ - linear functionals (covectors), of type $(1,1)$ - linear
operators, of type $(2,0)$ - bilinear forms, of type $(2,1)$ -
algebra structures, etc.

If $V$ is finite dimensional with basis $e_i$, $i=1,...,N$, and
$e^i$ is the dual basis of $V^*$, then a basis of $E$ is
the set of vectors
$$
e_{i_1}\otimes...\otimes e_{i_n}\otimes
e^{j_1}\otimes...\otimes e^{j_m},
$$
and a typical element of $E$ is
$$
\sum_{i_1,...,i_n,j_1,...,j_m=1}^NT^{i_1...i_n}_{j_1...j_m}
e_{i_1}\otimes...\otimes e_{i_n}\otimes
e^{j_1}\otimes...\otimes e^{j_m},
$$
where $T$ is a multidimensional table of numbers.

Physicists define a tensor as a collection of such
multidimensional tables $T_B$ attached to every basis $B$ in $V$,
which change according to a certain rule when the basis $B$ is
changed. Here it is important to distinguish upper and lower
indices, since lower indices of $T$ correspond to $V$ and upper
ones to $V^*$. The physicists don't write the sum sign, but remember
that one should sum over indices that repeat twice - once
as an upper index and once as lower. This convention is called
the {\it Einstein summation}, and it also stipulates
that if an index appears once, then there is no summation over
it, while no index is supposed to appear more than once as an
upper index or more than once as a lower index.

One can also define the tensor product of linear maps.
Namely, if $A: V\to V'$ and $B: W\to W'$ are linear maps, then
one can define the linear map $A\otimes B: V\otimes W\to
V'\otimes W'$ given by the formula $(A\otimes B)(v\otimes
w)=Av\otimes Bw$ (check that this is well defined!)

The most important properties of tensor products are summarized
in the following problem.

\begin{problem}\label{0:1}
(a) Let $U$ be any $k$-vector space.
Construct a natural bijection between
bilinear maps $V\times W\to U$ and linear maps
$V\otimes W\to U$.

(b) Show that if $\lbrace{v_i\rbrace}$ is a basis of $V$ and
$\lbrace{w_j\rbrace}$ is a basis of
$W$ then $\lbrace{v_i\otimes w_j\rbrace}$ is a basis of $V\otimes W$.

(c) Construct a natural isomorphism
$V^*\otimes W\to {\rm Hom}(V,W)$ in the case when $V$ is finite
dimensional (``natural'' means that the isomorphism is defined
without choosing bases).

(d) Let $V$ be a vector space over a field $k$.
Let $S^nV$ be the quotient of $V^{\otimes
n}$ ($n$-fold tensor product of $V$) by the subspace
spanned by the tensors $T-s(T)$ where
$T\in V^{\otimes n}$, and $s$ is some transposition.
Also let $\wedge^nV$ be the
quotient of $V^{\otimes
n}$ by the subspace spanned
by the tensors $T$ such that $s(T)=T$ for some transposition $s$.
These spaces are called the n-th symmetric, respectively
exterior, power of $V$.
If $\lbrace{v_i\rbrace}$ is a basis of $V$, can you
construct a basis of $S^nV,\wedge^nV$? If ${\rm dim}V=m$, what
are their dimensions?

(e) If $k$ has characteristic zero, find a natural identification
of $S^nV$ with the space of $T\in V^{\otimes n}$ such that $T=sT$
for all transpositions $s$, and of
$\wedge^nV$ with the space of $T\in V^{\otimes n}$ such that $T=-sT$
for all transpositions $s$.

(f) Let $A: V\to W$ be a linear operator. Then we have an operator
$A^{\otimes n}: V^{\otimes n}\to W^{\otimes n}$, and its
symmetric and exterior powers
$S^nA: S^nV\to S^nW$, $\wedge^nA: \wedge^nV\to
\wedge^nW$ which are defined in an obvious way.
Suppose $V=W$ and has dimension $N$, and assume that
the eigenvalues of $A$ are $\lambda_1,...,\lambda_N$.
Find $Tr(S^nA),Tr(\wedge^nA)$.

(g) Show that
$\wedge^N A=det(A){\rm Id}$, and use this equality to
give a one-line proof of the
fact that $\det(AB)=\det(A)\det(B)$.
\end{problem}

{\bf Remark.} Note that a similar definition to the above can be used
to define the tensor product $V\otimes _A W$, where $A$ is any
ring, $V$ is a right $A$-module, and $W$ is a left $A$-module.
Namely, $V\otimes_A W$ is the abelian group which is the quotient
of the group $V\bullet W$ freely generated by formal symbols $v\otimes
w$, $v\in V$, $w\in W$, modulo the relations
$$
(v_1+v_2)\otimes w-v_1\otimes w-v_2\otimes w,
v\otimes (w_1+w_2)-v\otimes w_1-v\otimes w_2,
va\otimes w-v\otimes aw, a\in A.
$$

\vskip .05in

{\bf Exercise.} 
Throughout this exercise, we let $k$ be an arbitrary field (not
necessarily of characteristic zero, and not necessarily
algebraically closed).

If $A$ and $B$ are two $k$-algebras, then an
\textbf{$\left(A,B\right)$-bimodule} will mean a $k$-vector space
$V$ with both a left $A$-module structure and a right $B$-module
structure which satisfy $\left(av\right)b=a\left(vb\right)$ for
any $v\in V$, $a\in A$ and $b\in B$. Note that both the notions
of "left $A$-module" and "right $A$-module" are particular cases
of the notion of bimodules; namely, a left $A$-module is the same
as an $\left(A,k\right)$-bimodule, and a right $A$-module is the
same as a $\left(k,A\right)$-bimodule.

Let $B$ be a $k$-algebra, $W$ a left $B$-module and $V$ a right
$B$-module. We denote by $V\otimes_B W$ the $k$-vector space
$\left(V\otimes_k W\right)\slash\left<vb\otimes w-v\otimes bw\mid
v\in V,\ w\in W,\ b\in B\right>$. We denote the projection of a
pure tensor $v\otimes w$ (with $v\in V$ and $w\in W$) onto the
space $V\otimes_B W$ by $v\otimes_B w$. (Note that this tensor
product $V\otimes_B W$ is the one defined in the Remark after
Problem\ref{0:1}.)

If, additionally, $A$ is another $k$-algebra, and if the right
$B$-module structure on $V$ is part of an
$\left(A,B\right)$-bimodule structure, then $V\otimes_B W$
becomes a left $A$-module by $a\left(v\otimes_B
w\right)=av\otimes_B w$ for any $a\in A$, $v\in V$ and $w\in W$.

Similarly, if $C$ is another $k$-algebra, and if the left
$B$-module structure on $W$ is part of a
$\left(B,C\right)$-bimodule structure, then $V\otimes_B W$
becomes a right $C$-module by $\left(v\otimes_B
w\right)c=v\otimes_B wc$ for any $c\in C$, $v\in V$ and $w\in W$.

If $V$ is an $\left(A,B\right)$-bimodule and $W$ is a
$\left(B,C\right)$-bimodule, then these two structures on
$V\otimes_B W$ can be combined into one
$\left(A,C\right)$-bimodule structure on $V\otimes_B W$.

(a) Let $A$, $B$, $C$, $D$ be four algebras. Let $V$ be an
$\left(A,B\right)$-bimodule, $W$ be a
$\left(B,C\right)$-bimodule, and $X$ a
$\left(C,D\right)$-bimodule. Prove that $\left(V\otimes_B
W\right)\otimes_C X\cong V\otimes_B\left(W\otimes_C X\right)$ as
$\left(A,D\right)$-bimodules. The isomorphism (from left to
right) is given by $\left(v\otimes_B w\right)\otimes_C x\mapsto
v\otimes_B\left(w\otimes_C x\right)$ for all $v\in V$, $w\in W$
and $x\in X$.

(b) If $A$, $B$, $C$ are three algebras, and if $V$ is an
$\left(A,B\right)$-bimodule and $W$ an
$\left(A,C\right)$-bimodule, then the vector space
$\mathrm{Hom}_A\left(V,W\right)$ (the space of all left
$A$-linear homomorphisms from $V$ to $W$) canonically becomes a
$\left(B,C\right)$-bimodule by setting
$\left(bf\right)\left(v\right)=f\left(vb\right)$ for all $b\in B$, $f\in\mathrm{Hom}_A\left(V,W\right)$ and $v\in V$
and
$\left(fc\right)\left(v\right)=f\left(v\right)c$ for all $c\in C$, $f\in\mathrm{Hom}_A\left(V,W\right)$ and $v\in V$.

Let $A$, $B$, $C$, $D$ be four algebras. Let $V$ be a
$\left(B,A\right)$-bimodule, $W$ be a
$\left(C,B\right)$-bimodule, and $X$ a
$\left(C,D\right)$-bimodule. Prove that
$\mathrm{Hom}_B\left(V,\mathrm{Hom}_C\left(W,X\right)\right)
\cong \mathrm{Hom}_C\left(W\otimes_B V,X\right)$ as
$\left(A,D\right)$-bimodules. The isomorphism (from left to
right) is given by $f\mapsto \left(w\otimes_B v\mapsto
f\left(v\right)w\right)$ for all $v\in V$, $w\in W$
and $f\in \mathrm{Hom}_B\left(V,\mathrm{Hom}_C\left(W,X\right)\right)$.

\subsection{The tensor algebra}

The notion of tensor product allows us to give more conceptual
(i.e., coordinate free) definitions of the free algebra, polynomial algebra, exterior
algebra, and universal enveloping algebra of a Lie algebra.

Namely, given a vector space $V$, define its {\it tensor algebra}
$TV$ over a field $k$ to be $TV=\oplus_{n\ge 0}V^{\otimes n}$,
with multiplication defined by $a\cdot b:=a\otimes b$, $a\in
V^{\otimes n}$, $b\in V^{\otimes m}$. Observe that
a choice of a basis $x_1,...,x_N$ in $V$ defines an isomorphism
of $TV$ with the free algebra $k<x_1,...,x_n>$.

Also, one can make the following definition.

\begin{definition}
(i) The symmetric algebra $SV$ of $V$ is the quotient of $TV$ by
the ideal generated by $v\otimes w-w\otimes v$, $v,w\in V$.

(ii) The exterior algebra $\wedge V$ of $V$ is the quotient of $TV$ by
the ideal generated by $v\otimes v$, $v\in V$.

(iii) If $V$ is a Lie algebra,
the universal enveloping algebra ${\mathfrak{U}}(V)$ of $V$ is the quotient of $TV$ by
the ideal generated by $v\otimes w-w\otimes v-[v,w]$, $v,w\in V$.
\end{definition}

It is easy to see that a choice of a basis $x_1,...,x_N$ in $V$ identifies
$SV$ with the polynomial algebra $k[x_1,...,x_N]$, $\wedge V$
with the exterior algebra $\wedge_k(x_1,...,x_N)$,
and the universal enveloping algebra ${\mathfrak{U}}(V)$ with one defined
previously.

Also, it is easy to see that we have decompositions
$SV=\oplus_{n\ge 0}S^nV$, $\wedge V=\oplus_{n\ge 0}\wedge^n V$.

\subsection{Hilbert's third problem}

\begin{problem}\label{0:3} It is known that if $A$ and $B$ are two
polygons of the same area then $A$ can be cut by finitely many
straight cuts into pieces from which one can make B. David Hilbert
asked in 1900 whether it is true for polyhedra in 3 dimensions.
In particular, is it true for a cube and a regular tetrahedron
of the same volume?

The answer is ``no'', as was found by Dehn in 1901.
The proof is very beautiful. Namely, to any
polyhedron $A$ let us attach
its ``Dehn invariant'' $D(A)$ in $V={\Bbb R}\otimes
({\Bbb R/\Bbb Q})$ (the tensor product of ${\Bbb Q}$-vector
spaces).
Namely,
$$
D(A)=\sum_a l(a)\otimes \frac{\beta(a)}{\pi},
$$ where $a$ runs
over edges of $A$, and $l(a),\beta(a)$ are the length of $a$ and the
angle at $a$.

(a) Show that if you cut $A$ into $B$ and $C$ by a straight cut,
then $D(A)=D(B)+D(C)$.

(b) Show that $\alpha={\rm arccos}(1/3)/\pi$ is not a rational number.

Hint. Assume that $\alpha=2m/n$, for integers $m,n$.
Deduce that roots of the equation $x+x^{-1}=2/3$
are roots of unity of degree n. Conclude that
$x^k+x^{-k}$ has denominator $3^k$ and get a contradiction.

(c) Using (a) and (b), show that the answer to Hilbert's question
is negative. (Compute the Dehn invariant of the regular
tetrahedron and the cube).
\end{problem}

\subsection{Tensor products and duals of representations of Lie algebras}

\begin{definition}
The tensor product of two representations
$V,W$ of a Lie algebra ${\mathfrak g}$
is the space $V\otimes W$ with $\rho_{V\otimes W}(x)=
\rho_V(x)\otimes Id+Id\otimes \rho_W(x)$.
\end{definition}

\begin{definition}
The dual representation $V^*$ to a representation $V$ of a Lie
algebra ${\mathfrak g}$ is the dual space $V^*$ to $V$
with $\rho_{V^*}(x)=-\rho_V(x)^*$.
\end{definition}

It is easy to check that these are indeed representations.

\begin{problem}
Let $V,W,U$ be finite dimensional representations
of a Lie algebra $\mathfrak{g}$. Show that
the space $\Hom_{\mathfrak{g}}(V\otimes W,U)$
is isomorphic to $\Hom_{\mathfrak{g}}(V,U\otimes W^*)$.
(Here $\Hom_{\mathfrak{g}}:=\Hom_{{\mathcal{U}}(\mathfrak{g})}$).
\end{problem}

\subsection{Representations of ${\mathfrak{sl}}(2)$}

This subsection is devoted to the representation theory of
${\mathfrak{sl}}(2)$, which is of central importance in many areas of mathematics.
It is useful to study this topic by solving the following
sequence of exercises, which every mathematician should do,
in one form or another.

\begin{problem}\label{4:1}
According to the above, a representation of ${\mathfrak{sl}}(2)$
is just a vector space $V$ with a triple of operators $E,F,H$ such that
$HE-EH=2E, HF-FH=-2F,EF-FE=H$ (the corresponding map $\rho$ is
given by $\rho(e)=E,\rho(f)=F$, $\rho(h)=H$).

Let $V$ be
a finite dimensional representation of ${\mathfrak{sl}}(2)$
(the ground field in this problem is ${\Bbb C}$).

(a) Take eigenvalues of $H$ and pick one with the biggest real part.
Call it $\lambda$. Let $\bar V(\lambda)$ be the generalized eigenspace corresponding
to $\lambda$. Show that $E|_{\bar V(\lambda)}=0$.

(b) Let $W$ be any representation of ${\mathfrak{sl}}(2)$ and
$w\in W$ be a nonzero vector such that $Ew=0$. For any
$k>0$ find a polynomial $P_k(x)$ of degree $k$ such that
$E^kF^kw=P_k(H)w$. (First compute $EF^kw$, then use induction in $k$).

(c) Let $v\in \bar V(\lambda)$ be a generalized eigenvector of $H$ with eigenvalue $\lambda$.
Show that there exists $N>0$ such that $F^Nv=0$.

(d) Show that $H$ is diagonalizable on $\bar V(\lambda)$. (Take $N$
to be such that $F^N=0$ on $\bar V(\lambda)$, and compute $E^NF^Nv$,
$v\in\bar V(\lambda)$, by (b). Use the fact that $P_k(x)$ does not have
multiple roots).

(e) Let $N_v$ be the smallest $N$ satisfying (c). Show that $\lambda=N_v-1$.

(f) Show that for each $N>0$, there exists a unique up to isomorphism
irreducible representation of ${\mathfrak{sl}}(2)$ of dimension $N$. Compute the matrices
$E,F,H$ in this representation using a convenient basis. (For $V$ finite
dimensional irreducible take $\lambda$ as in (a) and $v\in V(\lambda)$ an eigenvector
of $H$. Show that $v,Fv,...,F^\lambda v$ is a basis of $V$, and compute
the matrices of the operators $E,F,H$ in this basis.)

Denote the $\lambda+1$-dimensional irreducible representation from (f)
by $V_\lambda$. Below you will show that any finite dimensional representation
is a direct sum of $V_\lambda$.

(g) Show that the operator $C=EF+FE+H^2/2$ (the so-called Casimir operator)
commutes with $E,F,H$ and equals $\frac{\lambda(\lambda+2)}{2} Id$ on $V_\lambda$.

Now it will be easy to prove the direct sum decomposition.
Namely, assume the contrary, and
let $V$ be a reducible representation of the smallest dimension,
which is not a direct sum of smaller representations.

(h) Show that $C$ has only one eigenvalue on
$V$, namely $\frac{\lambda(\lambda+2)}{2}$
for some nonnegative integer $\lambda$. (use that the generalized eigenspace
decomposition of $C$ must be a decomposition of representations).

(i) Show that $V$ has a subrepresentation
$W=V_\lambda$ such that $V/W=nV_\lambda$
for some $n$ (use (h) and the fact that $V$ is the smallest
which cannot be decomposed).

(j) Deduce from (i) that the eigenspace $V(\lambda)$ of $H$ is $n+1$-dimensional.
If $v_1,...,v_{n+1}$ is its basis, show that $F^jv_i$, $1\le i\le n+1$,
$0\le j\le \lambda$ are linearly independent
and therefore form a basis of $V$ (establish
that if $Fx=0$ and $Hx=\mu x$ then
$Cx=\frac{\mu(\mu-2)}{2}x$ and hence $\mu=-\lambda$).

(k) Define $W_i=\text{span}(v_i,Fv_i,...,F^\lambda v_i)$. Show that $V_i$ are
subrepresentations of $V$ and derive a contradiction with the fact that
$V$ cannot be decomposed.

(l) (Jacobson-Morozov Lemma)
Let $V$ be a finite dimensional complex vector space and $A:V\to V$
a nilpotent operator. Show that there exists a unique, up to an isomorphism,
representation of ${\mathfrak{sl}}(2)$ on $V$ such that $E=A$. (Use the classification of
the representations and the Jordan normal form theorem)

(m) (Clebsch-Gordan decomposition)
Find the decomposition into irreducibles
of the representation $V_\lambda\otimes V_\mu$ of ${\mathfrak{sl}}(2)$.

Hint. For a finite dimensional representation $V$ of ${\mathfrak{sl}}(2)$
 it is useful to introduce the character $\chi_V(x)=Tr(e^{xH})$,
$x\in{\Bbb C}$. Show that $\chi_{V\oplus
W}(x)=\chi_V(x)+\chi_W(x)$ and
$\chi_{V\otimes W}(x)=\chi_V(x)\chi_W(x)$.
Then compute the character of $V_\lambda$ and
of $V_\lambda\otimes V_\mu$
and derive the decomposition. This decomposition is of fundamental
importance in quantum mechanics.

(n) Let $V={\Bbb C}^M\otimes {\Bbb C}^N$, and
$A=J_M(0)\otimes Id_N+Id_M\otimes J_N(0)$, where $J_n(0)$
is the Jordan block of size $n$ with eigenvalue zero
(i.e., $J_n(0)e_i=e_{i-1}$, $i=2,...,n$, and $J_n(0)e_1=0$).
Find the Jordan normal form of $A$ using (l),(m).
\end{problem}

\subsection{Problems on Lie algebras}

\begin{problem}\label{4:2} ({\bf Lie's Theorem})
The {\it commutant} $K({\frak g})$
of a Lie algebra ${\frak g}$ is the linear span of elements
$[x,y]$, $x,y\in {\frak g}$. This is an ideal in ${\frak g}$
(i.e., it is a subrepresentation of the adjoint representation).
A finite dimensional Lie algebra ${\frak g}$
over a field $k$
is said to be solvable if there exists $n$ such that $K^n({\frak
g})=0$. Prove the Lie theorem: if $k=\Bbb C$ and $V$ is a finite dimensional
irreducible representation of a solvable Lie algebra ${\frak g}$
then $V$ is 1-dimensional.

Hint. Prove the result by induction in dimension.
By the induction assumption, $K({\frak g})$ has a common
eigenvector $v$ in $V$, that is there is a linear function
$\chi :K({\frak g})\to \Bbb C$ such that $av=\chi (a)v$ for any
$a\in K({\frak g})$. Show that ${\frak g}$ preserves common
eigenspaces of $K({\frak g})$ (for this you will need to show
that $\chi ([x,a])=0$ for $x\in {\frak g}$ and $a\in K({\frak
g})$. To prove this, consider the smallest
vector subspace $U$ containing $v$ and invariant under $x$. This
subspace is invariant under $K({\frak g})$ and any $a\in K({\frak
g})$ acts with trace ${\rm dim} (U)\chi (a)$ in this subspace. In
particular $0={\rm Tr} ([x,a])={\rm dim} (U) \chi([x,a])$.).
\end{problem}

\begin{problem}\label{4:3}
Classify irreducible finite dimensional representations of the two dimensional
Lie algebra with basis $X,Y$ and commutation relation $[X,Y]=Y$.
Consider the cases of zero and positive characteristic.
Is the Lie theorem true in positive characteristic?
\end{problem}

\begin{problem}\label{4:4} (hard!)
For any element $x$ of a Lie algebra ${\frak g}$ let
$ad(x)$ denote the operator ${\frak g}\to {\frak g}, y\mapsto
[x,y]$. Consider the Lie algebra ${\frak g}_n$ generated by two
elements $x,y$ with the defining relations
$ad(x)^2(y)=ad(y)^{n+1}(x)=0$.

(a) Show that the Lie algebras ${\frak g}_1, {\frak g}_2, {\frak
g}_3$ are finite dimensional and find their dimensions.

(b) (harder!) Show that the Lie algebra ${\frak g}_4$ has infinite dimension.
Construct explicitly a basis of this algebra.
\end{problem}

\newpage \section{General results of representation theory}\label{Cha2}

\subsection{Subrepresentations in semisimple representations}

Let $A$ be an algebra.

\begin{definition} A semisimple
(or completely reducible) representation of $A$ is a direct sum of
irreducible representations.
\end{definition}

{\bf Example.} Let $V$ be an irreducible
representation of $A$ of dimension $n$. Then $Y=\End(V)$, with action of $A$
by left multiplication, is a semisimple representation of $A$,
isomorphic to $nV$ (the direct sum of $n$ copies of $V$).
Indeed, any basis $v_1,...,v_n$ of $V$ gives rise to an isomorphism
of representations $\End(V)\to nV$, given by $x\to
(xv_1,...,xv_n)$.

\vskip .1in

{\bf Remark.}
Note that by Schur's lemma, any semisimple representation $V$ of $A$
is canonically identified with $\oplus_X \Hom_A(X,V)\otimes X$,
where $X$ runs over all irreducible representations of $A$.
Indeed, we have a natural map $f: \oplus_X \Hom(X,V)\otimes X\to
V$, given by $g\otimes x\to g(x)$, $x\in X$, $g\in \Hom(X,V)$,
and it is easy to verify that this map is an isomorphism.

We'll see now how Schur's lemma allows us to classify subrepresentations
in finite dimensional semisimple representations.

\begin{proposition}\label{submo}
Let $V_i,1\le i\le m$ be irreducible finite dimensional
pairwise nonisomorphic representations of $A$, and
$W$ be a subrepresentation of $V=\oplus_{i=1}^m n_iV_i$.
Then $W$ is isomorphic to
$\oplus_{i=1}^m r_iV_i$, $r_i\le n_i$, and the inclusion $\phi:
W\to V$ is a direct sum of inclusions
$\phi_i: r_iV_i\to n_iV_i$ given by multiplication
of a row vector of elements of $V_i$
(of length $r_i$) by a certain $r_i$-by-$n_i$ matrix $X_i$
with linearly independent rows:
$\phi(v_1,...,v_{r_i})=(v_1,...,v_{r_i})X_i$.
\end{proposition}

\begin{proof}
The proof is by induction in $n:=\sum_{i=1}^m n_i$. The base of induction
($n=1$) is clear. To perform the induction step,
let us assume that $W$ is nonzero, and
fix an irreducible subrepresentation $P\subset W$.
Such $P$ exists (Problem \ref{1:1}).
\footnote{Another proof of the existence of $P$, which does not
use the finite dimensionality of $V$, is by induction in $n$.
Namely, if $W$ itself is not irreducible,
let $K$ be the kernel of the projection of $W$ to the first summand
$V_1$. Then $K$ is a subrepresentation of
$(n_1-1)V_1\oplus...\oplus n_mV_m$, which is nonzero
since $W$ is not irreducible,
so $K$ contains an irreducible subrepresentation by the
induction assumption.}
Now, by Schur's lemma, $P$ is isomorphic to $V_i$ for some $i$,
and the inclusion $\phi|_P: P\to V$ factors through $n_iV_i$, and
upon identification of $P$ with $V_i$ is given by the formula
$v\mapsto (vq_1,...,vq_{n_i})$, where $q_l\in k$ are not all zero.

Now note that the group $G_i=GL_{n_i}(k)$
of invertible $n_i$-by-$n_i$ matrices
over $k$ acts on $n_iV_i$ by $(v_1,...,v_{n_i})\to
(v_1,...,v_{n_i})g_i$ (and by the identity on $n_jV_j$, $j\ne
i$), and therefore acts on the set of
subrepresentations of $V$, preserving the property
we need to establish: namely, under the action of $g_i$,
the matrix $X_i$ goes to $X_ig_i$, while $X_j,j\ne i$ don't change.
Take $g_i\in G_i$ such that $(q_1,...,q_{n_i})g_i=(1,0,...,0)$.
Then $Wg_i$ contains the first summand $V_i$ of $n_iV_i$ (namely,
it is $Pg_i$),
hence $Wg_i=V_i\oplus W'$, where $W'\subset
n_1V_1\oplus...\oplus(n_i-1)V_i\oplus...\oplus n_mV_m$
is the kernel of the projection of $Wg_i$ to the first
summand $V_i$ along the other summands.
Thus the required statement follows from the induction assumption.
\end{proof}

\begin{remark} In Proposition \ref{submo}, it is not important
that $k$ is algebraically closed, nor it matters that $V$ is finite
dimensional. If these assumptions are dropped,
the only change needed is that the entries of the matrix
$X_i$ are no longer in $k$ but in $D_i=\End_A(V_i)$, which is, as we
know, a division algebra. The proof of this generalized version
of Proposition \ref{submo} is the same as before (check it!).
\end{remark}

\subsection{The density theorem}

Let $A$ be an algebra over an algebraically closed field $k$.

\begin{corollary}\label{linalg}
Let $V$ be an irreducible finite dimensional representation of $A$,
and $v_1,...,v_n\in V$ be any linearly independent vectors.
Then for any $w_1,...,w_n\in V$ there exists an element $a\in A$
such that $av_i=w_i$.
\end{corollary}

\begin{proof}
Assume the contrary. Then the image of the map $A\to nV$ given
by $a\to (av_1,...,av_n)$ is a proper subrepresentation,
so by Proposition \ref{submo} it
corresponds to an $r$-by-$n$ matrix $X$, $r<n$.
Thus, taking $a=1$, we see that there exist vectors
$u_1,...,u_r\in V$ such that $(u_1,...,u_r)X=(v_1,...,v_n)$.
Let $(q_1,...,q_n)$ be
a nonzero vector such that $X(q_1,...,q_n)^T=0$
(it exists because $r<n$).
Then $\sum q_iv_i=(u_1,...,u_r)X(q_1,...,q_n)^T=0$, i.e.
$\sum q_iv_i=0$
 - a contradiction with the linear independence of $v_i$.
\end{proof}

\begin{theorem}\label{DensityThm} (the Density Theorem).
(i) Let $V$ be an irreducible finite dimensional
representation of $A$. Then
the map $\rho: A\to {\rm End}V$ is surjective.

(ii) Let $V=V_1\oplus...\oplus V_r$, where $V_i$ are irreducible
pairwise nonisomorphic finite dimensional
representations of $A$. Then the map $\oplus_{i=1}^r\rho_i:
A\to \oplus_{i=1}^r\End(V_i)$ is surjective.
\end{theorem}

\begin{proof}
(i) Let $B$ be the image of $A$ in ${\rm End}(V)$. We want to show that $B=\End(V)$.
Let $c\in \End(V)$, $v_1,...,v_n$ be a basis of $V$, and
$w_i=cv_i$. By Corollary \ref{linalg},
there exists $a\in A$ such that $av_i=w_i$. Then $a$ maps to
$c$, so $c\in B$, and we are done.

(ii) Let $B_i$ be the image of $A$ in ${\rm End}(V_i)$,
and $B$ be the image of $A$ in $\oplus_{i=1}^r \End(V_i)$.
Recall that as a representation of $A$,
$\oplus_{i=1}^r \End(V_i)$ is semisimple: it is isomorphic to
$\oplus_{i=1}^r d_iV_i$, where $d_i=\dim V_i$.
Then by Proposition \ref{submo}, $B=\oplus_i B_i$.
On the other hand, (i) implies that $B_i=\End(V_i)$.
Thus (ii) follows.
\end{proof}

\subsection{Representations of direct sums of matrix algebras}

In this section we consider representations of algebras
$A=\bigoplus_i \Mat_{d_i}(k)$ for any field $k$.

\begin{theorem}\label{RepMatrix}
Let $A=\bigoplus_{i=1}^r \Mat_{d_i}(k)$.
Then the irreducible representations
of $A$ are $V_1=k^{d_1},\dots,V_r=k^{d_r}$, and any
finite dimensional representation of $A$ is a direct sum of
copies of $V_1,\dots,V_r$.
\end{theorem}

In order to prove Theorem \ref{RepMatrix},
we shall need the notion of a dual representation.

\begin{definition}(Dual representation)
Let $V$ be a representation of any algebra $A$.
Then the dual representation $V^*$ is the
representation of the opposite algebra $A^{\text{op}}$
(or, equivalently, right $A$-module)
with the action
\begin{equation*}
(f \cdot a)(v): = f(av).
\end{equation*}
\end{definition}

\begin{proof}[Proof of Theorem \ref{RepMatrix}]
First, the given representations are clearly irreducible, as for any
$v\ne 0,w\in V_i$, there exists $a\in A$ such that $av=w$. Next, let
$X$ be an $n$-dimensional representation of $A$.
Then, $X^*$ is an $n$-dimensional representation of $A^{\text{op}}$. But
$\left(\operatorname{Mat}_{d_i}(k)\right)^{\text{op}}
\cong \operatorname{Mat}_{d_i}(k)$ with isomorphism $\varphi(X)=X^T$, as
$(BC)^T=C^TB^T$. Thus, $A \cong A^{\text{op}}$ and $X^*$ may be
viewed as an $n$-dimensional representation of $A$.
Define
\begin{equation*}
\phi : \underbrace{A \oplus \cdots
\oplus A}_{n \mbox{ copies}} \longrightarrow X^*
\end{equation*}
by
\begin{equation*}
\phi{\left(a_1,\dots,a_n\right)}=a_1y_1+\cdots+a_ny_n
\end{equation*}
where $\left\{y_i\right\}$ is a basis of $X^*$.
$\phi$ is clearly surjective, as $k \subset A$. Thus, the dual
map $\phi^* : X \longrightarrow {A^n}^*$ is injective.
But ${A^n}^* \cong A^n$ as representations of $A$ (check it!).
Hence, $\operatorname{Im} \phi^* \cong X$ is a subrepresentation
of $A^n$. Next, $\Mat_{d_i}(k)=d_iV_i$, so
$A=\oplus_{i=1}^rd_iV_i$,
$A^n=\oplus_{i=1}^rnd_iV_i$,
 as a representation of $A$.
Hence by Proposition \ref{submo},
$X=\oplus_{i=1}^r m_iV_i$, as desired.
\end{proof}

\textbf{Exercise.}
The goal of this exercise is to give an 
alternative proof of Theorem \ref{RepMatrix}, not
using any of the previous results of Chapter \ref{Cha2}.

Let
$A_{1}$, $A_{2}$, $...$, $A_{n}$ be $n$ algebras with units $1_{1}$, $1_{2}%
$, $...$, $1_{n}$, respectively. Let $A=A_{1}\oplus A_{2}\oplus...\oplus
A_{n}$. Clearly, $1_i1_j=\delta_{ij}1_i$, and
the unit of $A$ is $1=1_{1}+1_{2}+...+1_{n}$.

For every representation
$V$ of $A$, it is easy to see that $1_{i}V$ is a representation of $A_{i}$ for
every $i\in\left\{  1,2,...,n\right\}  $. Conversely, if $V_{1}$, $V_{2}$,
$...$, $V_{n}$ are representations of $A_{1}$, $A_{2}$, $...$, $A_{n}$,
respectively, then $V_{1}\oplus V_{2}\oplus...\oplus V_{n}$ canonically
becomes a representation of $A$ (with $\left(  a_{1},a_{2},...,a_{n}\right)
\in A$ acting on $V_{1}\oplus V_{2}\oplus...\oplus V_{n}$ as $\left(
v_{1},v_{2},...,v_{n}\right)  \mapsto\left(  a_{1}v_{1},a_{2}v_{2}%
,...,a_{n}v_{n}\right)  $).

\textbf{(a)}  
Show that a representation $V$ of $A$ is irreducible if and only if $1_{i}V$
is an irreducible representation of $A_{i}$ for exactly one $i\in\left\{
1,2,...,n\right\}  $, while $1_{i}V=0$ for all the other $i$. Thus, classify
the irreducible representations of $A$ in terms of those of $A_{1}$, $A_{2}$,
$...$, $A_{n}$.

\textbf{(b)} Let $d\in\mathbb{N}$. Show that the only
irreducible representation of $\operatorname*{Mat}_{d}(k)$ is
$k^{d}$, and every
finite dimensional representation of $\operatorname*{Mat}_{d}(k)$ is a direct
sum of copies of $k^{d}$.

\textit{Hint:} For every $\left(
i,j\right)  \in\left\{  1,2,...,d\right\}  ^{2}$, 
let $E_{ij}\in \operatorname*{Mat}_{d}(k)$ be the matrix 
with $1$ in the $i$th row of the $j$th column and 0's everywhere else.
Let $V$ be a finite dimensional representation of $\operatorname*{Mat}_{d}(k)$.
Show that $V=E_{11}V\oplus E_{22}V\oplus...\oplus E_{dd}V$, and that
$\Phi_{i}:E_{11}V\rightarrow E_{ii}V$, $v\mapsto E_{i1}v$ is an isomorphism
for every $i\in\left\{  1,2,...,d\right\}  $.\ For every $v\in E_{11}V$,
denote $S\left(  v\right)  =\left\langle
E_{11}v,E_{21}v,...,E_{d1}v\right
\rangle $. Prove that $S\left(  v\right)  $ is a subrepresentation of
$V$ isomorphic to $k^{d}$ (as a representation of 
$\operatorname*{Mat}_{d}(k)$), and that 
$v\in S\left(  v\right)  $. 
Conclude that $V=S\left(  v_{1}\right)
\oplus S\left(  v_{2}\right)  \oplus...\oplus S\left(  v_{k}\right)  $, where
$\left\{  v_{1},v_{2},...,v_{k}\right\}  $ is a basis of
$E_{11}V$.

\textbf{(c)} Conclude Theorem \ref{RepMatrix}.

\subsection{Filtrations}

Let $A$ be an algebra.
Let $V$ be a representation of $A$.
A (finite) {\it filtration} of $V$ is a sequence of subrepresentations
$0=V_0\subset V_1\subset...\subset V_n=V$.

\begin{lemma}\label{simp}
Any finite dimensional representation $V$ of an algebra $A$
admits a finite filtration $0=V_0\subset V_1\subset...\subset
V_n=V$ such that the successive quotients $V_i/V_{i-1}$
are irreducible.
\end{lemma}

\begin{proof}
The proof is by induction in $\dim(V)$. The base is clear, and
only the induction step needs to be justified. Pick an
irreducible subrepresentation $V_1\subset V$, and consider the
representation $U=V/V_1$. Then by the induction assumption
$U$ has a filtration $0=U_0\subset
U_1\subset...\subset U_{n-1}=U$ such that $U_i/U_{i-1}$
are irreducible. Define $V_i$ for $i\ge 2$ to be the preimages of
$U_{i-1}$ under the tautological projection $V\to V/V_1=U$.
Then $0=V_0\subset V_1\subset V_2\subset...\subset V_n=V$ is a
filtration of $V$ with the desired property.
\end{proof}

\subsection{Finite dimensional algebras}\label{FiniteDimSec}

\begin{definition}
The {\bf radical} of a finite dimensional algebra $A$ is the set of all elements of $A$ which act by $0$ in all irreducible
representations of $A$. It is denoted $\Rad(A)$.
\end{definition}

\begin{proposition}
$\Rad(A)$ is a two-sided ideal.
\end{proposition}

\begin{proof}
Easy.
\end{proof}

\begin{proposition}\label{nilp} Let $A$ be a finite dimensional
algebra.

(i) Let $I$ be a nilpotent two-sided ideal in $A$, i.e., $I^n=0$ for some $n$.
Then $I\subset \Rad(A)$.

(ii) $\Rad(A)$ is a nilpotent ideal. Thus, $\Rad(A)$ is the
largest nilpotent two-sided ideal in $A$.
\end{proposition}

\begin{proof}
(i) Let $V$ be an irreducible representation of $A$.
Let $v\in V$. Then $Iv\subset V$ is a subrepresentation.
If $Iv\ne 0$ then $Iv=V$ so there is $x\in I$ such that
$xv=v$. Then $x^n\ne 0$, a contradiction. Thus $Iv=0$, so $I$
acts by $0$ in $V$ and hence $I\subset \Rad(A)$.

(ii) Let $0=A_0\subset A_1\subset...\subset A_n=A$ be a
filtration of the regular representation of $A$ by
subrepresentations such that $A_{i+1}/A_i$ are irreducible.
It exists by Lemma \ref{simp}. Let $x\in \Rad(A)$. Then $x$ acts
on $A_{i+1}/A_i$ by zero, so $x$ maps $A_{i+1}$ to $A_i$.
This implies that $\Rad(A)^n=0$, as desired.
\end{proof}

\begin{theorem} \label{cow}
A finite dimensional algebra $A$
has only finitely many
irreducible representations $V_i$ up to isomorphism, these representations
are finite dimensional, and
\begin{equation*}
A/\Rad(A) \cong \bigoplus_i \End V_i.
\end{equation*}
\end{theorem}

\begin{proof}
First, for any irreducible representation $V$ of $A$, and for any nonzero $v \in V$, $Av \subseteq V$ is a finite dimensional
subrepresentation of $V$. (It is finite dimensional as $A$ is finite dimensional.) As $V$ is irreducible and $Av \neq 0$, $V=Av$ and $V$
is finite dimensional.

Next, suppose we have non-isomorphic irreducible representations
$V_1,V_2,\dots,V_r$.
By Theorem \ref{DensityThm}, the homomorphism \begin{equation*}\bigoplus_i
\rho_i : A \longrightarrow \bigoplus_i \End V_i\end{equation*} is
surjective. So $r\le \sum_i \dim \End V_i \le
\dim A$. Thus, $A$ has
only finitely many non-isomorphic irreducible representations
(at most $\dim A$).

Now, let $V_1,V_2,\dots,V_r$ be all non-isomorphic irreducible finite dimensional representations of $A$. By Theorem \ref{DensityThm},
the homomorphism \begin{equation*}\bigoplus_i \rho_i : A \longrightarrow \bigoplus_i \End V_i\end{equation*} is surjective. The
kernel of this map, by definition, is exactly $\Rad(A)$.
\end{proof}

\begin{corollary}
$\sum_i \left(\dim V_i\right)^2 \leq \dim A$,
where the $V_i$'s are the irreducible representations of $A$.
\end{corollary}

\begin{proof}
As $\dim \End V_i = \left(\dim V_i\right)^2$, Theorem \ref{cow} implies that $\dim A - \dim \Rad(A) = \sum_i \dim \End V_i = \sum_i
\left(\dim V_i\right)^2$. As $\dim \Rad(A) \geq 0$, $\sum_i \left(\dim V_i\right)^2 \leq \dim A$.
\end{proof}

\begin{example}
1. Let $A=k[x]/(x^n)$. This algebra has a unique
irreducible representation, which is a 1-dimensional space $k$,
in which $x$ acts by zero. So the radical $\Rad(A)$ is the ideal
$(x)$.

2. Let $A$ be the algebra of upper triangular $n$ by $n$
matrices. It is easy to check that the
irreducible representations of $A$ are $V_i$, $i=1,...,n$,
which are 1-dimensional, and any matrix $x$ acts by $x_{ii}$.
So the radical $\Rad(A)$ is the ideal of strictly upper
triangular matrices (as it is a nilpotent ideal and contains the
radical). A similar result holds for block-triangular
matrices.
\end{example}

\begin{definition}
A finite dimensional algebra $A$ is said to be {\bf semisimple} if $\Rad(A)=0$.
\end{definition}

\begin{proposition}\label{semisi}
For a finite dimensional algebra $A$, the following are equivalent:
\begin{enumerate}
\item $A$ is semisimple. \label{moo1}
\item $\sum_i \left(\dim V_i\right)^2 = \dim A$, where the $V_i$'s are the irreducible representations of $A$. \label{moo2}
\item $A \cong \bigoplus_i \operatorname{Mat}_{d_i}(k)$ for some $d_i$. \label{moo3}
\item Any finite dimensional representation of $A$ is completely reducible (that is, isomorphic to a direct sum of irreducible
representations). \label{moo4}
\item $A$ is a completely reducible representation of $A$. \label{moo5}
\end{enumerate}
\end{proposition}

\begin{proof}
As $\dim A - \dim \Rad(A) = \sum_i \left(\dim V_i\right)^2$, clearly $\dim A = \sum_i \left(\dim V_i\right)^2$ if and only if $\Rad(A) = 0$. Thus,
$(\ref{moo1}) \Leftrightarrow (\ref{moo2})$.

Next, by Theorem \ref{cow}, if $\Rad(A)=0$, then clearly $A \cong \bigoplus_i \operatorname{Mat}_{d_i}(k)$ for $d_i = \dim V_i$. Thus,
$(\ref{moo1}) \Rightarrow (\ref{moo3})$. Conversely, if $A \cong
\bigoplus_i \operatorname{Mat}_{d_i}(k)$, then
by Theorem \ref{RepMatrix},
$\Rad(A)=0$, so $A$ is semisimple.
Thus $(\ref{moo3})\Rightarrow(\ref{moo1})$.

Next, $(\ref{moo3})\Rightarrow(\ref{moo4})$ by Theorem
\ref{RepMatrix}.
Clearly $(\ref{moo4})\Rightarrow(\ref{moo5})$. To see that
$(\ref{moo5})\Rightarrow(\ref{moo3})$, let $A=\bigoplus_i n_i V_i$. Consider $\End_A(A)$ (endomorphisms of $A$ as a representation of
$A$). As the $V_i$'s are pairwise non-isomorphic, by Schur's lemma, no copy of $V_i$ in $A$ can be mapped to a distinct $V_j$. Also, again by
Schur's lemma, $\End_A{\left(V_i\right)}=k$. Thus, $\End_A(A)\cong\bigoplus_i \operatorname{Mat}_{n_i}(k)$. But $\End_A(A)\cong A^{\text{op}}$
by Problem \ref{1:3}, so $A^{\text{op}}\cong\bigoplus_i \operatorname{Mat}_{n_i}(k)$. Thus, $A \cong \left(\bigoplus_i
\operatorname{Mat}_{n_i}(k)\right)^{\text{op}} = \bigoplus_i
\operatorname{Mat}_{n_i}(k)$, as desired.
\end{proof}

\subsection{Characters of representations}

Let $A$ be an algebra and $V$ a finite-dimensional representation
of $A$ with
action $\rho$.
Then the \emph{character} of $V$ is the linear function
$\chi_V : A \to k$ given by
\[
\chi_V(a) = \text{tr}|_V (\rho(a)).
\]
If $[A,A]$ is the span of commutators $[x,y]:= xy-yx$ over all $x,y\in
A$, then $[A,A] \subseteq \ker \chi_V$.
Thus, we may view the character as a mapping $\chi_V : A/[A,A] \to k$.

\vskip .05in

{\bf Exercise.} Show that if $W\subset V$ are finite dimensional
representations of $A$, then $\chi_V=\chi_W+\chi_{V/W}$.

\begin{theorem}\label{char}

(i) Characters of (distinct)
irreducible finite-dimensional representations of $A$ are
linearly independent.

(ii) If $A$ is a finite-dimensional semisimple algebra, then these
characters form
a basis of
$(A/[A,A])^*.$
\end{theorem}
\begin{proof}
(i) If $V_1, \dots, V_r$ are nonisomorphic irreducible finite-dimensional
representations of $A$, then
$
\rho_{V_1}\oplus\dots\oplus\rho_{V_r}:A\to\text{End }V_1\oplus\dots\oplus\text{End }V_r
$
is surjective by the density theorem, so $\chi_{V_1}, \dots, \chi_{V_r}$ are
linearly independent.
(Indeed, if $\sum \lambda_i\chi_{V_i}(a) = 0$ for all $a\in A$, then $\sum
\lambda_i \text{Tr}(M_i) = 0$ for all $M_i\in \text{End}_k V_i$. But each
$\text{tr}(M_i)$ can range independently over $k$, so it must be that $\lambda_1 =
\dots = \lambda_r = 0$.)

(ii)
First we prove that
$[\text{Mat}_d (k), \text{Mat}_d (k)] = sl_d(k)$, the set of all matrices with
trace 0. It is clear that $[\text{Mat}_d (k), \text{Mat}_d (k)]
\subseteq sl_d(k)$. If we denote by $E_{ij}$ the matrix with $1$ in
the $i$th row of the $j$th column and 0's everywhere else, we have
$[E_{ij}, E_{jm}] = E_{im}$ for $i\neq m$,
and $[E_{i,i+1}, E_{i+1,i}] = E_{ii} - E_{i+1,i+1}.$
Now $\{E_{im}\} \cup \{E_{ii}-E_{i+1,i+1}\}$ forms a basis in $sl_d(k)$, so
indeed $[\text{Mat}_d (k),\text{Mat}_d (k)] = sl_d(k)$, as claimed.

By semisimplicity, we can write $A = \text{Mat}_{d_1}(k) \oplus \dots
\oplus \text{Mat}_{d_r}(k).$ Then $[A,A] = sl_{d_1}(k)\oplus \dots \oplus
sl_{d_r}(k)$, and $A/[A,A] \cong k^r$. By Theorem \ref{RepMatrix},
there are exactly $r$ irreducible representations of $A$ (isomorphic to
$k^{d_1}, \dots, k^{d_r}$, respectively), and therefore $r$
linearly independent
characters on the $r$-dimensional vector space $A/[A,A]$.
Thus, the characters
form a basis.
\end{proof}

\subsection{The Jordan-H\"older theorem}

We will now state and prove two important theorems
about representations of finite dimensional algebras -
the Jordan-H\"older theorem and the Krull-Schmidt theorem.

\begin{theorem} (Jordan-H\"older theorem).
Let $V$ be a finite dimensional representation of $A$,
and $0=V_0\subset V_1\subset...\subset V_n=V$,
$0=V_0'\subset...\subset V_m'=V$ be filtrations
of $V$, such that the representations
$W_i:=V_i/V_{i-1}$ and $W_i':=V_i'/V_{i-1}'$ are
irreducible for all $i$. Then $n=m$, and there exists a
permutation $\sigma$ of $1,...,n$ such that $W_{\sigma(i)}$ is
isomorphic to $W_i'$.
\end{theorem}

\begin{proof} {\bf First proof} (for $k$ of characteristic zero).
The character of $V$ obviously equals the sum of
characters of $W_i$, and also the sum of characters of $W_i'$.
But by Theorem \ref{char}, the characters of irreducible
representations are linearly independent, so the multiplicity of
every irreducible representation $W$ of $A$ among $W_i$ and among
$W_i'$ are the same. This implies the theorem.
\footnote{This proof does not work in characteristic $p$
because it only implies that the multiplicities of $W_i$ and $W_i'$
are the same modulo $p$, which is not sufficient. In fact,
the character of the representation $pV$, where $V$ is any
representation, is zero.}

{\bf Second proof} (general). The proof is by induction on $\dim
V$. The base of induction is clear, so let us prove the induction
step. If $W_1= W_1'$ (as subspaces), we are done, since
by the induction assumption the theorem holds for $V/W_1$.
So assume $W_1\ne W_1'$. In this case $W_1\cap W_1'=0$ (as $W_1,W_1'$ are
irreducible), so we have an
embedding $f: W_1\oplus W_1'\to V$. Let $U=V/(W_1\oplus W_1')$,
and $0=U_0\subset U_1\subset...\subset U_p=U$ be a filtration of
$U$ with simple quotients $Z_i=U_i/U_{i-1}$ (it exists by Lemma \ref{simp}).
Then we see that:

1) $V/W_1$ has a filtration with successive quotients
$W_1',Z_1,...,Z_p$, and another filtration
with successive quotients $W_2,....,W_n$.

2) $V/W_1'$ has a filtration with successive quotients
$W_1,Z_1,...,Z_p$, and another filtration
with successive quotients $W_2',....,W_n'$.

By the induction assumption, this means that the
collection of irreducible representations with multiplicities
$W_1,W_1',Z_1,...,Z_p$ coincides on one hand with
$W_1,...,W_n$, and on the other hand, with $W_1',...,W_m'$.
We are done.
\end{proof}

The Jordan-H\"older theorem shows that the number $n$
of terms in a filtration of $V$ with irreducible successive
quotients does not depend on the choice of a filtration, and
depends only on $V$. This number is called the {\it length}
of $V$. It is easy to see that $n$ is also the maximal length of a
filtration of $V$ in which all the inclusions are strict.

The sequence of the irreducible representations $W_1,...,W_n$ 
enumerated in the order they appear from some filtration 
of $V$ as successive quoteints is 
called a {\bf Jordan-H\"older series} of $V$. 

\subsection{The Krull-Schmidt theorem}

\begin{theorem} (Krull-Schmidt theorem)
Any finite dimensional representation of $A$
can be uniquely (up to an isomorphism and order of summands) decomposed into
a direct sum of indecomposable representations.
\end{theorem}

\begin{proof} It is clear that a decomposition of $V$
into a direct sum of indecomposable representations
exists, so we just need to prove uniqueness.
We will prove it by induction on $\dim V$. Let
$V=V_1\oplus...\oplus V_m=V_1'\oplus...\oplus V_n'$.
Let $i_s: V_s\to V$, $i_s': V_s'\to V$,
$p_s: V\to V_s$, $p_s': V\to V_s'$ be the natural maps
associated to these decompositions. Let $\theta_s=p_1i_s'p_s'i_1: V_1\to
V_1$. We have $\sum_{s=1}^n \theta_s=1$.
Now we need the following lemma.

\begin{lemma} Let $W$ be a finite dimensional indecomposable representation of
$A$. Then

(i) Any homomorphism $\theta: W\to W$ is either an isomorphism or
nilpotent;

(ii) If $\theta_s: W\to W$, $s=1,...,n$ are nilpotent
homomorphisms, then so is $\theta:=\theta_1+...+\theta_n$.
\end{lemma}

\begin{proof} (i) Generalized eigenspaces of $\theta$ are
subrepresentations of $W$, and $W$ is their direct sum.
Thus, $\theta$ can have only one eigenvalue $\lambda$. If $\lambda$ is zero,
$\theta$ is nilpotent, otherwise it is an isomorphism.

(ii) The proof is by induction in $n$. The base is clear.
To make the induction step ($n-1$ to $n$), assume that $\theta$ is not nilpotent.
Then by (i) $\theta$  is an isomorphism,
so $\sum_{i=1}^n\theta^{-1}\theta_i=1$.
The morphisms $\theta^{-1}\theta_i$ are not isomorphisms, so they
are nilpotent. Thus
$1-\theta^{-1}\theta_n=\theta^{-1}\theta_1+...+\theta^{-1}\theta_{n-1}$
is an isomorphism, which is a contradiction with the induction
assumption.
\end{proof}

By the lemma, we find that for some $s$, $\theta_s$ must be an
isomorphism; we may assume that $s=1$. In this case,
$V_1'={\rm Im}(p_1'i_1)\oplus {\rm
Ker}(p_1i_1')$, so since $V_1'$ is indecomposable, we get
that $f:=p_1'i_1: V_1\to V_1'$ and $g:=p_1i_1': V_1'\to V_1$ are
isomorphisms.

Let $B=\oplus_{j>1}V_j$, $B'=\oplus_{j>1}V_j'$; then we have
$V=V_1\oplus B=V_1'\oplus B'$. Consider the map $h: B\to B'$
defined as a composition of the natural maps $B\to V\to B'$
attached to these decompositions. We claim that $h$ is an
isomorphism. To show this, it suffices to show that ${\rm
Ker}h=0$ (as $h$ is a map between spaces of the same dimension).
Assume that $v\in {\rm Ker}h\subset B$. Then $v\in V_1'$. On the
other hand, the projection of $v$ to $V_1$ is zero, so
$gv=0$. Since $g$ is an isomorphism, we get $v=0$, as desired.

Now by the induction assumption, $m=n$, and $V_j\cong
V_{\sigma(j)}'$ for some
permutation $\sigma$ of $2,...,n$. The theorem is proved.
\end{proof}

{\bf Exercise.} Let $A$ be the algebra of
real-valued continuous functions on $\Bbb R$ which are periodic with period $1$.
Let $M$ be the $A$-module of continuous functions $f$ on $\Bbb R$ 
which are antiperiodic with period $1$, i.e., $f(x+1)=-f(x)$. 

(i) Show that $A$ and $M$ are indecomposable $A$-modules. 

(ii) Show that $A$ is not isomorphic to $M$ but $A\oplus A$ is isomorphic to $M\oplus M$. 

{\bf Remark.} Thus, we see that in general, the Krull-Schmidt theorem fails for infinite dimensional modules. 
However, it still holds for modules of {\it finite length}, i.e., modules $M$ such that any filtration 
of $M$ has length bounded above by a certain constant $l=l(M)$. 

\subsection{Problems}

\begin{problem}\label{2:1}{\bf Extensions of representations.}
Let $A$ be an algebra, and
$V,W$ be a pair of representations of $A$. We would like to
classify representations $U$ of $A$ such that
$V$ is a subrepresentation of $U$, and $U/V=W$.
Of course, there is an obvious example $U=V\oplus W$, but are
there any others?

Suppose we have a representation $U$ as above.
As a vector space, it can be (non-uniquely)
identified with $V\oplus W$, so that
for any $a\in A$ the corresponding operator $\rho_U(a)$ has block
triangular form
$$
\rho_U(a)=\begin{pmatrix} \rho_V(a) & f(a)\\
0 & \rho_W(a)\end{pmatrix},
$$
where $f: A\to {\rm Hom}_k(W,V)$ is a linear map.

(a) What is the necessary and sufficient condition on $f(a)$
under which $\rho_U(a)$ is a representation? Maps $f$ satisfying
this condition are called (1-)cocycles (of $A$ with coefficients
in ${\rm Hom}_k(W,V)$). They form a vector space denoted $Z^1(W,V)$.

(b) Let $X: W\to V$ be a linear map.
The coboundary of $X$, $dX$, is defined to
be the function $A\to {\rm Hom}_k(W,V)$
given by $dX(a)=\rho_V(a)X-X\rho_W(a)$.
Show that $dX$ is a cocycle, which vanishes if and only if
$X$ is a homomorphism of representations.
Thus coboundaries form a subspace $B^1(W,V)\subset Z^1(W,V)$,
which is isomorphic to ${\rm Hom}_k(W,V)/{\rm Hom}_A(W,V)$.
The quotient $Z^1(W,V)/B^1(W,V)$ is denoted ${\rm Ext}^1(W,V)$.

(c) Show that if $f,f'\in Z^1(W,V)$ and $f-f'\in B^1(W,V)$
then the corresponding extensions $U,U'$ are isomorphic
representations of $A$. Conversely, if $\phi: U\to U'$
is an isomorphism such that
$$
\phi(a)=\begin{pmatrix} 1_V & *\\
0 & 1_W\end{pmatrix}
$$
then $f-f'\in B^1(V,W)$. Thus, the space ${\rm Ext^1}(W,V)$
``classifies'' extensions of $W$ by $V$.

(d) Assume that $W,V$ are finite dimensional
irreducible representations of $A$.
For any $f\in Ext^1(W,V)$, let $U_f$ be the corresponding
extension. Show that $U_{f}$ is isomorphic
to $U_{f'}$ as representations if and only if
$f$ and $f'$ are proportional.
Thus isomorphism classes (as representations)
of nontrivial extensions of $W$ by $V$
(i.e., those not isomorphic to $W\oplus V$)
are parametrized by the projective space ${\Bbb
P}{\rm Ext}^1(W,V)$. In particular, every extension
is trivial if and only if ${\rm Ext}^1(W,V)=0$.
\end{problem}

\begin{problem}\label{2:2}
 (a) Let $A=\bold C[x_1,...,x_n]$, and $V_a,V_b$ be one-dimensional
representations in which $x_i$ act by $a_i$ and $b_i$,
respectively ($a_i,b_i\in \bold C$). Find ${\rm Ext}^1(V_a,V_b)$
and classify 2-dimensional representations of $A$.

(b) Let $B$ be the algebra over $\bold C$ generated by
$x_1,...,x_n$ with the defining relations $x_ix_j=0$ for all
$i,j$. Show that for $n>1$ the algebra
$B$ has infinitely many non-isomorphic indecomposable representations.
\end{problem}

\begin{problem}\label{2:3} Let $Q$ be a quiver without oriented cycles,
and $P_Q$ the path algebra of $Q$.
Find irreducible representations of $P_Q$ and compute ${\rm
Ext}^1$ between them. Classify 2-dimensional representations of
$P_Q$.
\end{problem}

\begin{problem}\label{2:4}
Let $A$ be an algebra, and $V$ a representation of $A$.
Let $\rho: A\to {\rm End}V$. A formal deformation of $V$ is a formal series
$$
\tilde\rho=\rho_0+t\rho_1+...+t^n\rho_n+...,
$$
where $\rho_i: A\to {\rm End}(V)$ are linear maps,
$\rho_0=\rho$, and $\tilde\rho(ab)=\tilde\rho(a)\tilde\rho(b)$.

If $b(t)=1+b_1t+b_2t^2+...$, where $b_i\in {\rm End}(V)$,
and $\tilde\rho$ is a formal deformation of $\rho$,
then $b\tilde\rho b^{-1}$ is also a deformation of $\rho$,
which is said to be isomorphic to $\tilde\rho$.

(a) Show that if ${\rm Ext}^1(V,V)=0$, then any deformation
of $\rho$ is trivial, i.e., isomorphic to $\rho$.

(b) Is the converse to (a) true? (consider the algebra of dual
numbers $A=k[x]/x^2$).
\end{problem}

\begin{problem}{\bf The Clifford algebra.} Let $V$ be a finite
dimensional complex vector space equipped with a symmetric bilinear
form $(,)$. The {\it Clifford algebra} ${\rm Cl}(V)$ is the quotient
of the tensor algebra $TV$ by the ideal generated by
the elements $v\otimes v-(v,v)1$, $v\in V$. More explicitly, if
$x_i, 1\le i\le N$ is a basis of $V$ and $(x_i,x_j)=a_{ij}$ then
${\rm Cl}(V)$ is generated by $x_i$ with defining relations
$$
x_ix_j+x_jx_i=2a_{ij}, x_i^2=a_{ii}.
$$
Thus, if $(,)=0$, ${\rm Cl}(V)=\wedge V$.

(i) Show that if $(,)$ is nondegenerate then ${\rm Cl}(V)$ is semisimple,
and has one irreducible representation of dimension $2^n$ if
$\dim V=2n$ (so in this case ${\rm Cl}(V)$ is a matrix algebra), and
two such representations if $\dim(V)=2n+1$ (i.e., in this case
${\rm Cl}(V)$ is a direct sum of two matrix algebras).

Hint. In the even case, pick a basis $a_1,...,a_n,b_1,...,b_n$ of
$V$ in which $(a_i,a_j)=(b_i,b_j)=0$, $(a_i,b_j)=\delta_{ij}/2$, and
construct a representation of ${\rm Cl}(V)$ on $S:=\wedge
(a_1,...,a_n)$ in which $b_i$ acts as ``differentiation'' with
respect to $a_i$. Show that $S$ is irreducible.
In the odd case the situation is similar,
except there should be an additional basis vector $c$ such that
$(c,a_i)=(c,b_i)=0$, $(c,c)=1$, and the action of $c$ on $S$ may
be defined either by $(-1)^{\rm degree}$ or
by $(-1)^{\rm degree+1}$, giving two representations $S_+,S_-$
(why are they non-isomorphic?). Show that there is no other
irreducible representations by finding a spanning set
of ${\rm Cl}(V)$ with $2^{\dim V}$ elements.

(ii) Show that ${\rm Cl}(V)$ is semisimple if and only if $(,)$ is
nondegenerate. If $(,)$ is degenerate, what is
${\rm Cl}(V)/\Rad({\rm Cl}(V))$?
\end{problem}

\subsection{Representations of tensor products}

Let $A,B$ be algebras. Then $A\otimes B$ is also an algebra, with
multiplication $(a_1\otimes b_1)(a_2\otimes b_2)=a_1a_2\otimes b_1b_2$. 

{\bf Exercise.} Show that ${\rm Mat}_m(k)\otimes {\rm Mat}_n(k)\cong {\rm Mat}_{mn}(k)$.  

The following theorem describes irreducible finite
dimensional representations of $A\otimes B$ in terms of
irreducible finite dimensional representations of $A$ and those
of $B$.

\begin{theorem}\label{produ}
(i) Let $V$ be an irreducible finite dimensional representation of $A$ and $W$ an
irreducible finite dimensional representation of $B$. Then
$V\otimes W$ is an irreducible representation of $A\otimes B$.

(ii) Any irreducible finite dimensional representation $M$ of $A\otimes B$
has the form (i) for unique $V$ and $W$.
\end{theorem}

\begin{remark}
Part (ii) of the theorem typically fails for infinite dimensional
representations; e.g. it fails when $A$ is the Weyl algebra in
characteristic zero. Part (i) also may fail. E.g. let $A=B=V=W=\Bbb
C(x)$. Then (i) fails, as $A\otimes B$ is not a field.
\end{remark}

\begin{proof}
(i) By the density theorem, the maps $A\to \End V$ and $B\to \End
W$ are surjective. Therefore, the map $A\otimes B\to \End
V\otimes \End W=\End (V\otimes W)$ is surjective.
Thus, $V\otimes W$ is irreducible.

(ii) First we show the existence of $V$ and $W$. Let $A',B'$ be
the images of $A,B$ in $\End M$. Then $A',B'$ are finite
dimensional algebras, and $M$ is a representation of $A'\otimes
B'$, so we may assume without loss of generality that $A$ and $B$
are finite dimensional.

In this case, we claim that
$\Rad (A\otimes B)=\Rad(A)\otimes B+A\otimes \Rad(B)$.
Indeed, denote the latter by $J$. Then $J$ is a nilpotent ideal
in $A\otimes B$, as $\Rad(A)$ and $\Rad(B)$ are nilpotent.
On the other hand, $(A\otimes B)/J=(A/\Rad(A))\otimes (B/\Rad(B))$,
which is a product of two semisimple algebras, hence semisimple.
This implies $J\supset \Rad(A\otimes B)$. Altogether, by
Proposition \ref{nilp}, we see that $J=\Rad(A\otimes B)$,
proving the claim.

Thus, we see that
$$
(A\otimes B)/\Rad(A\otimes B)=A/\Rad(A)\otimes B/\Rad(B).
$$
Now, $M$ is an irreducible representation of $(A\otimes B)/\Rad(A\otimes B)$,
so it is clearly of the form $M=V\otimes W$, where $V$ is an irreducible
representation of $A/\Rad(A)$ and $W$ is an irreducible
representation of $B/\Rad(B)$, and $V,W$ are uniquely determined
by $M$ (as all of the algebras involved are direct sums of matrix algebras).
\end{proof}

\newpage \section{Representations of finite groups: basic results}

Recall that a {\bf representation} of a group $G$ over
a field $k$ is a $k$-vector space $V$ together with a group
homomorphism  $\rho: G\to GL(V)$. As we have explained above, a
representation of a group $G$ over $k$ is the same thing as a representation of
its group algebra $k[G]$.

In this section, we begin a systematic development of
representation theory of finite groups.

\subsection{Maschke's Theorem}

\begin{theorem}\label{maschk}(Maschke)
Let $G$ be a finite group and $k$ a field whose
characteristic does not divide $\left| G \right|$. Then:

(i) The algebra $k[G]$ is semisimple.

(ii) There is an isomorphism of algebras $\psi: k[G]\to \oplus_i
{\rm End}V_i$ defined by $g\mapsto \oplus_i g|_{V_i}$, where $V_i$ are the irreducible
representations of $G$. In particular, this is an isomorphism of
representations of $G$ (where $G$ acts on both sides by left
multiplication). Hence, the regular representation
$k[G]$ decomposes into irreducibles as $\oplus_i \dim(V_i)V_i$, and
one has
$$
|G|=\sum_i \dim(V_i)^2.
$$
(the ``sum of squares formula'').
\end{theorem}
\begin{proof}
By Proposition \ref{semisi}, (i) implies (ii), and to prove (i),
it is sufficient
to show that if $V$ is a finite-dimensional representation of $G$ and
$W\subset V$ is any subrepresentation,
then there exists a subrepresentation $W'
\subset V$ such that $V = W \oplus W'$ as representations.

Choose any complement $\hat{W}$ of $W$ in $V$. (Thus
$V = W \oplus \hat{W}$ as \emph{vector spaces}, but not
necessarily as \emph{representations}.) Let $P$ be the projection along
$\hat{W}$ onto $W$, i.e., the operator on $V$ defined by $P|_W = \text{Id}$ and
$P|_{\hat{W}} = 0$. Let
\[
\overline{P} := \frac{1}{|G|} \sum_{g\in G} \rho(g)P\rho(g^{-1}),
\]
where $\rho(g)$ is the action of $g$ on $V$, and let
\[
W' = \ker\overline{P}.
\]
Now $\overline{P}|_W = \text{Id}$ and $\overline{P}(V) \subseteq W$, so
$\overline{P}^2 = \overline{P}$, so $\overline{P}$ is a projection along $W'$.
Thus, $V = W \oplus W'$ as vector spaces.

Moreover, for any $h\in G$ and any $y \in W'$,
\[
\overline{P} \rho(h) y = \frac{1}{|G|}\sum_{g\in G} \rho(g)P\rho(g^{-1}h)y
= \frac{1}{|G|}\sum_{\ell\in G} \rho(h\ell)P\rho(\ell^{-1})y
= \rho(h) \overline{P} y = 0,
\]
so $\rho(h) y \in \ker\overline{P}= W'$.
Thus, $W'$ is invariant under the action of $G$ and is therefore a
subrepresentation of $V$.
Thus, $V = W
\oplus W'$ is the desired decomposition into subrepresentations.
\end{proof}

The converse to Theorem \ref{maschk}(i) also holds. 

\begin{proposition}
If $k[G]$ is semisimple, then the characteristic of $k$ does not
divide $|G|$.
\end{proposition}
\begin{proof}
Write
$
k[G] = \bigoplus_{i=1}^r \text{End } V_i,
$
where the $V_i$ are irreducible representations and
$V_1 = k$ is the trivial one-dimensional representation. Then
\[
k[G] = k\oplus\bigoplus_{i=2}^r\text{End }V_i = k\oplus\bigoplus_{i=2}^r d_iV_i,
\]
where $d_i = \dim V_i$. By Schur's Lemma,
\begin{gather*}
\text{Hom}_{k[G]} (k, k[G]) = k\Lambda \\
\text{Hom}_{k[G]} (k[G], k) = k\epsilon,
\end{gather*}
for nonzero homomorphisms of representations $\epsilon:k[G]\to k$ and $\Lambda:k\to k[G]$ unique
up to scaling. We can take $\epsilon$ such that $\epsilon (g) = 1$ for all $g
\in G$, and $\Lambda$ such that $\Lambda (1) = \sum_{g\in G} g$. Then
\[
\epsilon\circ\Lambda (1)=\epsilon\biggl(\sum_{g\in G}g\biggr)=\sum_{g\in G}1=|G|.
\]
If $|G| = 0$, then $\Lambda$ has no left inverse, as
$(a\epsilon)\circ \Lambda(1)=0$ for any $a\in k$.
This is a contradiction.
\end{proof}

\begin{example} If $G=\Bbb Z/p\Bbb Z$ and $k$ has characteristic
$p$, then every irreducible representation of $G$ over $k$ is
trivial (so $k[\Bbb Z/p\Bbb Z]$ indeed is not
semisimple). Indeed, an irreducible representation of this group
is a 1-dimensional space, on which the generator acts by a
$p$-th root of unity, and every $p$-th root of unity in $k$
equals $1$, as $x^p-1=(x-1)^p$ over $k$.
\end{example}

\begin{problem}
Let $G$ be a group of order $p^n$. Show that
every irreducible representation of $G$ over a field $k$ of
characteristic $p$ is trivial.
\end{problem}

\subsection{Characters}

If $V$ is a finite-dimensional
representation of a finite group $G$, then its character $\chi_V:
G\to k$ is
defined by the formula $\chi_V(g) = \text{tr}|_V(\rho(g))$.
Obviously, $\chi_V(g)$ is simply the restriction
of the character $\chi_V(a)$ of $V$ as a representation of the
algebra $A=k[G]$ to the basis $G\subset A$, so it carries
exactly the same information.
The character is a \emph{central} or \emph{class function}:
$\chi_V(g)$ depends only on the conjugacy class of $g$; i.e.,
$\chi_V(hgh^{-1}) = \chi_V(g)$.

\begin{theorem}
If the characteristic of $k$ does not divide $|G|$,
characters of irreducible representations of $G$ form a basis in the space
$F_c(G,k)$ of class functions on $G$.
\end{theorem}
\begin{proof}
By the Maschke theorem, $k[G]$ is semisimple, so by Theorem
\ref{char}, the
characters are linearly independent and are a basis of $(A/[A,A])^*$, where $A =
k[G]$. It suffices to note that, as vector spaces over $k$,
\begin{align*}
(A/[A,A])^* & \cong \{\varphi \in \text{Hom}_k (k[G],k) \mid gh-hg \in
\ker \varphi \ \forall g,h\in G\} \\
& \cong \{f \in \text{Fun} (G, k) \mid f(gh) = f(hg) \ \forall g,h\in G\},
\end{align*}
which is precisely $F_c(G,k)$.
\end{proof}
\begin{corollary}
The number of isomorphism classes of 
irreducible representations of $G$ equals the number of conjugacy
classes of $G$ (if $|G|\ne 0$ in $k$).
\end{corollary}

{\bf Exercise.} Show that if $|G|=0$ in $k$ then the number of isomorphism classes of irreducible representations 
of $G$ over $k$ is strictly less than the number of conjugacy classes in $G$. 

Hint. Let $P=\sum_{g\in G}g\in k[G]$. Then $P^2=0$. So $P$ has zero trace in every 
finite dimensional representation of $G$ over $k$. 

\begin{corollary}
Any representation of $G$ is determined by
its character if $k$ has characteristic 0; namely,
$\chi_V = \chi_W$ implies
$V \cong W$.
\end{corollary}

\subsection{Examples}

The following are examples of representations of finite groups over $\CC$.

\begin{enumerate}
\item
Finite abelian groups $G = \ZZ_{n_1} \times \dots \times \ZZ_{n_k}$.
Let $G^\vee$ be the set of irreducible representations of $G$.
Every element of $G$ forms a conjugacy class, so $|G^\vee| = |G|$.
Recall that all irreducible representations
over $\CC$ (and algebraically closed fields in general)
of commutative algebras and groups
are one-dimensional.
Thus, $G^\vee$ is an abelian group: if $\rho_1, \rho_2 : G \to \CC^\times$
are irreducible representations then so are
$\rho_1(g)\rho_2(g)$ and $\rho_1(g)^{-1}$.
$G^\vee$ is called the \emph{dual} or \emph{character group} of $G$.

For given $n\geq 1$, define $\rho:\ZZ_n\to\CC^\times$ by
$\rho(m)=e^{2\pi im/n}$. Then $\ZZ_n^\vee =
\{\rho^k:k=0,\dots ,n-1\}$, so $\ZZ_n^\vee\cong\ZZ_n$.
In general,
\[
(G_1\times G_2\times\dots\times G_n)^\vee
    = G_1^\vee\times G_2^\vee\times\dots\times G_n^\vee,
\]
so $G^\vee \cong G$ for any finite abelian group $G$.
This isomorphism is, however, noncanonical: the particular decomposition of $G$ as
$\ZZ_{n_1}\times\dots\times\ZZ_{n_k}$ is not unique as far as which elements of
$G$ correspond to $\ZZ_{n_1}$, etc. is concerned.
On the other hand, $G\cong (G^\vee)^\vee$ is a canonical isomorphism, given by
$\varphi:G\to (G^\vee)^\vee$, where $\varphi(g)(\chi) = \chi(g)$.

\item
The symmetric group $S_3$.
In $S_n$, conjugacy classes are determined by cycle decomposition sizes: two
permutations are conjugate if and only if they have the same number of cycles of each
length. For $S_3$, there are 3 conjugacy classes, so there are 3 different
irreducible representations over $\CC$. If their dimensions are $d_1, d_2, d_3$,
then $d_1^2+d_2^2+d_3^2 = 6$, so $S_3$ must have two 1-dimensional and
one 2-dimensional representations. The 1-dimensional representations are the
trivial representation $\Bbb C_+$ given by
$\rho(\sigma) = 1$ and the sign representation
$\Bbb C_-$ given by $\rho(\sigma)=(-1)^\sigma$.

The 2-dimensional representation can be visualized as representing the
symmetries of the equilateral triangle with vertices 1, 2, 3 at
the points
$(\cos 120^\circ, \sin 120^\circ)$,
$(\cos 240^\circ, \sin 240^\circ)$,
$(1, 0)$ of the coordinate plane, respectively. Thus, for example,
\begin{align*}
\rho((12)) &= \left( \begin{matrix} 1 & 0 \\ 0 & -1 \end{matrix} \right), &
\rho((123)) & = \left( \begin{matrix} \cos 120^\circ & -\sin 120^\circ \\
\sin 120^\circ & \cos 120^\circ \end{matrix} \right).
\end{align*}
To show that this representation is irreducible,
consider any subrepresentation $V$. $V$ must be the span of a subset of the
eigenvectors of $\rho((12))$, which are the nonzero multiples of $(1,0)$ and
$(0,1)$. $V$ must also be the span of a subset of the eigenvectors of
$\rho((123))$, which are different vectors.
Thus, $V$ must be either $\CC^2$ or $0$.

\item
The quaternion group $Q_8 = \{\pm 1, \pm i, \pm j, \pm k\}$,
with defining relations
\begin{align*}
i&=jk=-kj, & j&=ki=-ik, & k&=ij=-ji, & -1&=i^2=j^2=k^2.
\end{align*}
The 5 conjugacy classes are
$\{1\}, \{-1\}, \{\pm i\}, \{\pm j\}, \{\pm k\}$, so there are 5 different
irreducible representations, the sum of the squares of whose dimensions is 8,
so their dimensions must be 1, 1, 1, 1, and 2.

The center $Z(Q_8)$ is $\{\pm 1\}$, and $Q_8/Z(Q_8) \cong \ZZ_2\times\ZZ_2$.
The four 1-dimensional irreducible representations of $\ZZ_2\times\ZZ_2$
can be ``pulled back'' to $Q_8$. That is, if $q : Q_8 \to Q_8/Z(Q_8)$ is the
quotient map, and $\rho$ any representation of $Q_8/Z(Q_8)$, then $\rho\circ q$ gives
a representation of $Q_8$.

The 2-dimensional representation is $V = \CC^2$, given by
$\rho(-1) = -\text{Id}$ and
\begin{align}\label{Pauli}
\rho(i) &= \left( \begin{matrix} 0 & 1 \\ -1 & 0\end{matrix}\right), &
\rho(j) &= \left( \begin{matrix} \sqrt{-1} & 0 \\ 0 &
-\sqrt{-1}\end{matrix}\right), &
\rho(k) &= \left( \begin{matrix} 0 & -\sqrt{-1} \\ -\sqrt{-1} & 0
\end{matrix}\right).
\end{align}
These are the Pauli matrices, which arise in quantum mechanics.

\vskip .05in

{\bf Exercise.} Show that the 2-dimensional irreducible
representation of $Q_8$ can be realized in the space of functions
$f:Q_8\to \Bbb C$ such that $f(gi)=\sqrt{-1} f(g)$ (the action of
$G$ is by right multiplication, $g\circ f(x)=f(xg)$).

\item
The symmetric group $S_4$. The order of $S_4$
is 24, and there are 5 conjugacy classes:
\linebreak $e, (12), (123), (1234), (12)(34)$.
Thus the sum of the squares of the dimensions of 5 irreducible representations
is 24. As with $S_3$, there are two of dimension 1: the trivial and sign
representations, $\CC_+$ and $\CC_-$.
The other three must then have dimensions 2, 3,
and 3.
Because $S_3 \cong S_4/\Bbb Z_2\times \Bbb Z_2$, where
$\Bbb Z_2\times \Bbb Z_2$ is $\{e, (12)(34), (13)(24),
(14)(23)\}$,
the 2-dimensional representation of $S_3$ can be pulled back to the
2-dimensional representation of $S_4$, which we will call $\CC^2$.

We can consider $S_4$ as the group of rotations of a cube acting by permuting
the interior diagonals (or, equivalently, on a regular octahedron permuting
pairs of opposite faces); this gives the 3-dimensional representation $\CC^3_+$.

The last 3-dimensional representation is $\CC^3_-$, the product of $\CC^3_+$ with the
sign representation. $\CC^3_+$ and $\CC^3_-$ are different, for if $g$ is a transposition,
$\det g|_{\CC^3_+} = 1$ while
$\det g|_{\CC^3_-} = (-1)^3 = -1$. Note that another realization
of $\CC^3_-$ is by action of $S_4$ by symmetries (not necessarily
rotations) of the regular tetrahedron. Yet another realization 
of this representation 
is the space of functions on the set of 4 elements (on which $S_4$ acts by permutations)
with zero sum of values. 
\end{enumerate}

\subsection{Duals and tensor products of representations}

If $V$ is a representation of a group $G$, then $V^*$ is also a
representation, via
\[
\rho_{V^*}(g) = (\rho_V(g)^*)^{-1} = (\rho_V(g)^{-1})^* = \rho_V(g^{-1})^*.
\]
The character is
$
\chi_{V^*}(g) = \chi_V(g^{-1}).
$

We have
$\chi_V(g) = \sum \lambda_i$, where the $\lambda_i$ are the eigenvalues of
$g$ in $V$. These eigenvalues must be roots of unity because
$
\rho(g)^{|G|} = \rho(g^{|G|}) = \rho(e) = \text{Id}.
$
Thus for complex representations
\[
\chi_{V^*}(g) = \chi_V(g^{-1}) = \sum \lambda_i^{-1} = \sum \overline{\lambda_i}
= \overline{\sum \lambda_i} = \overline{\chi_V(g)}.
\]
In particular, $V \cong V^*$ as \emph{representations} (not just as vector
spaces) if and only if $\chi_V(g) \in \RR$ for all $g\in G$.

If $V,W$ are representations of $G$, then $V \otimes W$ is also a
representation, via
\[
\rho_{V\otimes W} (g) = \rho_V (g) \otimes \rho_W (g).
\]
Therefore, $\chi_{V\otimes W}(g)=\chi_V(g)\chi_W(g)$.

An interesting problem discussed below
is to decompose $V\otimes W$ (for irreducible $V,W$) into the direct sum of
irreducible representations.

\subsection{Orthogonality of characters}
We define a positive definite Hermitian
inner product on ${\rm F}_c(G,\CC)$ (the space of central functions) by
\[
        (f_1,f_2) = \frac{1}{|G|}\sum_{g\in G}{f_1(g)\overline{f_2(g)}}.
\]
The following theorem says that characters of irreducible representations of $G$ form an orthonormal basis of $F_c(G,\CC)$ under this inner product.
\begin{theorem}\label{orthogonality-of-characters}
For any representations $V,W$
\[
	(\chi_V,\chi_W) = \dim\Hom_G(W,V),
\]
and
\[
	(\chi_V,\chi_W) =
	\left\{
		\begin{array}{l}
			1, \; {\rm if} \; V\cong W, \\
			0, \; {\rm if} \; V\ncong W
		\end{array}
	\right.
\]
if $V,W$ are irreducible.
\end{theorem}
\begin{proof}
By the definition
\begin{eqnarray*}
	(\chi_V,\chi_W)
	&=& \frac1{|G|}\sum_{g\in G}\chi_V(g)\overline{\chi_W(g)} = \frac1{|G|}\sum_{g\in G}\chi_V(g)\chi_{W^*}(g) \\
	&=&\frac1{|G|}\sum_{g\in G}{\chi_{V\otimes W^*}(g)} = \tr|_{V\otimes W^*}(P),
\end{eqnarray*}
where $P = \frac1{|G|}\sum_{g\in G}g \in Z(\CC[G]).$
(Here $Z(\CC[G])$ denotes the center of $\CC[G]$).
If $X$ is an irreducible representation of $G$ then
\[
	P|_{X} =
	\left\{
		\begin{array}{l}
			{\rm Id,} \; {\rm if} \; X = \CC, \\
			0, \; {\rm} \; X \ne \CC.
		\end{array}
	\right.
\]
Therefore, for any representation $X$ the operator $P|_X$
is the $G$-invariant projector onto the subspace $X^G$ of
$G$-invariants in $X$. Thus,

\begin{eqnarray*}
	\tr|_{V\otimes W^*}(P)
	&=& \dim\Hom_G(\CC, V\otimes W^*) \\
	&=& \dim(V\otimes W^*)^G = \dim\Hom_G(W,V).
\end{eqnarray*}
\end{proof}

Theorem \ref{orthogonality-of-characters} gives a powerful method of checking 
if a given complex representation $V$ of a finite group $G$ 
is irreducible. Indeed, it implies that 
$V$ is irreducible if and only if $(\chi_V,\chi_V)=1$. 

\textbf{Exercise.}
Let $G$ be a finite group. Let $V_{i}$ be the irreducible
complex representations of $G$.

For every $i$, let 
$$
\psi_{i}=\dfrac{\dim
V_{i}}{\left\vert G\right\vert }\sum\limits_{g\in G}\chi_{V_{i}}\left(
g\right)  \cdot g^{-1}\in \Bbb C\left[  G\right].  $$

\textbf{(i)} Prove that
$\psi_{i}$ acts on $V_{j}$ as the identity if $j=i$, and as the null map if
$j\neq i$.

\textbf{(ii)} Prove that $\psi_i$ are {\bf idempotents}, i.e., $\psi_{i}^{2}=\psi_{i}$ for any
$i$, and $\psi_{i}\psi_{j}=0$ for any $i\neq j$.

\textit{Hint:}
In \textbf{(i)}, notice that $\psi_{i}$ commutes with any element of $k\left[
G\right]  $, and thus acts on $V_{j}$ as an intertwining operator. Corollary
\ref{slacf} thus yields that $\psi_{i}$ acts on $V_{j}$ as a
scalar. Compute this scalar by taking its trace in $V_j$. 

\vskip .1in

Here is another ``orthogonality formula'' for characters, in
which summation is taken over irreducible representations
rather than group elements.

\begin{theorem}\label{column-orthogonality}
Let $g,h\in G$, and let $Z_g$ denote the centralizer of $g$ in $G$. Then
$$
\sum_{V}\chi_V(g)\overline{\chi_V(h)}=
\left\{
		\begin{array}{l}
			|Z_g| \text{ if }g\text{ is conjugate to
}h \\
			0, \text{ otherwise }
		\end{array}\right.
$$
where the summation is taken over all irreducible representations
of $G$.
\end{theorem}

\begin{proof} As noted above,
$\overline{\chi_V(h)}=\chi_{V^*}(h)$, so
the left hand side equals (using Maschke's theorem):
$$
\sum_{V}\chi_V(g)\chi_{V^*}(h)=
{\rm Tr}|_{\oplus_V V\otimes V^*}(g\otimes (h^*)^{-1})=
$$
$$
{\rm Tr}|_{\oplus_V {\rm End}V}(x\mapsto gxh^{-1})=
{\rm Tr}|_{\CC[G]}(x\mapsto gxh^{-1}).
$$
If $g$ and $h$ are not conjugate, this trace is clearly zero,
since the matrix of the operator $x\mapsto gxh^{-1}$ in the basis of
group elements has zero diagonal entries. On the other hand, if
$g$ and $h$ are in the same conjugacy class, the trace is equal
to the number of elements $x$ such that $x=gxh^{-1}$, i.e., the
order of the centralizer $Z_g$ of $g$. We are done.
\end{proof}

{\bf Remark.} Another proof of this result is as follows.
Consider the matrix $U$ whose rows are labeled by irreducible
representations of $G$ and columns by conjugacy classes, with
entries $U_{V,g}=\chi_V(g)/\sqrt{|Z_g|}$. Note that
the conjugacy class of $g$ is $G/Z_g$,
thus $|G|/|Z_g|$ is the number of
elements conjugate to $G$. Thus,
by Theorem \ref{orthogonality-of-characters},
the rows of the matrix $U$ are orthonormal. This means that
$U$ is unitary and hence its columns are also orthonormal,
which implies the statement.

\subsection{Unitary representations. Another proof of Maschke's theorem for
complex representations}

\begin{definition}
A unitary finite dimensional
representation of a group $G$ is a representation of $G$ on a
complex finite dimensional 
vector space $V$ over $\CC$ equipped
with a $G$-invariant  positive definite Hermitian
form\footnote{We agree that Hermitian forms are linear in the
first argument and antilinear in the second one.}  $(,)$,
i.e.,  such that $\rho_V(g)$ are unitary operators: $(\rho_V(g)v, \rho_V(g)w) = (v,w).$
\end{definition}

\begin{theorem}\label{uni}
If $G$ is finite, then any finite dimensional representation of
$G$ has a unitary structure.
If the representation is irreducible, this
structure is unique up to scaling by a positive real number.
\end{theorem}
\begin{proof}
Take any positive definite form $B$ on $V$ and define another form $\ov{B}$ as follows:
\[
	\ov{B}(v,w) = \sum_{g\in G}B(\rho_V(g)v, \rho_V(g)w)
\]
Then $\ov{B}$ is a positive definite Hermitian form on $V,$
and $\rho_V(g)$ are unitary operators.
If $V$ is an irreducible representation and $B_1, B_2$ are two
positive definite Hermitian forms on $V,$
then $B_1(v,w) = B_2(Av,w)$ for some homomorphism $A: V\to V$
(since any positive definite Hermitian form is
nondegenerate).
By Schur's lemma, $A = \lambda {\rm Id},$ and clearly $\lambda>0$.
\end{proof}

Theorem \ref{uni} implies that if $V$ is a finite dimensional
representation of a finite group $G$, then the {\it complex conjugate
representation} $\overline{V}$ (i.e., the same space $V$ with the same addition
and the same action of $G$, but complex conjugate
action of scalars) is isomorphic to the dual representation
$V^*$. Indeed, a homomorphism of representations $\overline{V}\to V^*$
is obviously the same thing as an invariant sesquilinear form on
$V$ (i.e. a form additive on both arguments which is linear on the 
first one and antilinear on the second one), 
and an isomorphism is the same thing as a nondegenerate
invariant sesquilinear form. So one can use a unitary structure on $V$ to
define an isomorphism $\overline{V}\to V^*$.

\begin{theorem}\label{uni1}
A finite dimensional
unitary representation $V$ of any group $G$ is completely reducible.
\end{theorem}

\begin{proof}
Let $W$ be a subrepresentation of $V$. Let $W^\perp$ be the
orthogonal complement of $W$ in $V$ under the Hermitian inner
product. Then $W^\perp$ is a subrepresentation of $W$, and
$V=W\oplus W^\perp$. This implies that $V$ is completely
reducible.
\end{proof}

Theorems \ref{uni} and \ref{uni1} imply Maschke's theorem for
complex representations (Theorem \ref{maschk}). Thus, we have obtained a new proof of this
theorem over the field of complex numbers. 

\begin{remark} Theorem \ref{uni1} shows that for infinite groups $G$,
a finite dimensional representation may fail
to admit a unitary structure
(as there exist finite dimensional representations,
e.g. for $G=\Bbb Z$, which are indecomposable but not
irreducible).
\end{remark}

\subsection{Orthogonality of matrix elements}

Let $V$ be an irreducible representation of a finite group
$G,$ and $v_1, v_2, \ldots, v_n$ be an orthonormal basis of $V$
under the invariant Hermitian form.
The matrix elements of $V$ are $t^V_{ij}(x) = (\rho_V(x)v_i, v_j).$

\begin{proposition}\label{matr-orthogonality}
	\begin{itemize}
		\item[(i)] Matrix elements of nonisomorphic
irreducible representations are orthogonal in ${\rm Fun}(G,\CC)$ under the form $(f,g) =
			\frac1{|G|}\sum_{x\in G}f(x)\ov{g(x)}.$
		\item[(ii)] $(t^V_{ij}, t^V_{i'j'}) = \delta_{ii'}\delta_{jj'}\cdot \frac1{\dim V}$
	\end{itemize}
Thus, matrix elements of irreducible representations of
$G$ form an orthogonal basis of ${\rm Fun}(G,\Bbb C)$.
\end{proposition}
\begin{proof}
Let $V$ and $W$ be two irreducible representations
of $G.$ Take $\{v_i\}$ to be an orthonormal basis of $V$ and
$\{w_i\}$ to be an orthonormal basis of $W$
under their positive definite invariant Hermitian forms.
Let $w_i^*\in W^*$ be the linear function on $W$ defined by
taking the inner product with $w_i$: $w_i^*(u)=(u,w_i)$.
Then for $x\in G$ we have
$(xw_i^*,w_j^*)=\overline{(xw_i,w_j)}$. Therefore,
putting $P = \frac{1}{|G|}\sum_{x\in G}x,$ we have
\[
(t_{ij}^V,t_{i'j'}^W)=	|G|^{-1}\sum_{x\in G}
( xv_i, v_j ) \ov{( xw_{i'},
w_{j'})} = |G|^{-1}\sum_{x\in G} ( xv_i, v_j)( xw_{i'}^*, w_{j'}^*) = ( P(v_i \otimes w_{i'}^*), v_j\otimes w_{j'}^* )
\]
If $V \ne W,$ this is zero, since $P$ projects to the trivial
representation, which does not occur in $V\otimes W^*$.
If $V = W,$ we need to consider
$( P(v_i \otimes v_{i'}^*), v_j \otimes v_{j'}^*).$
We have a $G$-invariant decomposition
\begin{eqnarray*}
	V\otimes V^* &=& \CC \oplus L \\
	\CC &=& \spann(\sum v_k \otimes v_k^*) \\
	L &=& \spann_{a: \sum_k a_{kk}=0}(\sum_{k,l} a_{kl} v_k \otimes v_l^*),
\end{eqnarray*}
and $P$ projects to the first summand along the second one. 
The projection of $v_i\otimes v_{i'}^*$ to $\CC \subset \CC \oplus L$ is thus
\[
	\frac{\delta_{ii'}}{\dim V} \sum v_k \otimes v_k^*
\]
This shows that
\[
	( P(v_i \otimes v_{i'}^*), v_j
\otimes v_{j'}^*) = \frac{\delta_{ii'}\delta_{jj'}}{\dim V}
\]
which finishes the proof of (i) and (ii). The last statement
follows immediately from the sum of squares formula.
\end{proof}

\subsection{Character tables, examples}

The characters of all the irreducible representations of a finite
group can be arranged into a character table, with conjugacy
classes of elements as the columns, and characters as the rows.
More specifically, the
first row in a character table lists representatives of conjugacy classes, the
second one the numbers of elements in the conjugacy classes, and
the other rows list the values of the characters on the conjugacy
classes. Due to Theorems \ref{orthogonality-of-characters} and
\ref{column-orthogonality} the rows and columns of a character
table are orthonormal with respect
to the appropriate inner products.

Note that in any character table, the row corresponding to the
trivial representation consists of ones, and the column
corresponding to the neutral element consists of the dimensions
of the representations.

Here is, for example, the character table of $S_3:$
\begin{tabular}{|c|c|c|c|}
 \hline
 	$S_3$ & Id & $(12)$ & $(123)$ \\
 \hline
	\# & 1 & 3 & 2 \\
 \hline
	$\CC_+$ & 1 & 1 & 1 \\
 \hline
	$\CC_{-}$ & 1 & -1 & 1 \\
 \hline
	$\CC^2$ & 2 & 0 & -1 \\
 \hline
\end{tabular}

It is obtained by explicitly computing traces in the irreducible
representations.

For another example consider $A_4,$ the group of even
permutations of $4$ items.
There are three one-dimensional representations (as $A_4$ has a
normal subgroup $\Bbb Z_2\oplus \Bbb Z_2$, and $A_4/\Bbb
Z_2\oplus \Bbb Z_2=\Bbb Z_3$). Since there are four
conjugacy classes in total, there is one more irreducible representation of dimension $3.$ Finally, the character table is

\begin{tabular}{|c|c|c|c|c|}
 \hline
 	$A_4$ & Id & $(123)$ & $(132)$ & $(12)(34)$ \\
 \hline
	\# & 1 & 4 & 4 & 3 \\
 \hline
	$\CC$ & 1 & 1 & 1 & 1 \\
 \hline
	$\CC_{\epsilon}$ & 1 & $\epsilon$ & $\epsilon^2$ & 1 \\
 \hline
	$\CC_{\epsilon^2}$ & 1 & $\epsilon^2$ & $\epsilon$ & 1 \\
 \hline
	$\CC^3$ & 3 & 0 & 0 & $-1$ \\
 \hline
\end{tabular}

where $\epsilon = \exp(\frac{2\pi i}3).$

The last row can be computed using the orthogonality of rows.
Another way to compute the last row is to note that
$\Bbb C^3$ is the representation of $A_4$ by rotations
of the regular tetrahedron: in this case $(123),(132)$ are the
rotations by $120^0$ and $240^0$ around a perpendicular to a face
of the tetrahedron, while $(12)(34)$ is the rotation by $180^0$
around an axis perpendicular to two opposite edges.

\begin{example}The following three character tables are of $Q_8,$ $S_4,$ and $A_5$ respectively.

\begin{tabular}{|c|c|c|c|c|c|}
 \hline
 	$Q_8$ & 1 & -1 & $i$ & $j$ & $k$ \\
 \hline
	\# & 1 & 1 & 2 & 2 & 2 \\
 \hline
	$\CC_{++}$ & 1 & 1 & 1 & 1 & 1 \\
 \hline
	$\CC_{+-}$ & 1 & 1 & 1 & -1 & -1 \\
 \hline
	$\CC_{-+}$ & 1 & 1 & -1 & 1 & -1 \\
 \hline
	$\CC_{--}$ & 1 & 1 & -1 & -1 & 1 \\
 \hline
	$\CC^2$ & 2 & -2 & 0 & 0 & 0 \\
 \hline
\end{tabular}

\begin{tabular}{|c|c|c|c|c|c|}
 \hline
 	$S_4$ & Id & $(12)$ & $(12)(34)$ & $(123)$ & $(1234)$ \\
 \hline
	\# & 1 & 6 & 3 & 8 & 6 \\
 \hline
	$\CC_+$ & 1 & 1 & 1 & 1 & 1\\
 \hline
	$\CC_{-}$ & 1 & -1 & 1 & 1 & -1 \\
 \hline
	$\CC^2$ & 2 & 0 & 2 & -1 & 0 \\
 \hline
	$\CC_+^3$ & 3 & -1 & -1 & 0 & 1 \\
 \hline
	$\CC_{-}^3$ & 3 & 1 & -1 & 0 & -1 \\
 \hline
\end{tabular}

\begin{tabular}{|c|c|c|c|c|c|}
 \hline
 	$A_5$ & Id & $(123)$ & $(12)(34)$ & $(12345)$ & (13245) \\
 \hline
	\# & 1 & 20 & 15 & 12 & 12 \\
 \hline
	$\CC$ & 1 & 1 & 1 & 1 & 1 \\
 \hline
	$\CC^3_+$ & 3 & 0 & -1 & $\frac{1+\sqrt5}2$ & $\frac{1-\sqrt5}2$ \\
 \hline
	$\CC^3_-$ & 3 & 0 & -1 & $\frac{1-\sqrt5}2$ & $\frac{1+\sqrt5}2$ \\
 \hline
	$\CC^4$ & 4 & 1 & 0 & -1 & -1 \\
 \hline
	$\CC^5$ & 5 & -1 & 1 & 0 & 0 \\
 \hline
\end{tabular}

Indeed, the computation of the characters of the 1-dimensional
representations is straightforward.

The character of the
2-dimensional representation of $Q_8$ is obtained
from the explicit formula (\ref{Pauli}) for this representation, or by using
the orthogonality.

For $S_4$, the 2-dimensional irreducible representation
is obtained from the 2-dimensional irreducible representation of
$S_3$ via the surjective homomorphism $S_4\to S_3$, which allows
to obtain its character from the character table of $S_3$.

The character of the 3-dimensional representation $\CC^3_+$ is computed from
its geometric realization by rotations of the cube. Namely,
by rotating the cube, $S_4$ permutes the main diagonals. Thus
(12) is the rotation by $180^0$ around an axis that is perpendicular to two
opposite edges, (12)(34) is the rotation
by $180^0$ around an axis that is perpendicular to two
opposite faces, (123) is the rotation around a main diagonal by
$120^0$, and $(1234)$ is the rotation by $90^0$ around an axis
that is perpendicular to two
opposite faces; this allows us to compute the traces easily,
using the fact that the trace of a rotation by the angle $\phi$
in $\Bbb R^3$ is $1+2\cos \phi$.
Now the character of $\CC^3_-$ is found by multiplying the
character of $\CC^3_+$ by the character of the sign
representation.

Finally, we explain how to obtain the character table of $A_5$
(even permutations of 5 items).
The group $A_5$ is the group of rotations of the regular
icosahedron. Thus it has a 3-dimensional
``rotation representation''  $\CC^3_+$, in which
(12)(34) is  the rotation by $180^0$ around an axis perpendicular to
two opposite edges, (123) is the rotation by $120^0$ around an axis
perpendicular to two opposite faces, and (12345), (13254) are
the rotations by $72^0$, respectively $144^0$, around axes going
through two opposite vertices. The character of this
representation is computed from this description in a
straightforward way.

Another representation of $A_5$, which is also 3-dimensional,
is $\CC^3_+$ twisted by the automorphism of $A_5$ given by
conjugation by $(12)$ inside $S_5$. This representation
is denoted by $\CC^3_-$. It has the same character as $\CC^3_+$,
except that the conjugacy classes (12345) and (13245) are
interchanged.

There are two remaining irreducible representations, and
by the sum of squares formula their dimensions are 4 and 5.
So we call them $\CC^4$ and $\CC^5$.

The representation $\CC^4$ is
realized on the space of functions on the set
$\lbrace{1,2,3,4,5\rbrace}$ with zero sum of values, where $A_5$
acts by permutations (check that it is irreducible!).
The character of this representation is
equal to the character of the 5-dimensional permutation
representation minus the character of the 1-dimensional trivial
representation (constant functions). The former at an element $g$
equals to the number of items among 1,2,3,4,5 which are fixed by
$g$.

The representation $\CC^5$ is realized on the space of functions
on pairs of opposite vertices of the icosahedron which
has zero sum of values (check that it is irreducible!).
The character of this representation is
computed similarly to the character of $\CC^4$, or from the
orthogonality formula.
\end{example}

\subsection{Computing tensor product multiplicities using character tables}

Character tables allow us to compute the tensor product
multiplicities
$N_{ij}^k$ using
\[
	V_i \otimes V_j = \sum N_{ij}^k V_k, \quad N_{ij}^k = (\chi_i\chi_j, \chi_k)
\]
\begin{example} The following tables represent computed tensor product multiplicities of irreducible representations of $S_3, S_4,$ and $A_5$ respectively.
\begin{tabular}{|c|c|c|c|}
\hline
$S_3$ & $\CC_+$ & $\CC_{-}$ & $\CC^2$ \\
\hline
 $\CC_+$ & $\CC_+$ & $\CC_{-}$ & $\CC^2$ \\
\hline
$\CC_{-}$ & & $\CC_+$ & $\CC^2$ \\
\hline
$\CC^2$ & & & $\CC_+ \oplus \CC_- \oplus \CC^2$ \\
\hline
\end{tabular}

\begin{tabular}{|c|c|c|c|c|c|}
 \hline
$S_4$ & $\CC_+$ & $\CC_-$ & $\CC^2$ & $\CC^3_+$ & $\CC^3_-$ \\
 \hline
 $\CC_+$ & $\CC_+$ & $\CC_-$ & $\CC^2$ & $\CC^3_+$ & $\CC^3_-$ \\
 \hline
 $\CC_-$ & & $\CC_+$ & $\CC^2$ & $\CC^3_-$ & $\CC^3_+$ \\
 \hline
 $\CC^2$   & & & $\CC_+ \oplus \CC_- \oplus \CC^2$ & $\CC^3_+ \oplus \CC^3_-$ & $\CC^3_+ \oplus \CC^3_-$ \\
 \hline
 $\CC^3_+$ & & & & $\CC_+\oplus \CC^2 \oplus \CC_+^3 \oplus \CC_-^3$ & $\CC_- \oplus \CC^2 \oplus \CC_+^3 \oplus \CC_-^3$ \\
 \hline
 $\CC^3_-$ & & & & & $\CC_+ \oplus \CC^2 \oplus \CC_+^3 \oplus \CC_-^3$ \\
 \hline
\end{tabular}

\begin{tabular}{|c|c|c|c|c|c|}
 \hline
$A_5$ & $\CC$ & $\CC^3_+$ & $\CC^3_-$ & $\CC^4$ & $\CC^5$  \\
 \hline
  $\CC$ & $\CC$ &  $\CC_3^+$ & $\CC^3_-$ & $\CC^4$ & $\CC^5$  \\
 \hline
$\CC^3_+$ & & $\CC \oplus \CC^5 \oplus \CC^3_+$ & $\CC^4 \oplus \CC^5$ & $\CC^3_- \oplus \CC^4 \oplus \CC^5$ & $\CC^3_+ \oplus \CC^3_- \oplus \CC^4 \oplus \CC^5$  \\
 \hline
$\CC^3_-$ & & & $\CC \oplus \CC^5 \oplus \CC^3_+$ & $\CC^3_+ \oplus \CC^4 \oplus \CC^5$ & $\CC^3_+ \oplus \CC^3_- \oplus \CC^4 \oplus \CC^5$ \\
 \hline
$\CC^4$ & & & & $\CC^3_+ \oplus \CC^3_- \oplus \CC \oplus \CC^4 \oplus \CC^5$ & $\CC^3_+ \oplus \CC^3_- \oplus 2\CC^5 \oplus \CC^4$  \\
 \hline
$\CC^5$ & & & & & $\CC \oplus \CC^3_+ \oplus \CC^3_- \oplus 2\CC^4 \oplus 2\CC^5$ \\
 \hline
\end{tabular}
\end{example}

\subsection{Problems}
\begin{problem}\label{2:5} Let $G$ be the group of symmetries of a
regular N-gon (it has 2N elements).

(a) Describe all irreducible complex
representations  of this group (consider the cases of odd and
even $N$)

(b) Let $V$ be the 2-dimensional complex representation of
$G$ obtained by complexification of the standard representation on
the real plane (the plane of the polygon). Find the decomposition
of $V\otimes V$ in a direct sum of irreducible representations.
\end{problem}

\begin{problem}\label{2:6}
Let $G$ be the group of 3 by 3 matrices over $\Bbb F_p$ which are upper
triangular and have ones on the diagonal, under multiplication
 (its order is $p^3$). It is called the Heisenberg
group. For any complex number $z$ such that $z^p=1$
we define a representation of $G$ on the space $V$ of complex functions on
$\Bbb F_p$, by
$$
(\rho\begin{pmatrix} 1&1&0\\ 0&1&0\\
0&0&1\end{pmatrix}f)(x)=f(x-1),
$$
$$
(\rho\begin{pmatrix} 1&0&0\\ 0&1&1\\
0&0&1\end{pmatrix}f)(x)=z^xf(x).
$$
(note that $z^x$ makes sense since $z^p=1$).

(a) Show that such a representation exists and is unique, and compute $\rho(g)$
for all $g\in G$.

(b) Denote this representation by $R_z$. Show that $R_z$ is irreducible
if and only if $z\ne 1$.

(c) Classify all 1-dimensional representations of $G$.
Show that $R_1$ decomposes into a direct sum of 1-dimensional representations,
where each of them occurs exactly once.

(d) Use (a)-(c) and the ``sum of squares'' formula to classify all
irreducible  representations of $G$.
\end{problem}

\begin{problem}\label{3:1}
Let $V$ be a finite dimensional complex vector space, and $GL(V)$
be the group of invertible linear transformations of $V$.
Then $S^nV$ and $\Lambda^mV$ ($m\le \text{dim}(V)$) are representations
of $GL(V)$ in a natural way. Show that they are irreducible representations.

Hint: Choose a basis $\lbrace{e_i\rbrace}$ in
$V$. Find a diagonal element $H$ of $GL(V)$
such that $\rho(H)$ has distinct eigenvalues.
(where $\rho$ is one of the above representations). This shows that
if $W$ is a subrepresentation, then it is spanned by a subset S
of a basis of eigenvectors of $\rho(H)$.
Use the invariance of $W$ under the operators $\rho(1+E_{ij})$ (where
$E_{ij}$ is defined by $E_{ij}e_k=\delta_{jk}e_i$) for all
$i\ne j$
to show that if the subset $S$ is nonempty, it is necessarily
the entire basis.
\end{problem}

\begin{problem}\label{3:2}
Recall that the adjacency matrix of a graph $\Gamma$
(without multiple edges) is the matrix in which the $ij$-th entry
is $1$ if
the vertices $i$ and $j$ are connected with an edge,
and zero otherwise.
Let $\Gamma$ be a finite graph whose
automorphism group is nonabelian.
Show that the adjacency matrix of $\Gamma$
must have repeated eigenvalues.
\end{problem}

\begin{problem}\label{3:3}
Let $I$ be the set of vertices of a regular icosahedron ($|I|=12$).
Let ${\rm Fun}(I)$ be the space of complex functions on $I$.
Recall that the group $G=A_5$ of even permutations of 5 items acts
on the icosahedron, so we have a 12-dimensional representation of
$G$ on ${\rm Fun}(I)$.

(a) Decompose this representation in
a direct sum of irreducible representations (i.e., find
the multiplicities of occurrence of all irreducible representations).

(b) Do the same for the representation of $G$ on the space of
functions on the set of faces and the set of edges of the
icosahedron.
\end{problem}

\begin{problem}\label{3:4}
Let $\Bbb F_q$ be a finite field with $q$ elements, and $G$ be the group
of nonconstant inhomogeneous linear transformations, $x\to ax+b$, over $\Bbb F_q$
(i.e., $a\in \Bbb F_q^\times, b\in \Bbb F_q$). Find all irreducible
complex representations of $G$, and compute their characters.
Compute the tensor products of irreducible representations.

Hint. Let $V$ be the representation of $G$ on the space of
functions on $\Bbb F_q$ with sum of all values equal to zero.
Show that $V$ is an irreducible representation of $G$.
\end{problem}

\begin{problem}\label{3:5}
Let $G=SU(2)$ (unitary 2 by 2 matrices with determinant 1),
and $V={\Bbb C}^2$ the standard 2-dimensional representation of $SU(2)$.
We consider $V$ as a real representation, so it is 4-dimensional.

(a) Show that $V$ is irreducible (as a real representation).

(b) Let ${\Bbb H}$ be the
subspace of ${\rm End}_{\Bbb R}(V)$ consisting of endomorphisms
of $V$ as a real representation. Show that ${\Bbb H}$ is 4-dimensional
and closed under multiplication. Show that every nonzero element
in ${\Bbb H}$ is invertible, i.e., ${\Bbb H}$ is an algebra with division.

(c) Find a basis $1,i,j,k$ of ${\Bbb H}$ such that $1$ is the unit and
$i^2=j^2=k^2=-1$, $ij=-ji=k,jk=-kj=i,ki=-ik=j$.
Thus we have that $Q_8$ is a subgroup of the group ${\Bbb H}^\times$
of invertible elements of ${\Bbb H}$ under multiplication.

The algebra ${\Bbb H}$ is called the quaternion algebra.

(d) For $q=a+bi+cj+dk$, $a,b,c,d\in {\Bbb R}$, let $\bar q=a-bi-cj-dk$,
and $||q||^2=q\bar q=a^2+b^2+c^2+d^2$. Show that
$\overline{q_1q_2}=\bar q_2\bar q_1$, and $||q_1q_2||=||q_1||\cdot||q_2||$.

(e) Let $G$ be the group of quaternions of norm 1. Show that this group
is isomorphic to $SU(2)$.
(Thus geometrically $SU(2)$ is the 3-dimensional sphere).

(f) Consider the action of $G$ on the space $V\subset {\Bbb H}$
spanned by $i,j,k$, by $x\to qxq^{-1}$, $q\in G$, $x\in V$.
Since this action preserves the norm on $V$, we have a homomorphism
$h:SU(2)\to SO(3)$, where $SO(3)$ is the group of rotations of
the three-dimensional Euclidean space. Show that this homomorphism is surjective
and that its kernel is $\{1,-1\}$.
\end{problem}

\begin{problem}\label{6:2} It is known that the classification of finite subgroups of $SO(3)$
is as follows:

1) the cyclic group $\Bbb Z/n\Bbb Z$, $n\ge 1$, generated by a rotation by
$2\pi/n$ around an axis;

2) the dihedral group $D_n$ of order $2n$, $n\ge 2$ (the group of
rotational symmetries in 3-space of a plane containing a regular
$n$-gon\footnote{A regular 2-gon is just a line segment.};

3) the group of rotations of the regular tetrahedron ($A_4$).

4) the group of rotations of the cube or regular octahedron ($S_4$).

5) the group of rotations of a regular dodecahedron or
icosahedron ($A_5$).

(a) Derive this classification.

Hint. Let $G$ be a finite subgroup of $SO(3)$. Consider the action
of $G$ on the unit sphere. A point of the sphere preserved by
some nontrivial element of $G$ is called a pole. Show that every nontrivial
element of $G$ fixes a unique pair of opposite poles, and that the subgroup of
$G$ fixing a particular pole P is cyclic, of some order $m$
(called the order of P). Thus the orbit of $P$ has $n/m$
elements, where $n=|G|$. Now let $P_1,...,P_k$ be the poles
representing all the orbits of $G$ on the set of poles,
and $m_1,...,m_k$ be their orders. By counting nontrivial
elements of $G$, show that
$$
2(1-\frac{1}{n})=\sum_i (1-\frac{1}{m_i}).
$$
Then find all possible $m_i$ and $n$ that can satisfy this
equation and classify the corresponding groups.

(b) Using this classification, classify finite subgroups of $SU(2)$ (use
the homomorphism $SU(2)\to SO(3)$).
\end{problem}

\begin{problem}\label{3:6}
Find the characters and tensor products of irreducible complex
representations of the Heisenberg group from Problem \ref{2:6}.
\end{problem}

\begin{problem}\label{6:1}
Let $G$ be a finite group, and $V$ a complex representation
of $G$ which is faithful, i.e., the corresponding map $G\to GL(V)$
is injective. Show that any irreducible representation of $G$
occurs inside $S^nV$ (and hence inside $V^{\otimes n}$) for some $n$.
\end{problem}

Hint. Show that there exists a vector $u\in V^*$ whose stabilizer
in $G$ is $1$. Now define the map $SV\to {\rm Fun}(G,\Bbb C)$ sending a
polynomial $f$ on $V^*$ to the function $f_u$ on $G$ given by
$f_u(g)=f(gu)$. Show that this map is surjective and use this to
deduce the desired result.

\begin{problem}\label{3:8}
This problem is about an application of representation theory
to physics (elasticity theory). We first describe the physical
motivation and then state the mathematical problem.

Imagine a material which occupies a
certain region $U$ in the physical space $V={\Bbb R}^3$
(a space with a positive definite inner product).
Suppose the material is deformed. This means, we have applied
a diffeomorphism (=change of coordinates)
$g:U\to U'$. The question in elasticity theory
is how much stress in the material this deformation will cause.

For every point $P\in U$, let $A_P:V\to V$ be defined by
$A_P=dg(P)$. $A_P$ is nondegenerate, so it has a polar
decomposition $A_P=D_PO_P$, where $O_P$ is orthogonal and $D_P$ is symmetric.
The matrix $O_P$ characterizes the rotation part of $A_P$ (which clearly
produces no stress), and $D_P$ is the distortion part, which actually
causes stress. If the deformation is small, $D_P$ is close to 1,
so $D_P=1+d_P$, where $d_P$ is a small symmetric matrix,
i.e., an element of $S^2V$. This matrix is
called the deformation tensor at $P$.

Now we define the stress tensor, which characterizes stress.
Let $v$ be a small nonzero vector in $V$, and $\sigma$ a small disk
perpendicular to $v$ centered at $P$ of area $||v||$. Let $F_v$ be the force
with which the part of the material on the $v$-side of $\sigma$
acts on the part on the opposite side. It is easy to deduce from
Newton's laws that $F_v$ is linear in $v$, so there exists
a linear operator $S_P:V\to V$ such that $F_v=S_Pv$.
It is called the stress tensor.

An elasticity law is an equation $S_P=f(d_P)$, where $f$ is a function.
The simplest such law is a linear law (Hooke's law): $f: S^2V\to
End(V)$ is a linear function. In general, such a function
is defined by $9\cdot 6=54$ parameters, but we will show there are actually
only two essential ones -- the compression modulus $K$ and the shearing modulus $\mu$.
For this purpose we will use representation theory.

Recall that the group $SO(3)$ of rotations acts on $V$, so $S^2V$,
$End(V)$ are representations of this group.
The laws of physics must be invariant under this group
(Galileo transformations), so $f$ must be a homomorphism
of representations.

(a) Show that $End(V)$ admits a decomposition ${\Bbb R}\oplus V\oplus W$,
where ${\Bbb R}$ is the trivial representation, $V$ is
the standard 3-dimensional
representation, and $W$ is a 5-dimensional representation
of SO(3). Show that
$S^2V={\Bbb R}\oplus W$

(b) Show that $V$ and $W$ are irreducible, even after complexification.
Deduce using Schur's lemma
that $S_P$ is always symmetric, and for $x\in {\Bbb R},y\in W$ one has
$f(x+y)=Kx+\mu y$ for some real numbers $K,\mu$.

In fact, it is clear from physics that $K,\mu$ are positive.
Physically, the compression modulus $K$ characterises resistance 
of the material to compression or dilation, while the shearing modulus
$\mu$ characterizes its resistance to changing the shape of the object without 
changing its volume. For instance, clay (used for sculpting) has a large compression modulus 
but a small shearing modulus. 

\end{problem}

\newpage \section{Representations of finite groups: further results}

\subsection{Frobenius-Schur indicator}
Suppose that $G$ is a finite group and $V$ is an irreducible
representation of $G$ over $\CC.$
We say that $V$ is
\begin{itemize}
	\item[-] of complex type, if $V \ncong V^*,$
	\item[-] of real type, if $V$
has a nondegenerate symmetric form invariant under $G$,
	\item[-] of quaternionic type, if $V$ has a nondegenerate
skew form invariant under $G.$
\end{itemize}

\begin{problem} (a) Show that $\End_{\RR[G]}V$ is $\CC$ for
$V$ of complex type, $\Mat_2(\RR)$ for $V$ of real type, and
$\HH$ for $V$ of quaternionic type, which motivates the names above.

Hint. Show that the complexification $V_\CC$ of $V$ decomposes as
$V\oplus V^*$. Use this to compute the dimension
of $\End_{{\Bbb R}[G]} V$ in all three cases. Using the fact that
$\Bbb C\subset \End_{{\Bbb R}[G]} V$, prove the result in the
complex case. In the remaining two cases, let $B$ be the
invariant bilinear form on $V$, and $(,)$ the invariant positive Hermitian
form (they are defined up to a nonzero complex scalar and a
positive real scalar, respectively), and define the operator
$j:V\to V$ such that $B(v,w)=(v,jw)$. Show that $j$ is complex
antilinear ($ji=-ij$), and $j^2=\lambda\cdot Id$, where $\lambda$ is a real
number, positive in the real case and negative in the
quaternionic case (if $B$ is renormalized, $j$ multiplies by a
nonzero complex number $z$ and $j^2$ by $z\bar z$, as $j$ is
antilinear).  Thus $j$ can be normalized so that $j^2=1$ for the
real case, and $j^2=-1$ in the quaternionic case. Deduce the
claim from this.

(b) Show that $V$ is of real type if and only if $V$ is the
complexification of a representation $V_{\Bbb R}$ over the field
of real numbers.
\end{problem}

\begin{example}
For $\Bbb Z/n\Bbb Z$ all irreducible representations are
of complex type, except the trivial one and, if $n$ is even,
the ``sign'' representation, $m\to (-1)^m$, which are of real type.
For $S_3$ all three irreducible representations $\CC_+, \CC_-,
\CC^2$ are of real type. For $S_4$ there are
five irreducible representations $\CC_+,$
$\CC_-,$ $\CC^2,$ $\CC_+^3,$ $\CC^3_-,$ which are all
of real type. Similarly, all five irreducible representations of $A_5$ --
$\CC,$ $\CC^3_+,$ $\CC^3_-,$
$\CC^4,$ $\CC^5$ are of real type. As for $Q_8,$ its
one-dimensional representations are of real type, and the two-dimensional one
is of quaternionic type.
\end{example}

\begin{definition}
The Frobenius-Schur indicator $FS(V)$ of an irreducible representation
$V$ is $0$ if it is of complex type, $1$ if it is of real type,
and $-1$ if it is of quaternionic type.
\end{definition}

\begin{theorem}({Frobenius-Schur})\label{frobenius-schur-indicator}
The number of involutions
(=elements of order $\le 2$)
in $G$ is equal to $\sum_V \dim(V)FS(V)$, i.e., the sum of
dimensions of all representations of $G$ of real type minus the sum
of dimensions of its representations of quaternionic type.
\end{theorem}
\begin{proof}
Let $A: V\to V$ have eigenvalues $\lambda_1, \lambda_2, \ldots, \lambda_n.$ We have
\begin{eqnarray*}
\Tr|_{S^2V} (A\ot A)&=& \sum_{i\le j}\lambda_i\lambda_j \\
\Tr|_{\Lambda^2V} (A\ot A) &=& \sum_{i < j}\lambda_i\lambda_j
\end{eqnarray*}

Thus,
\[
	\Tr|_{S^2V}(A\ot A) - \Tr|_{\Lambda^2V}(A\ot A) =
\sum_{1\le i \le n}\lambda_i^2={\rm Tr}(A^2).
\]

Thus for $g\in G$ we have
\[
	\chi_V(g^2) = \chi_{S^2V}(g) - \chi_{\Lambda^2V}(g)
\]
Therefore,
\[
	|G|^{-1}\chi_V(\sum_{g\in G}g^2) =
\chi_{S^2V}(P)-\chi_{\wedge^2V}(P)= \dim (S^2V)^G-\dim (\wedge^2
V)^G=
	\left\{
		\begin{array}{rl}
			1, & {\rm if}\; V\; {\rm is\; of\; real\; type} \\
			-1, & {\rm if}\; V\; {\rm is\; of\;
quaternionic\;
type}\\
			0, & {\rm if}\; V\; {\rm is\; of\;
complex\; type}
	\end{array}
	\right.
\]
Finally, the number of involutions in $G$ equals
\[
	\frac1{|G|}\sum_{V}\dim V\chi_V(\sum_{g\in G}g^2)
=\sum_{{\rm real}\; V}\dim V - \sum_{{\rm quat}\; V}\dim V.
\]
\end{proof}

\begin{corollary}
Assume that all representations of a finite group $G$ are
defined over real numbers (i.e., all complex representations
of $G$ are obtained by complexifying real representations).
Then the sum of dimensions of irreducible representations of
$G$ equals the number of involutions in $G$.
\end{corollary}

{\bf Exercise.} Show that any nontrivial finite group of odd order has 
an irreducible representation which is not defined over $\Bbb R$ 
(i.e., not realizable by real matrices). 

\subsection{Frobenius determinant}

Enumerate the elements of a finite group $G$ as follows:
$g_1,g_2,\ldots,g_n.$ Introduce $n$ variables indexed with the
elements of $G:$
\[
	x_{g_1},x_{g_2},\ldots,x_{g_n}.
\]

\begin{definition}
Consider the matrix $X_G$ with entries $a_{ij}=x_{g_ig_j}.$
The determinant of $X_G$ is some polynomial of degree $n$ of
$x_{g_1},x_{g_2},\ldots,x_{g_n}$ that is called {\it the Frobenius determinant}.
\end{definition}

The following theorem, discovered by Dedekind and proved by
Frobenius, became the starting point for creation of
representation theory (see \cite{Cu}).

\begin{theorem}
 \label{frobdet}
\[
       \det X_G = \prod_{j=1}^{r}{P_j(\xx)^{\deg{P_j}}}
\]
for some pairwise non-proportional irreducible polynomials
$P_j(\xx),$ where $r$ is the number of conjugacy classes of $G$.
\end{theorem}

We will need the following simple lemma.

\begin{lemma} \label{trivial-lemma-det}
Let $Y$ be an $n\times n$ matrix with entries
$y_{ij}.$ Then $\det{Y}$ is an irreducible polynomial of
$\{y_{ij}\}.$
\end{lemma}

\begin{proof}
Let $Y=t\cdot Id+\sum_{i=1}^{n}x_iE_{i,i+1}$, 
where $i+1$ is computed modulo $n$, and $E_{i,j}$ are the elementary matrices. 
Then $\det(Y)=t^n-(-1)^nx_1...x_n$, which is obviously irreducible. 
Hence $\det(Y)$ is irreducible (since factors of a homogeneous 
polynomial are homogeneous).  
\end{proof}

Now we are ready to proceed to the proof of Theorem \ref{frobdet}.

\begin{proof}
Let $V=\BC[G]$ be the regular representation of $G.$ Consider the
operator-valued polynomial
\[
        L(\xx) = \sum_{g\in G}{x_g\rho(g)},
\]
where $\rho(g)\in {\rm End}V$ is induced by $g.$
The action of $L(\xx)$ on an element $h\in G$ is
\[
       L(\xx)h = \sum_{g\in G}{x_g\rho(g)h} = \sum_{g\in G}{x_g}gh = \sum_{z\in
       G}{x_{zh^{-1}}z}
\]
So the matrix of the
linear operator $L(\xx)$ in the basis
$g_1,g_2,\ldots,g_n$ is $X_G$ with permuted columns and hence has
the same determinant up to sign.

Further, by Maschke's theorem, we have
\[{\rm det}_V{L(\xx)} = \prod_{i=1}^r({\rm det}_{V_i}L(\xx))^{\dim{V_i}},\]
where $V_i$ are the irreducible representations of $G$.
We set $P_i = {\rm det}_{V_i}L(\xx).$
Let $\{e_{im}\}$ be bases of $V_i$ and $E_{i,jk} \in \End{V_i}$
be the matrix units in these bases. Then $\{E_{i,jk}\}$ is a
basis of $\BC[G]$ and
\[
	L(\xx)|_{V_i} = \sum_{j,k}{y_{i,jk}}E_{i,jk},
\]
where $y_{i,jk}$ are new coordinates on $\BC[G]$ related to $x_g$
by a linear transformation. Then
\[
	P_i(\xx) =
	\det|_{V_i}L(\xx) = \det(y_{i,jk})
\]
Hence, $P_i$ are irreducible (by Lemma \ref{trivial-lemma-det})
and not proportional to each other (as they depend on different
collections of variables $y_{i,jk}$).
The theorem is proved.
\end{proof}

\subsection{Algebraic numbers and algebraic integers}

We are now passing to deeper results in representation theory
of finite groups. These results require the theory of algebraic
numbers, which we will now briefly review.

\begin{definition}\label{polynomial}
$z \in \BC$ is an {\bf algebraic number} (respectively,
an {\bf algebraic integer}), if $z$
is a root of a monic polynomial with rational (respectively,
integer) coefficients.
\end{definition}

\begin{definition}\label{eigenvalue}
$z \in \BC$ is an {\bf algebraic number}, (respectively,
an {\bf algebraic integer}), if $z$
is an eigenvalue of a matrix with rational (respectively, integer) entries.
\end{definition}

\begin{proposition}
Definitions (\ref{polynomial}) and (\ref{eigenvalue}) are
equivalent.
\end{proposition}

\begin{proof}
To show (\ref{eigenvalue}) $\Rightarrow$ (\ref{polynomial}),
notice that $z$ is a root of the characteristic polynomial of the
matrix (a monic polynomial
with rational, respectively integer, coefficients).\\
To show (\ref{polynomial}) $\Rightarrow$ (\ref{eigenvalue}),
suppose $z$ is a root of
$$p(x)=x^n+a_1x^{n-1}+\dotsc+a_{n-1}x+a_n.$$ Then the characteristic
polynomial of the following matrix (called the {\bf companion
matrix}) is $p(x)$:
\medskip\\
$$\begin{pmatrix}
0 & 0 & 0 & \dotsc & 0 & -a_n\\
1 & 0 & 0 & \dotsc & 0 & -a_{n-1}\\
0 & 1 & 0 & \dotsc & 0 & -a_{n-2}\\
& & & \vdots\\
0 & 0 & 0 & \dotsc & 1 & -a_1\\
\end{pmatrix}.$$\\
\medskip\\
Since $z$ is a root of the characteristic polynomial of
this matrix, it is an eigenvalue of this matrix.
\end{proof}

The set of algebraic numbers is denoted by $\overline{\Bbb Q}$,
and the set of algebraic integers by $\BA$.

\begin{proposition}\label{A_ring}
(i) $\mathbb{A}$ is a ring.

(ii) $\overline{\Bbb Q}$ is a field.
Namely, it is an algebraic closure of the field
of rational numbers.
\end{proposition}

\begin{proof}
We will be using definition (\ref{eigenvalue}). Let $\alpha$ be an
eigenvalue of $$\cA \in {\rm Mat}_n(\BC)$$ with eigenvector $v$,
let $\beta$ be an eigenvalue of $$\mB \in {\rm Mat}_m(\BC)$$ with
eigenvector $w$. Then $\alpha\pm \beta$ is an eigenvalue of $$\cA \ot
\Id_m \pm \Id_n \ot \mB,$$ and $\alpha\beta$ is an eigenvalue of
$$\cA \ot \mB.$$ The corresponding eigenvector is in both cases $v
\ot w$. This shows that both $\BA$ and $\overline{\Bbb Q}$ are
rings. To show that the latter is a field, it suffices to note
that if $\alpha\ne 0$ is a root of a polynomial $p(x)$ of degree
$d$, then $\alpha^{-1}$ is a root of $x^dp(1/x)$.
The last statement is easy, since a number $\alpha$ is algebraic
if and only if it defines a finite extension of $\Bbb Q$.
\end{proof}

\begin{proposition}\label{AcapQ_Z}
$\BA \cap \BQ = \BZ$.
\end{proposition}

\begin{proof}
We will be using definition (\ref{polynomial}). Let $z$ be a root of
$$p(x)=x^n+a_1x^{n-1}+\dotsc+a_{n-1}x+a_n,$$ and suppose
$$z=\frac{p}{q} \in \BQ, \gcd(p,q)=1.$$ Notice that the leading
term of $p(x)$ will have $q^n$ in the denominator, whereas all the other
terms will have a lower power of $q$ there. Thus, if $q \ne
\pm1,$ then $p(z) \notin \BZ,$ a contradiction. Thus, $z \in
\BA \cap \BQ \Rightarrow z \in \BZ.$ The reverse inclusion
follows because $n \in \BZ$ is a root of $x-n$.
\end{proof}

Every algebraic number $\alpha$ has a {\bf minimal polynomial}
$p(x)$, which is the monic polynomial with rational coefficients
of the smallest degree such that $p(\alpha)=0$.
Any other polynomial $q(x)$ with rational coefficients such that
$q(\alpha)=0$ is divisible by $p(x)$. Roots of $p(x)$ are called
the {\bf algebraic conjugates} of $\alpha$; they are roots
of any polynomial $q$ with rational coefficients
such that $q(\alpha)=0$.

Note that any algebraic conjugate of an algebraic integer is
obviously also an algebraic integer.
Therefore, by the Vieta theorem,
the minimal polynomial of an algebraic integer
has integer coefficients.

Below we will need the following lemma:

\begin{lemma}\label{conju}
If $\alpha_1,...,\alpha_m$ are algebraic numbers,
then all algebraic conjugates to $\alpha_1+...+\alpha_m$ are of
the form $\alpha_1'+...+\alpha_m'$, where $\alpha_i'$ are some
algebraic conjugates of $\alpha_i$.
\end{lemma}

\begin{proof}
It suffices to prove this for two summands.
If $\alpha_i$ are eigenvalues of rational matrices $A_i$ of
smallest size (i.e., their characteristic polynomials are the
minimal polynomials of $\alpha_i$), then
$\alpha_1+\alpha_2$ is an eigenvalue of $A:=A_1\otimes \Id+\Id\otimes
A_2$. Therefore, so is any algebraic conjugate to $\alpha_1+\alpha_2$. But
all eigenvalues of $A$ are of the form $\alpha_1'+\alpha_2'$, so
we are done.
\end{proof}

\begin{problem}
(a) Show that for any finite group $G$ 
there exists a finite Galois extension $K\subset \Bbb C$ of $\Bbb Q$ 
such that any finite dimensional complex representation of $G$ 
has a basis in which the matrices of the group elements 
have entries in $K$. 

Hint. Consider the representations of $G$ over the field $\overline{\Bbb Q}$ 
of algebraic numbers. 

(b) Show that if $V$ is an irreducible complex representation of a
finite group $G$ of dimension $>1$ then there exists
$g\in G$ such that $\chi_V(g)=0$.

Hint: Assume the contrary. Use orthonormality of characters to
show that the arithmetic mean of the numbers $|\chi_V(g)|^2$ for
$g\ne 1$ is $<1$. Deduce that their product $\beta$ satisfies
$0<\beta <1$. Show that all conjugates of $\beta$ satisfy the
same inequalities (consider the Galois conjugates of the representation $V$, 
i.e. representations obtained from $V$ by the action of 
the Galois group of $K$ over $\Bbb Q$ on the matrices of group elements in the 
basis from part (a)). Then derive a contradiction.

{\bf Remark.} Here is a modification of this argument, 
which does not use (a). 
Let $N=|G|$. For any $0<j<N$ coprime to $N$, show that the map $g\mapsto g^j$ 
is a bijection $G\to G$. Deduce that $\prod_{g\ne 1} |\chi_V(g^j)|^2=\beta$.
Then show that $\beta\in K:=\Bbb Q(\zeta)$, $\zeta=e^{2\pi i/N}$, 
and does not change under the automorphism of $K$ 
given by $\zeta\mapsto \zeta^j$. Deduce that $\beta$ is an integer, 
and derive a contradiction.  
\end{problem}

\subsection{Frobenius divisibility}

\begin{theorem}
Let $G$ be a finite group, and let $V$ be an irreducible
representation of $G$ over $\BC$.
Then $$\dim V\text{ divides }\abs{G}.$$
\end{theorem}

\begin{proof}
Let $C_1, C_2, \dotsc, C_n$ be the conjugacy classes of $G$.
Set $$\lambda_i = \chi_V(g_{C_i})\frac{\abs{C_i}}{\dim V},$$
where $g_{C_i}$ is a representative of $C_i$.

\begin{proposition} \label{alint}
The numbers $\lambda_i$ are algebraic integers for all
$i$.
\end{proposition}

\begin{proof}
Let $C$ be a conjugacy class in $G$, and
$P=\sum_{h\in C}h$. Then $P$ is a central element of $\Bbb
Z[G]$, so it acts on $V$ by some scalar $\lambda$,
which is an algebraic integer (indeed, since $\Bbb Z[G]$ is a
finitely generated $\Bbb Z$-module, any element of $\Bbb Z[G]$
is integral over $\Bbb Z$, i.e., satisfies a monic polynomial
equation with integer coefficients).
On the other hand, taking the trace of $P$ in $V$, we
get $|C|\chi_V(g)=\lambda\dim V$, $g\in C$, so
$\lambda=\frac{|C|\chi_V(g)}{\dim V}$.
\end{proof}

Now, consider
$$
\sum_{i} \lambda_i \ov{\chi_V(g_{C_i})}.
$$
This is an algebraic integer, since:

(i) $\lambda_i$ are
algebraic integers by Proposition \ref{alint},

(ii) $\chi_V(g_{C_i})$ is a sum of roots of unity
(it is the sum of eigenvalues of the matrix of $\rho(g_{C_i})$,
and since $g_{C_i}^{\abs{G}}=e$ in $G$, the eigenvalues of
$\rho(g_{C_i})$ are roots of unity), and

(iii) $\BA$ is a ring
(Proposition \ref{A_ring}).

On the other hand, from
the definition of $\lambda_i$,\\
$$\sum_{C_i} \lambda_i \ov{\chi_V(g_{C_i})} =
\sum_{i} \frac{\abs{C_i} \chi_V(g_{C_i}) \ov{ \chi_V(g_{C_i}) } }{\dim V}.$$\\
Recalling that $\chi_V$ is a class function, this is equal to
$$\sum_{g \in G} \frac{ \chi_V(g) \ov{ \chi_V(g) } }{\dim V} =
\frac {\abs{G}(\chi_V, \chi_V)} {\dim V}.$$ Since $V$ is an
irreducible representation, $(\chi_V, \chi_V)=1,$ so $$\sum_{C_i} \lambda_i
\ov{\chi_V(g_{C_i})} = \frac {\abs{G}} {\dim V}.$$\\
Since $\frac{\abs{G}} {\dim V} \in \BQ$ and $\sum_{C_i}
\lambda_i \ov{\chi_V(g_{C_i})} \in \BA,$ by Proposition \ref{AcapQ_Z}
$\frac {\abs{G}}{\dim V} \in \BZ.$
\end{proof}

\subsection{Burnside's Theorem}

\begin{definition} \rm A group $G$ is called \emph{solvable} if there exists a series of nested normal subgroups
$$\{e\}=G_1 \nrm G_2 \nrm \ldots \nrm G_n=G$$ where $G_{i+1} / G_i$ is abelian
for all $1 \le i \le n-1$.
\end{definition}

\begin{remark} \rm
Such groups are called solvable because they first arose as Galois groups
of polynomial equations which are solvable in radicals.
\end{remark}

\begin{theorem}[Burnside]
Any group $G$ of order $p^a q^b$, where $p$ and $q$ are prime and $a,b \ge 0$, is solvable.
\end{theorem}

This famous result in group theory was proved by the British
mathematician William Burnside in the early 20-th century, using representation
theory (see \cite{Cu}). Here is this proof, presented in modern language.

Before proving Burnside's theorem we will prove several other results
which are of independent interest.

\begin{theorem}\label{thm:rep}
Let $V$ be an irreducible representation of a finite group $G$ and let  $C$
be a conjugacy class of $G$ with $\gcd(|C|, \dim(V))=1$.
Then for any $g \in C$, either $\chi_V(g)=0$ or $g$ acts as a scalar on $V$.
\end{theorem}

The proof will be based on the following lemma.

\begin{lemma}\label{lem:2}
If $\varepsilon_1, \varepsilon_2\ldots  \varepsilon_n$ are roots of unity such that
$\displaystyle \frac{1}{n}(\varepsilon_1+\varepsilon_2+\ldots
+\varepsilon_n)$ is an algebraic integer, then either
$\varepsilon_1=\ldots =\varepsilon_n$ or $\varepsilon_1+\ldots +\varepsilon_n=0$.
\end{lemma}

\begin{proof}
Let $a=\frac{1}{n}(\varepsilon_1+\ldots +\varepsilon_n)$.
If not all $\varepsilon_i$ are equal, then $|a|<1$.
Moreover, since any algebraic conjugate of a root of unity is
also a root of unity, $|a'|\le 1$ for any algebraic conjugate
$a'$ of $a$. But the product of all algebraic conjugates of $a$
is an integer. Since it has absolute value $<1$, it must equal
zero. Therefore, $a=0$.
\end{proof}

\emph{Proof of theorem \ref{thm:rep}.}

Let $\dim V=n$.
Let $\varepsilon_1, \varepsilon_2,\ldots \varepsilon_n$
be the eigenvalues of $\rho_V(g)$. They are roots of unity,
so $\chi_V (g)$ is an algebraic integer.
Also, by Proposition \ref{alint}, $\frac{1}{n}|C| \chi_V (g)$
is an algebraic integer. Since $\gcd(n, |C|)=1$, 
there exist integers $a,b$ such that $a|C|+bn=1$. This implies that
$$
\frac {\chi_V(g)}
{n}=\frac{1}{n}(\varepsilon_1+\ldots +\varepsilon_n).
$$
is an algebraic integer. Thus, by Lemma \ref{lem:2},
we get that either
$\varepsilon_1=\ldots =\varepsilon_n$ or $\varepsilon_1+\ldots +\varepsilon_n=
\chi_V(g)=0$. In the first case,
since $\rho_V(g)$ is diagonalizable, it must be scalar.
In the second case, $\chi_V (g)=0$. The theorem is proved.

\begin{theorem}\label{thm:2}
Let $G$ be a finite group, and let $C$
be a conjugacy class in $G$ of order $p^k$ where $p$ is prime and $k >0$.
Then $G$ has a proper nontrivial normal subgroup (i.e., $G$ is not simple). 
\end{theorem}

\begin{proof}
Choose an element $g \in C$.
Since $g\ne e$, by orthogonality of columns of the character table,
\begin{equation}\label{ort}
\sum_{V \in {\rm Irr}G} \dim V \chi_V(g)=0.
\end{equation}

We can divide
${\rm Irr}G$ into three parts:
\begin{enumerate}
\item
the trivial representation,
\item
$D$, the set of irreducible representations whose dimension is
divisible by $p$, and
\item
$N$, the set of non-trivial irreducible representations
whose dimension is not divisible by $p$.
\end{enumerate}

\begin{lemma}\label{lem:3}
There exists $V \in N$ such that $\chi_V(g)\ne 0$.
\end{lemma}

\begin{proof}
If $V \in D$, the number $\frac{1}{p} \dim(V) \chi_V(g)$ is
an algebraic integer, so
$$
a=\sum_{V\in D} \frac{1}{p}\dim (V) \chi_V(g)
$$
is an algebraic integer.

Now, by (\ref{ort}), we have
$$
0= \chi_{\Bbb C}(g)+\sum_{V \in D}\dim V \chi_V(g)+\sum_{V \in N} \dim V
\chi_V(g)=1+pa+\sum_{V \in N} \dim V
\chi_V(g).
$$
This means that the last summand is nonzero.
\end{proof}

Now pick $V\in N$ such that $\chi_V(g)\ne 0$; it exists by Lemma \ref{lem:3}.
Theorem \ref{thm:rep} implies that $g$ (and hence any element of $C$) acts
by a scalar in $V$. Now let $H$ be the subgroup of $G$ generated by elements
$ab^{-1}$, $a,b\in C$. It is normal and acts trivially in $V$,
so $H\ne G$, as $V$ is nontrivial.
Also $H\ne 1$, since $|C|>1$.
\end{proof}

\emph{Proof of Burnside's theorem.}

Assume Burnside's theorem is false. Then there exists a
nonsolvable group $G$ of order $p^aq^b$. Let $G$ be the smallest such group. 
Then $G$ is simple, and by Theorem \ref{thm:2}, it cannot have a
conjugacy class of order $p^k$ or $q^k$, $k\ge 1$.
So the order of any conjugacy class in $G$ is either $1$ or
is divisible by $pq$. Adding the orders of conjugacy classes and equating the
sum to $p^aq^b$, we see that there has to be more than one
conjugacy class consisting just of one element. So
$G$ has a nontrivial center, which gives a contradiction.

\subsection{Representations of products}

\begin{theorem}
Let $G,H$ be finite groups, $\lbrace{V_i\rbrace}$ be the
irreducible representations of $G$ over a field $k$ (of any characteristic),
and $\lbrace{W_j\rbrace}$ be the irreducible
representations of $H$ over $k$. Then the irreducible
representations of $G\times H$ over $k$ are $\lbrace{ V_i\otimes
W_j\rbrace}$.
\end{theorem}

\begin{proof}
This follows from Theorem \ref{produ}.
\end{proof}

\subsection{Virtual representations}

\begin{definition}
A {\it virtual representation}
of a finite group $G$ is an integer linear combination
of irreducible representations of $G$, $V=\sum n_iV_i$, $n_i\in
\Bbb Z$ (i.e., $n_i$ are not assumed to be nonnegative).
The character of $V$ is $\chi_V:=\sum n_i\chi_{V_i}$.
\end{definition}

The following lemma is often very useful (and will be used
several times below).

\begin{lemma}\label{virtu}
Let $V$ be a virtual representation
 with character $\chi_V$.
If $(\chi_V,\chi_V)=1$ and $\chi_V(1)>0$ then $\chi_V$ is a
character of an irreducible representation of $G$.
\end{lemma}

\begin{proof}
Let $V_1, V_2, \dotsc, V_m$
be the irreducible representations of $G$, and
$V=\sum n_iV_i$. Then by orthonormality of characters,
$(\chi_V,\chi_V)=\sum_i n_i^2$. So $\sum_i n_i^2=1$, meaning that
$n_i=\pm 1$ for exactly one $i$, and $n_j=0$ for $j\ne i$.
But $\chi_V(1)>0$, so $n_i=+1$ and we are done.
\end{proof}

\subsection{Induced Representations}

Given a representation $V$ of a group $G$ and a subgroup $H \subset G$, there is a natural way to
construct a representation of $H$. The restricted representation of $V$ to $H$, $\mathrm{Res}_H^GV$ is
the representation given by the vector space $V$ and the action
$\rho_{\mathrm{Res}_H^GV}=\rho_V|_H$.

There is also a natural, but more complicated way to construct a
representation of a group $G$ given a representation $V$ of its subgroup $H$.

\begin{definition} \rm  \label{dfn:ind}

If $G$ is a group, $H \subset G$, and $V$ is a representation of $H$, then the \emph{induced representation}
$Ind_H^GV$ is the representation of $G$ with
$${\rm Ind}_H^GV=\{f:G \to V|f(hx)=\rho_V(h)f(x) \forall x \in
G, h \in H\}
$$
and the action $g(f)(x)=f(xg) \space\ \forall g \in G$.
\end{definition}

\begin{remark}\rm
In fact, $\mathrm{Ind}_H^G V$ is naturally isomorphic to
$\mathrm{Hom}_H(k[G], V)$.
\end{remark}

Let us check that $\Ind_H^GV$ is indeed a representation:

$g(f)(hx)=f(hxg)=\rho_V(h)f(xg)=\rho_V(h)g(f)(x)$, and
$g(g'(f))(x)=g'(f)(xg)=f(xgg')=(gg')(f)(x)$
for any $g, g', x \in G$ and
$h \in H$.

\begin{remark}\label{natiso}\rm

Notice that if we choose a representative  $x_{\sigma}$ from
every right $H$-coset $\sigma$ of $G$, then any $f\in \mathrm{Ind}_H^GV$ is
uniquely determined by $\{f(x_{\sigma})\}$.

Because of this, $$\dim(\mathrm{Ind}_H^GV)=\dim V \cdot \frac{|G|}{|H|}.$$
\end{remark}

\begin{problem}
Check that if $K\subset H\subset G$ are groups
and $V$ a representation of $K$ then $\Ind_H^G\Ind_K^H V$ is
isomorphic to $\Ind_K^G V$.
\end{problem}

{\bf Exercise.} Let $K\subset G$ be finite groups,
and $\chi: K\to \Bbb C^*$ be a homomorphism. Let $\Bbb C_\chi$ be the 
corresponding 1-dimensional representation of $K$. Let 
$$
e_\chi=\frac{1}{|K|}\sum_{g\in K}\chi(g)^{-1}g\in \Bbb C[K]
$$
be the idempotent corresponding to $\chi$. 
Show that the $G$-representation 
${\rm Ind}_K^G\Bbb C_\chi$ is naturally isomorphic to 
${\Bbb C}[G]e_\chi$ (with $G$ acting by left multiplication). 

\subsection{The Mackey formula}

Let us now compute the character $\chi$ of $\mathrm{Ind}_H^GV$.
In each right coset $\sigma\in H\backslash G$, choose a representative 
$x_\sigma$. 

\begin{theorem}\label{indr} (The Mackey formula) One has
$$
\chi(g)=\sum_{\sigma \in H\backslash G: x_\sigma gx_\sigma^{-1} \in H}
\chi_V(x_\sigma g x_\sigma^{-1}).
$$
\end{theorem}

{\bf Remark.} If the characteristic of the ground field $k$ is relatively prime to $|H|$, 
then this formula can be written as 
$$
\chi(g)=\frac{1}{|H|}\sum_{x \in G: xgx^{-1} \in H}
\chi_V(xg x^{-1}).
$$

\begin{proof}
For a right $H$-coset $\sigma$
of $G$, let us define $$V_{\sigma}=\{f \in
\mathrm{Ind}_H^GV |f(g)=0\ \forall g \nin \sigma\}.$$
Then one has
$$\mathrm{Ind}_H^GV=\bigoplus_{\sigma} V_\sigma,$$
and so $$\chi(g)=\sum_\sigma \chi_{\sigma}(g),$$ where
$\chi_{\sigma}(g)$ is the trace of the diagonal block of
$\rho(g)$ corresponding to $V_\sigma$.

Since $g(\sigma)=\sigma g$ is a right $H$-coset for any
right $H$-coset $\sigma$, $\chi_{\sigma}(g)=0$ if $\sigma \ne \sigma
g$.

Now assume that $\sigma=\sigma g$. 
Then $x_\sigma g=hx_\sigma$ where $h=x_\sigma gx_\sigma^{-1} \in
H$. Consider the vector space homomorphism $\alpha: V_\sigma \to
V$ with $\alpha(f)=f(x_\sigma)$. Since $f \in V_\sigma$ is
uniquely determined by $f(x_\sigma)$, $\alpha$ is an isomorphism. We have
$$\alpha(gf)=g(f)(x_\sigma)=f(x_\sigma g)=f(h x_\sigma)=
\rho_V(h)f(x_\sigma)=h \alpha(f),$$
and $gf=\alpha^{-1}h\alpha(f)$. This means that
$\chi_{\sigma}(g)=\chi_V(h)$. Therefore
$$\chi(g)=\sum_{\sigma \in H \backslash G, \sigma g=\sigma} \chi_V(x_\sigma g
x_\sigma^{-1}).$$
\end{proof}

\subsection{Frobenius reciprocity}
A very important result about induced representations is the
Frobenius Reciprocity Theorem which connects the operations $\mathrm{Ind}$ and
$\mathrm{Res}$.

\begin{theorem}\label{Frobenius reciprocity}\label{frobrec}\rm (Frobenius Reciprocity)

{\rm Let $H\subset G$ be groups, $V$ be a representation of $G$ and $W$ a representation of $H$.}\ Then
$\mathrm{Hom}_G(V, \mathrm{Ind}_H^GW)$ is naturally isomorphic to $\mathrm{Hom}_H(\mathrm{Res}^G_HV, W)$.
\end{theorem}

\begin{proof}
Let $E=\mathrm{Hom}_G(V, \mathrm{Ind}_H^GW)$ and $E'=\mathrm{Hom}_H(\mathrm{Res}^G_HV, W)$. Define
$F:E \to E'$ and $F':E' \to E$ as follows:
$F(\alpha)v=(\alpha v)(e)$
for any $\alpha \in E$ and
$(F'(\beta)v)(x)=\beta(xv)$ for any $\beta \in E'$.

In order to check that $F$ and $F'$ are well defined and inverse to each
other, we need to check the following five statements.

Let $\alpha \in E$, $\beta \in E'$, $v \in V$, and $x, g \in G$.

(a) $F(\alpha)$ is an $H$-homomorphism, i.e., $F(\alpha)hv=hF(\alpha)v$.

Indeed, $F(\alpha)hv=(\alpha hv)(e)=(h\alpha v)(e)=(\alpha
v)(he)=(\alpha v)(eh)=h\cdot (\alpha v)(e)=hF(\alpha)v$.

(b) $F'(\beta)v \in \mathrm{Ind}^G_H W$,
i.e., $(F'(\beta)v)(hx)=h(F'(\beta)v)(x)$.

Indeed,
$(F'(\beta)v)(hx)=\beta(hxv)=h\beta(xv)=h(F'(\beta)v)(x)$.

(c) $F'(\beta)$ is a $G$-homomorphism, i.e.
$F'(\beta)gv=g(F'(\beta)v)$.

Indeed,
$(F'(\beta)gv)(x)=\beta(xgv)=(F'(\beta)v)(xg)=(g(F'(\beta)v))(x)$.

(d) $F \circ F'=Id_{E'}$.

This holds since $F(F'(\beta))v=(F'(\beta)v)(e)=\beta(v)$.

(e) $F' \circ F=Id_E$, i.e., $(F'(F(\alpha))v)(x)=(\alpha v)(x)$.

Indeed, $(F'(F(\alpha))v)(x)
=F(\alpha xv)=(\alpha xv)(e)=(x \alpha v)(e)=(\alpha v)(x)$,
and we are done.
\end{proof}

{\bf Exercise.} 
The purpose of this exercise is to understand the notions of
restricted and induced representations as part of a more advanced
framework. This framework is the notion of tensor products over
$k$-algebras (which generalizes the tensor product over $k$ which
we defined in Definition \ref{tenpro}). In particular, this understanding
will lead us to a new proof of the Frobenius reciprocity and to
some analogies between induction and restriction.

Throughout this exercise, we will use the notation and results of 
the Exercise in Section \ref{tenprod}. 

Let $G$ be a finite group and
$H\subset G$ a subgroup.  We consider $k\left[G\right]$ as a
$\left(k\left[H\right],k\left[G\right]\right)$-bimodule (both
module structures are given by multiplication inside
$k\left[G\right]$). We denote this bimodule by
$k\left[G\right]_1$. On the other hand, we can also consider
$k\left[G\right]$ as a
$\left(k\left[G\right],k\left[H\right]\right)$-bimodule (again,
both module structures are given by multiplication). We denote
this bimodule by $k\left[G\right]_2$.

(a) Let $V$ be a representation of $G$. Then, $V$ is a left
$k\left[G\right]$-module, thus a
$\left(k\left[G\right],k\right)$-bimodule. Thus, the tensor
product $k\left[G\right]_1\otimes_{k\left[G\right]} V$ is a
$\left(k\left[H\right],k\right)$-bimodule, i. e., a left
$k\left[H\right]$-module. Prove that this tensor product is
isomorphic to $\mathrm{Res}^G_H V$ as a left
$k\left[H\right]$-module. The isomorphism $\mathrm{Res}^G_H V\to
k\left[G\right]_1\otimes_{k\left[G\right]} V$ is given by
$v\mapsto 1\otimes_{k\left[G\right]}v$ for every $v\in
\mathrm{Res}^G_H V$.

(b) Let $W$ be a representation of $H$. Then, $W$ is a left
$k\left[H\right]$-module, thus a
$\left(k\left[H\right],k\right)$-bimodule. Then,
$\mathrm{Ind}^G_H W\cong
\mathrm{Hom}_H\left(k\left[G\right],W\right)$, according to
Remark \ref{natiso}. In other words, $\mathrm{Ind}^G_H W\cong
\mathrm{Hom}_{k\left[H\right]}\left(k\left[G\right]_1,W\right)$. Now,
use part (b) of the Exercise in Section \ref{tenprod} to conclude Theorem \ref{frobrec}.

(c) Let $V$ be a representation of $G$. Then, $V$ is a left
$k\left[G\right]$-module, thus a
$\left(k\left[G\right],k\right)$-bimodule. Prove that not only
$k\left[G\right]_1\otimes_{k\left[G\right]} V$, but also
$\mathrm{Hom}_{k\left[G\right]}\left(k\left[G\right]_2,V\right)$
is isomorphic to $\mathrm{Res}^G_H V$ as a left
$k\left[H\right]$-module. The isomorphism
$\mathrm{Hom}_{k\left[G\right]}\left(k\left[G\right]_2,V\right)\to
\mathrm{Res}^G_H V$ is given by $f\mapsto f\left(1\right)$ for
every $f\in
\mathrm{Hom}_{k\left[G\right]}\left(k\left[G\right]_2,V\right)$.

(d) Let $W$ be a representation of $H$. Then, $W$ is a left
$k\left[H\right]$-module, thus a
$\left(k\left[H\right],k\right)$-bimodule. Show that
$\mathrm{Ind}^G_H W$ is not only isomorphic to
$\mathrm{Hom}_{k\left[H\right]}\left(k\left[G\right]_1,W\right)$,
but also isomorphic to \linebreak
$k\left[G\right]_2\otimes_{k\left[H\right]}W$. The isomorphism
$\mathrm{Hom}_{k\left[H\right]}\left(k\left[G\right]_1,W\right)\to
k\left[G\right]_2\otimes_{k\left[H\right]}W$ is given by
$f\mapsto \sum_{g\in
P}g^{-1}\otimes_{k\left[H\right]}f\left(g\right)$ for every
$f\in\mathrm{Hom}_{k\left[H\right]}\left(k\left[G\right]_1,W\right)$,
where $P$ is a set of distinct representatives for the right
$H$-cosets in $G$. (This isomorphism is independent of the choice
of representatives.)

(e) Let $V$ be a representation of $G$ and $W$ a representation
of $H$. Use (b) to prove that
$\mathrm{Hom}_G\left(\mathrm{Ind}^G_H W,V\right)$ is naturally
isomorphic to $\mathrm{Hom}_H\left(W,\mathrm{Res}^G_H
V\right)$. 

(f) Let $V$ be a representation of $H$. Prove that
$\mathrm{Ind}^G_H \left(V^{\ast}\right) \cong
\left(\mathrm{Ind}^G_H V\right)^{\ast}$ as representations of
$G$. [Hint: Write $\mathrm{Ind}^G_H V$ as
$k\left[G\right]_2\otimes_{k\left[H\right]}V$ and write
$\mathrm{Ind}^G_H \left(V^{\ast}\right)$ as
$\mathrm{Hom}_{k\left[H\right]}\left(k\left[G\right]_1,V^{\ast}\right)$. Prove
that the map
$\mathrm{Hom}_{k\left[H\right]}\left(k\left[G\right]_1,V^{\ast}\right)\times
\left(\mathrm{Ind}^G_H \left(V^{\ast}\right)\right) \to k$ given
by $\left(f, \left(x\otimes_{k\left[H\right]} v\right)\right)\mapsto
\left(f\left(Sx\right)\right)\left(v\right)$ is a nondegenerate
$G$-invariant bilinear form, where $S:k\left[G\right]\to
k\left[G\right]$ is the linear map defined by
$Sg=g^{-1}$ for every $g\in G$.]

\subsection{Examples}

Here are some examples of induced representations
(we use the notation for representations from the character tables).

\begin{enumerate}
\item Let $G=S_3$, $H=\Bbb Z_2$. Using the Frobenius reciprocity,
we obtain: ${\rm Ind}_H^G\Bbb C_+=\Bbb C^2\oplus \Bbb C_+$,
${\rm Ind}_H^G\Bbb C_-=\Bbb C^2\oplus \Bbb C_-$.

\item Let $G=S_3$, $H=\Bbb Z_3$. Then we obtain
${\rm Ind}_H^G\Bbb C_+=\Bbb C_+\oplus \Bbb C_-$,
${\rm Ind}_H^G\Bbb C_\epsilon={\rm Ind}_H^G\Bbb
C_{\epsilon^2}=\Bbb C^2$.

\item Let $G=S_4$, $H=S_3$. Then ${\rm Ind}_H^G\Bbb C_+=
\Bbb C_+\oplus \Bbb C_-^3$,  ${\rm Ind}_H^G\Bbb C_-=
\Bbb C_-\oplus \Bbb C_+^3$, ${\rm Ind}_H^G\Bbb C^2=
\Bbb C^2\oplus \Bbb C_-^3\oplus \Bbb C^3_+$.

\end{enumerate}

\begin{problem}
Compute the decomposition into irreducibles
of all the representations of $A_5$ induced from
all the irreducible representations of 

(a) $\Bbb Z_2$

(b) $\Bbb Z_3$

(c) $\Bbb Z_5$

(d) $A_4$

(e) $\Bbb Z_2\times \Bbb Z_2$

\end{problem}

\subsection{Representations of $S_n$}

In this subsection we give a description of the representations
of the symmetric group $S_n$ for any $n$.

\begin{definition} A {\bf partition}
$\lambda$ of $n$ is a representation of $n$
in the form $n=\lambda_1+\lambda_2+...+\lambda_p$,
where $\lambda_i$ are positive integers, and $\lambda_i\ge \lambda_{i+1}$.
\end{definition}

To such $\lambda$ we will attach a {\bf Young diagram} $Y_\lambda$,
which is the union of rectangles $-i\le y\le -i+1$,
$0\le x\le \lambda_i$ in the coordinate plane, for $i=1,...,p$.
Clearly, $Y_\lambda$ is a collection of $n$ unit squares.
A {\bf Young tableau} corresponding to $Y_\lambda$ is the result of
filling the numbers $1,...,n$ into the squares of $Y_\lambda$ in
some way (without repetitions). For example, we will consider
the Young tableau $T_\lambda$ obtained by filling in the numbers
in the increasing order, left to right, top to bottom.

We can define two subgroups of $S_n$ corresponding to
$T_\lambda$:

1. The row subgroup $P_\lambda$: the subgroup which maps every
element of $\{1,...,n\}$ into an element standing in the same row
in $T_\lambda$.

2. The column subgroup $Q_\lambda$: the subgroup which maps every
element of $\{1,...,n\}$ into an element standing in the same column
in $T_\lambda$.

Clearly, $P_\lambda\cap Q_\lambda=\{1\}$.

Define the {\it Young projectors:}
$$
a_\lambda:=\frac{1}{|P_\lambda|}\sum_{g\in P_\lambda}g,
$$
$$
b_\lambda:=\frac{1}{|Q_\lambda|}\sum_{g\in Q_\lambda}(-1)^gg,
$$
where $(-1)^g$ denotes the sign of the permutation $g$.
Set $c_\lambda=a_\lambda b_\lambda$. Since $P_\lambda\cap
Q_\lambda=\lbrace{ 1\rbrace}$, this element is nonzero.

The irreducible representations of $S_n$ are described by the
following theorem.

\begin{theorem}\label{symgroup}
The subspace $V_\lambda:=\Bbb C[S_n]c_\lambda$ of $\Bbb
C[S_n]$ is an irreducible representation of $S_n$ under
left multiplication. Every irreducible
representation of $S_n$ is isomorphic to $V_\lambda$
for a unique $\lambda$.
\end{theorem}

The modules $V_\lambda$ are called the {\bf Specht modules}.

The proof of this theorem is given in the next subsection.

\begin{example}

\item For the partition $\lambda=(n)$, $P_\lambda=S_n$,
$Q_\lambda=\{1\}$, so $c_\lambda$ is the symmetrizer, and hence
$V_\lambda$ is the trivial representation.

\item For the partition $\lambda=(1,...,1)$, $Q_\lambda=S_n$,
$P_\lambda=\{1\}$, so $c_\lambda$ is the antisymmetrizer, and hence
$V_\lambda$ is the sign representation.

\item $n=3$. For $\lambda=(2,1)$, $V_\lambda=\Bbb C^2$.

\item $n=4$. For $\lambda=(2,2)$, $V_\lambda=\Bbb C^2$;
for $\lambda=(3,1)$, $V_\lambda=\Bbb C_-^3$;
for $\lambda=(2,1,1)$, $V_\lambda=\Bbb C_+^3$.
\end{example}

\begin{corollary} All irreducible representations of $S_n$ can
be given by matrices with rational entries.
\end{corollary}

\begin{problem}\label{3:7}
Find the sum of dimensions of all irreducible representations of the
symmetric group $S_n$.

Hint. Show that all irreducible representations of $S_n$ are real, i.e., admit
a nondegenerate invariant symmetric form. Then use the Frobenius-Schur
theorem.
\end{problem}

\subsection{Proof of Theorem \ref{symgroup}}

\begin{lemma}\label{le1}
Let $x\in \Bbb C[S_n]$. Then $a_\lambda
xb_\lambda=\ell_\lambda(x)c_\lambda$,
where $\ell_\lambda$ is a linear function.
\end{lemma}

\begin{proof}
If $g\in P_\lambda Q_\lambda$, then $g$ has a unique representation as
$pq$, $p\in P_\lambda, q\in Q_\lambda$, so $a_\lambda
gb_\lambda=(-1)^qc_\lambda$. Thus, to prove the required statement, we need to
show that if $g$ is a permutation which is not in $P_\lambda Q_\lambda$ then $a_\lambda
gb_\lambda=0$.

To show this, it is sufficient to find a transposition $t$ such
that $t\in P_\lambda$ and $g^{-1}tg\in Q_\lambda$;
then
$$
a_\lambda gb_\lambda=a_\lambda tgb_\lambda=a_\lambda g
(g^{-1}tg)b_\lambda=-a_\lambda g b_\lambda,
$$
so $a_\lambda g b_\lambda=0$. In other words, we have to find
two elements $i,j$ standing in the same row in the tableau $T=T_\lambda$,
and in the same column in the tableau $T'=gT$ (where $gT$ is the
tableau of the same shape as $T$ obtained by permuting the
entries of $T$ by the permutation $g$). Thus, it suffices to
show that if such a pair does not exist, then $g\in P_\lambda
Q_\lambda$, i.e., there exists $p\in P_\lambda$, $q'\in Q_\lambda':=gQ_\lambda
g^{-1}$ such that $pT=q'T'$ (so that $g=pq^{-1}, q=g^{-1}q'g\in Q_\lambda$).

Any two elements in the first row of $T$
must be in different columns of $T'$, so there exists
$q_1'\in Q_\lambda'$ which moves all these elements to the first
row. So there is $p_1\in P_\lambda$ such
that $p_1T$ and $q_1'T'$ have the same first row.
Now do the same procedure with the second row, finding
elements $p_2,q_2'$ such that $p_2p_1T$ and $q_2'q_1'T'$ have the
same first two rows. Continuing so, we will construct the desired
elements $p,q'$. The lemma is proved.
\end{proof}

Let us introduce the {\bf lexicographic ordering} on partitions:
$\lambda>\mu$ if the first nonvanishing $\lambda_i-\mu_i$ is
positive.

\begin{lemma}\label{le2}
If $\lambda>\mu$ then $a_\lambda \Bbb C[S_n]b_\mu=0$.
\end{lemma}

\begin{proof} Similarly to the previous lemma, it suffices to show
that for any $g\in S_n$ there exists a transposition
$t\in P_\lambda$ such that $g^{-1}tg\in Q_\mu$.
Let $T=T_\lambda$ and $T'=gT_\mu$. We claim that there are two
integers which are in the same row of $T$ and the same
column of $T'$. Indeed, if $\lambda_1>\mu_1$, this is clear by
the pigeonhole principle (already for the first row). Otherwise,
if $\lambda_1=\mu_1$, like in the proof of the previous lemma,
we can find elements $p_1\in P_\lambda, q_1'\in gQ_\mu g^{-1}$
such that $p_1T$ and $q_1'T'$ have the same first row,
and repeat the argument for the second row, and so on.
Eventually, having done $i-1$ such steps,
we'll have $\lambda_i>\mu_i$, which means that some two elements
of the $i$-th row of the first tableau are in the same column of
the second tableau, completing the proof.
\end{proof}

\begin{lemma}\label{le3} $c_\lambda$ is proportional to an
idempotent. Namely, $c_\lambda^2=\frac{n!}{\dim
V_\lambda}c_\lambda$.
\end{lemma}

\begin{proof} Lemma \ref{le1} implies that $c_\lambda^2$ is
proportional to $c_\lambda$. Also,
it is easy to see that the trace of $c_\lambda$ in the regular
representation is $n!$ (as the coefficient of the identity element
in $c_\lambda$ is $1$). This implies the statement.
\end{proof}

\begin{lemma}\label{le4}
Let $A$ be an algebra and $e$ be an idempotent in $A$.
Then for any left $A$-module $M$, one has
$\Hom_A(Ae,M)\cong eM$ (namely, $x\in eM$ corresponds to $f_x: Ae\to M$ 
given by $f_x(a)=ax$, $a\in Ae$).
\end{lemma}

\begin{proof} Note that $1-e$ is also an idempotent in $A$.
Thus the statement immediately follows from the fact that
$\Hom_A(A,M)\cong M$ and the decomposition
$A=Ae\oplus A(1-e)$.
\end{proof}

Now we are ready to prove Theorem \ref{symgroup}.
Let $\lambda\ge \mu$. Then by Lemmas \ref{le3}, \ref{le4}
$$
\Hom_{S_n}(V_\lambda,V_\mu)=\Hom_{S_n}(\Bbb C[S_n]c_\lambda,\Bbb
C[S_n]c_\mu)=c_\lambda\Bbb C[S_n]c_\mu.
$$
The latter space is zero for $\lambda>\mu$ by Lemma \ref{le2},
and 1-dimensional if $\lambda=\mu$ by Lemmas \ref{le1} and
\ref{le3}. Therefore, $V_\lambda$ are irreducible, and
$V_\lambda$ is not isomorphic to $V_\mu$ if $\lambda\ne \mu$.
Since the number of partitions equals the number of conjugacy
classes in $S_n$, the representations $V_\lambda$ exhaust all the
irreducible representations of $S_n$. The theorem is proved.

\subsection{Induced representations for $S_n$}

Denote by $U_\lambda$ the representation
${\rm Ind}_{P_\lambda}^{S_n}\Bbb C$. It is easy to see that
$U_\lambda$ can be alternatively defined as $U_\lambda=\Bbb
C[S_n]a_\lambda$.

\begin{proposition}\label{multip}
$\Hom(U_\lambda,V_\mu)=0$ for $\mu<\lambda$, and
$\dim \Hom(U_\lambda,V_\lambda)=1$. Thus,
$U_\lambda=\oplus_{\mu\ge \lambda}K_{\mu\lambda}V_\mu$,
where $K_{\mu\lambda}$ are nonnegative integers
and $K_{\lambda \lambda}=1$.
\end{proposition}

\begin{definition} The integers $K_{\mu\lambda}$ are called the
{\bf Kostka numbers}.
\end{definition}

\begin{proof} By Lemmas \ref{le3} and \ref{le4},
$$
\Hom(U_\lambda,V_\mu)=\Hom(\CC[S_n]a_\lambda,\CC[S_n]a_\mu b_\mu)=
a_\lambda \CC[S_n]a_\mu b_\mu,
$$
and the result follows from Lemmas \ref{le1} and \ref{le2}.
\end{proof}

Now let us compute the character of $U_\lambda$.
Let $C_{\bold i}$ be the conjugacy class in $S_n$
having $i_l$ cycles of length $l$ for all $l\ge 1$
(here $\bold i$ is a shorthand notation for $(i_1,...,i_l,...)$).
Also let $x_1,...,x_N$ be variables, and let
$$
H_m(x)=\sum_i x_i^m
$$
be the power sum polynomials.

\begin{theorem}\label{charU} Let $N\ge p$ (where $p$ is the number of
parts of $\lambda$). Then $\chi_{U_\lambda}(C_{\bold i})$
is the coefficient\footnote{If $j>p$, we define $\lambda_j$ to be zero.}
of $x^\lambda:=\prod x_j^{\lambda_j}$
in the polynomial
$$
\prod_{m\ge 1}H_m(x)^{i_m}.
$$
\end{theorem}

\begin{proof}
The proof is obtained easily from the Mackey formula.
Namely, $\chi_{U_\lambda}(C_{\bold i})$ is the number
of elements $x\in S_n$ such that $xgx^{-1}\in P_\lambda$
(for a representative $g\in C_{\bold i}$), divided by $|P_\lambda|$. 
The order of $P_\lambda$ is $\prod_i \lambda_i!$, and 
the number of elements $x$ such that $xgx^{-1}\in P_\lambda$ 
is the number of elements in $P_\lambda$ conjugate to $g$ (i.e. $|C_{\bold i}\cap P_\lambda|$)
times the order of the centralizer $Z_g$ of $g$ (which is $n!/|C_{\bold i}|$). 
Thus,
$$
\chi_{U_\lambda}(C_{\bold i})=\frac{|Z_g|}{\prod_j \lambda_j!}
|C_{\bold i}\cap P_\lambda |.
$$
Now, it is easy to see that the centralizer $Z_g$ of $g$ is isomorphic to 
$\prod_m S_{i_m}\ltimes (\Bbb Z/m\Bbb Z)^{i_m}$, so 
$$
|Z_g|=\prod_m m^{i_m} i_m!,
$$
and we get
$$
\chi_{U_\lambda}(C_{\bold i})=\frac{\prod_m m^{i_m} i_m!}
{\prod_j \lambda_j!}
|C_{\bold i}\cap P_\lambda |.
$$
Now, since $P_\lambda=\prod_j S_{\lambda_j}$,
we have
$$
|C_{\bold i}\cap P_\lambda |=
\sum_r \prod_{j\ge 1} \frac{\lambda_j!}{\prod_{m\ge 1} m^{r_{jm}}
r_{jm}!},
$$
where $r=(r_{jm})$ runs over all collections of nonnegative
integers such that
$$
\sum_m mr_{jm}=\lambda_j,
\sum_j r_{jm}=i_m.
$$
Indeed, an element of $C_{\bold i}$ that is in 
$P_\lambda$ would define an ordered 
partition of each $\lambda_j$ into parts (namely, cycle lengths), 
with $m$ occuring $r_{jm}$ times, such that the total (over all $j$)
number of times each part $m$ occurs is $i_m$. 
Thus we get
$$
\chi_{U_\lambda}(C_{\bold i})
=\sum_r \prod_m \frac{i_m!}{\prod_j r_{jm}!}
$$
But this is exactly the coefficient of
$x^{\lambda}$ in
$$
\prod_{m\ge 1}(x_1^m+...+x_N^m)^{i_m}
$$
($r_{jm}$ is the number of times we take $x_j^m$).
\end{proof}

\subsection{The Frobenius character formula}

Let $\Delta(x)=\prod_{1\le i<j\le N}(x_i-x_j)$.
Let $\rho=(N-1,N-2,...,0)\in \Bbb C^N$.
The following theorem, due to Frobenius, gives a character
formula for the Specht modules $V_\lambda$.

\begin{theorem}\label{charV} Let $N\ge p$. Then
$\chi_{V_\lambda}(C_{\bold i})$
is the coefficient of $x^{\lambda+\rho}:=\prod x_j^{\lambda_j+N-j}$
in the polynomial
$$
\Delta(x)\prod_{m\ge 1}H_m(x)^{i_m}.
$$
\end{theorem}

{\bf Remark.} Here is an equivalent formulation of Theorem \ref{charV}:
$\chi_{V_\lambda}(C_{\bold i})$
is the coefficient of $x^\lambda$
in the (Laurent) polynomial
$$
\prod_{i<j}\biggl(1-\frac{x_j}{x_i}\biggr)\prod_{m\ge 1}H_m(x)^{i_m}.
$$

\begin{proof} Denote $\chi_{V_\lambda}$ shortly by
$\chi_\lambda$.
Let us denote the class function defined in the theorem by
$\theta_\lambda$. We claim that this function has the property $\theta_\lambda=\sum_{\mu\ge
\lambda}L_{\mu\lambda}\chi_\mu$, where $L_{\mu\lambda}$ are
integers and $L_{\lambda\lambda}=1$. Indeed, 
from Theorem \ref{charU} we have 
$$
\theta_\lambda=\sum_{\sigma\in S_N}(-1)^\sigma \chi_{U_{\lambda+\rho-\sigma(\rho)}},
$$
where if the vector $\lambda+\rho-\sigma(\rho)$ has a negative entry, 
the corresponding term is dropped, and if it has nonnegative entries which fail to be nonincreasing, 
then the entries should be reordered in the nonincreasing order, making a partition that we'll denote 
$\langle \lambda+\rho-\sigma(\rho)\rangle$ (i.e., we agree that $U_{\lambda+\rho-\sigma(\rho)}:=
U_{\langle \lambda+\rho-\sigma(\rho)\rangle}$). Now note that $\mu=\langle \lambda+\rho-\sigma(\rho)\rangle$ is obtained 
from $\lambda$ by adding vectors of the form $e_i-e_j$, $i<j$, which implies that $\mu>\lambda$ or $\mu=\lambda$, and the case 
$\mu=\lambda$ arises only if $\sigma=1$, as desired.  

Therefore, to show that
$\theta_\lambda=\chi_\lambda$, by Lemma \ref{virtu}, it suffices to show that
$(\theta_\lambda,\theta_\lambda)=1$.

We have
$$
(\theta_\lambda,\theta_\lambda)=\frac{1}{n!}\sum_{\bold i}|C_{\bold
i}|\theta_\lambda(C_{\bold i})^2.
$$
Using that 
$$
|C_{\bold i}|=\frac{n!}{\prod_m m^{i_m}i_m!},
$$
we conclude that $(\theta_\lambda,\theta_\lambda)$ 
is the coefficient of $x^{\lambda+\rho}y^{\lambda+\rho}$
in the series $R(x,y)=\Delta(x)\Delta(y)S(x,y)$, where
$$
S(x,y)=
\sum_{\bold i}\prod_m \frac{(\sum_{j}
x_j^m)^{i_m}(\sum_k y_k^m)^{i_m}}{m^{i_m}i_m!}=
\sum_{\bold i}\prod_m \frac{(\sum_{j,k}
x_j^my_k^m/m)^{i_m}}{i_m!}.
$$
Summing over $\bold i$ and $m$, we get
$$
S(x,y)=\prod_m \exp(\sum_{j,k}
x_j^my_k^m/m)=\exp(-\sum_{j,k}\log(1-x_jy_k))=\prod_{j,k}(1-x_jy_k)^{-1}
$$
Thus,
$$
R(x,y)=\frac{\prod_{i<j}(x_i-x_j)(y_i-y_j)}{\prod_{i,j}(1-x_iy_j)}.
$$
Now we need the following lemma.

\begin{lemma}\label{dete}
$$
\frac{\prod_{i<j}(z_j-z_i)(y_i-y_j)}{\prod_{i,j}(z_i-y_j)}=
\det(\frac{1}{z_i-y_j}).
$$
\end{lemma}

\begin{proof} Multiply both sides by
$\prod_{i,j}(z_i-y_j)$. Then the right hand side must vanish
on the hyperplanes $z_i=z_j$ and $y_i=y_j$ (i.e., be divisible by
$\Delta(z)\Delta(y)$), and is a homogeneous polynomial
of degree $N(N-1)$. This implies that the right hand side and the left
hand side are proportional. The proportionality coefficient
(which is equal to $1$) is found by induction by multiplying both
sides by $z_N-y_N$ and then setting $z_N=y_N$.
\end{proof}

Now setting in the lemma $z_i=1/x_i$, we get

\begin{corollary}\label{formR} (Cauchy identity)
$$
R(x,y)=\det(\frac{1}{1-x_iy_j})=\sum_{\sigma\in S_N}\frac{1}{\prod_{j=1}^N(1-x_jy_{\sigma(j)})}.
$$
\end{corollary}

Corollary \ref{formR} easily implies that
the coefficient of $x^{\lambda+\rho}y^{\lambda+\rho}$ is $1$.
Indeed, if $\sigma\ne 1$ is a permutation in $S_N$, the coefficient of this
monomial in $\frac{1}{\prod (1-x_jy_{\sigma(j)})}$ is obviously
zero.
\end{proof}

{\bf Remark.} For partitions $\lambda$ and $\mu$ of $n$, 
let us say that $\lambda\preceq \mu$ or $\mu\succeq \lambda$ if $\mu-\lambda$ 
is a sum of vectors of the form $e_i-e_j$, $i<j$ (called positive roots).
This is a partial order, and $\mu\succeq \lambda$ 
implies $\mu\ge \lambda$. It follows from Theorem \ref{charV} 
and its proof that 
$$
\chi_\lambda=\oplus_{\mu\succeq \lambda}\widetilde{K}_{\mu\lambda}\chi_{U_\mu}.
$$
This implies that the Kostka numbers $K_{\mu\lambda}$ vanish unless 
$\mu\succeq \lambda$. 

\subsection{Problems}

In the following problems, we do not make a distinction between
Young diagrams and partitions.

\begin{problem}
For a Young diagram $\mu$, let $A(\mu)$ be the set of Young
diagrams obtained by adding a square to $\mu$, and $R(\mu)$
be the set of Young diagrams obtained by removing a square from
$\mu$.

(a) Show that ${\rm Res}_{S_{n-1}}^{S_n} V_\mu=\oplus_{\lambda\in
R(\mu)}V_\lambda$.
\end{problem}

(b) Show that $\Ind_{S_{n-1}}^{S_n} V_\mu=\oplus_{\lambda\in
A(\mu)}V_\lambda$.

\begin{problem}\label{cont}
The content $c(\lambda)$ of a Young diagram $\lambda$ is the sum
$\sum_j\sum_{i=1}^{\lambda_j} (i-j)$.
Let $C=\sum_{i<j}(ij)\in \CC[S_n]$ be the sum of all transpositions.
Show that $C$ acts on the Specht module $V_\lambda$ by
multiplication by $c(\lambda)$.
\end{problem}

\begin{problem}
(a) Let $V$ be any finite dimensional representation of $S_n$. 
Show that the element $E:=(12)+...+(1n)$ is diagonalizable and 
has integer eigenvalues on $V$, which are between $1-n$ and $n-1$. 

Hint. Represent $E$ as $C_n-C_{n-1}$, where $C_n=C$ is the element from
Problem \ref{cont}. 

(b) Show that the element $(12)+...+(1n)$ acts on $V_\lambda$ by a
scalar if and only if $\lambda$ is a rectangular Young diagram,
and compute this scalar.
\end{problem}

\subsection{The hook length formula}

Let us use the Frobenius character formula to compute the
dimension of $V_\lambda$. According to the character formula,
$\dim V_\lambda$ is the coefficient of $x^{\lambda+\rho}$ in
$\Delta(x)(x_1+...+x_N)^n$. Let $l_j=\lambda_j+N-j$. Then, 
using the determinant formula for $\Delta(x)$ and expanding the determinant 
as a sum over permutations, we get
$$
\dim V_\lambda=\sum_{s\in S_N: l_j\ge N-s(j)}(-1)^s \frac{n!}
{\prod_j (l_j-N+s(j))!}=\frac{n!}{l_1!...l_N!}
\sum_{s\in S_N} (-1)^s \prod_j l_j(l_j-1)...(l_j-N+s(j)+1)=
$$
$$
\frac{n!}{\prod_j l_j!}
\det (l_j(l_j-1)...(l_j-N+i+1)).
$$
Using column reduction and the Vandermonde determinant formula,
we see from this expression that
\begin{equation}
\label{prodfo}
\dim V_\lambda=\frac{n!}{\prod_j l_j!}\det(l_j^{N-i})=
\frac{n!}{\prod_j l_j!}\prod_{1\le i<j\le N}(l_i-l_j)
\end{equation}
(where $N\ge p$).

In this formula, there are many cancelations. After making some
of these cancelations, we obtain the hook length formula.
Namely, for a square $(i,j)$ in a Young diagram $\lambda$ ($i,j\ge
1,i\le \lambda_j$), define the hook of $(i,j)$ to be the set
of all squares $(i',j')$ in $\lambda$ with $i'\ge i$, $j'=j$ or
$i'=i$, $j'\ge j$. Let $h(i,j)$ be the length of the hook of
$i,j$, i.e., the number of squares in it.

\begin{theorem} (The hook length formula)
One has
$$
\dim V_\lambda=\frac{n!}{\prod_{i\le \lambda_j}h(i,j)}.
$$
\end{theorem}

\begin{proof}
The formula follows from formula (\ref{prodfo}).
Namely, note that
$$
\frac{l_1!}{\prod_{1<j\le N}(l_1-l_j)}=\prod_{1\le k\le l_1, k\ne
l_1-l_j}k.
$$
It is easy to see that the factors in this product are exactly the hooklengths
$h(i,1)$. Now delete the first row of the diagram and proceed by induction.
\end{proof}

\subsection{Schur-Weyl duality for ${\mathfrak{gl}}(V)$}

We start with a simple result which is called the Double
Centralizer Theorem.

\begin{theorem}\label{dcl}
Let $A$, $B$ be two subalgebras of the algebra $\End E$ of
endomorphisms of a finite dimensional vector space $E$, such that
$A$ is semisimple, and $B=\End_A E$. Then:

(i) $A=\End_B E$ (i.e., the centralizer of the centralizer of $A$
is $A$);

(ii) $B$ is semisimple;

(iii) as a representation of $A\otimes B$,
$E$ decomposes as $E=\oplus_{i\in I}V_i\otimes W_i$, where $V_i$
are all the irreducible representations of $A$, and $W_i$
are all the irreducible representations of $B$.
In particular, we have a natural bijection between irreducible
representations of $A$ and $B$.
\end{theorem}

\begin{proof}
Since $A$ is semisimple, we have a natural decomposition
$E=\oplus_{i\in I}V_i\otimes W_i$, where $W_i:=\Hom_A(V_i,E)$,
and $A=\oplus_i \End V_i$. Therefore, by Schur's lemma,
$B=\End_A(E)$ is naturally
identified with $\oplus_i \End(W_i)$. This implies all the
statements of the theorem.
\end{proof}

We will now apply Theorem \ref{dcl} to the following situation:
$E=V^{\otimes n}$, where $V$ is a finite dimensional vector
space over a field of characteristic zero, and $A$ is the image of
$\Bbb C[S_n]$ in $\End E$. Let us now characterize the algebra
$B$. Let ${\mathfrak{gl}}(V)$ be $\End V$ regarded as a Lie algebra with
operation $ab-ba$.

\begin{theorem}
The algebra $B=\End_A E$ is the image of the universal enveloping
algebra ${\mathfrak{U}}({\mathfrak{gl}}(V))$ under its natural action on $E$. In other
words, $B$ is generated by elements of the form
$$
\Delta_n(b):=b\otimes 1\otimes...\otimes 1+1\otimes b\otimes...\otimes 1+...+
1\otimes 1\otimes...\otimes b,
$$
$b\in {\mathfrak{gl}}(V)$.
\end{theorem}

\begin{proof}
Clearly, the image of ${\mathfrak{U}}({\mathfrak{gl}}(V))$ is contained in $B$, so we just
need to show that any element of $B$ is contained in the image of
${\mathfrak{U}}({\mathfrak{gl}}(V))$. By definition, $B=S^n\End V$, so the result follows
from part (ii) of the following lemma.

\begin{lemma}\label{symfu} Let $k$ be a field of characteristic
zero.

(i) For any finite dimensional vector space $U$ over $k$, the
space $S^nU$ is spanned by elements of the form
$u\otimes...\otimes u$, $u\in U$.

(ii) For any algebra $A$ over $k$,
the algebra $S^nA$ is generated by elements $\Delta_n(a)$,
$a\in A$.
\end{lemma}

\begin{proof}
(i) The space $S^nU$ is an irreducible representation of $GL(U)$
(Problem \ref{3:1}). The subspace spanned by $u\otimes...\otimes
u$ is a nonzero subrepresentation, so it must be everything.

(ii) By the fundamental theorem on symmetric functions,
there exists a polynomial $P$ with rational coefficients such
that $P(H_1(x),...,H_n(x))=x_1...x_n$ (where $x=(x_1,...,x_n)$). Then
$$
P(\Delta_n(a),\Delta_n(a^2),...,\Delta_n(a^n))=a\otimes...\otimes
a.
$$
The rest follows from (i).
\end{proof}

\end{proof}

Now, the algebra $A$ is semisimple
by Maschke's theorem, so the double centralizer theorem
applies, and we get the following result, which goes under the
name ``Schur-Weyl duality''.

\begin{theorem}
(i) The image $A$ of $\Bbb C[S_n]$ and the image $B$ of ${\mathfrak{U}}({\mathfrak{gl}}(V))$
in $\End(V^{\otimes n})$ are centralizers of each other.

(ii) Both $A$ and $B$ are semisimple. In particular,
$V^{\otimes n}$ is a semisimple ${\mathfrak{gl}}(V)$-module.

(iii) We have a decomposition of $A\otimes B$-modules
$V^{\otimes n}=\oplus_{\lambda}V_\lambda\otimes L_\lambda$, where
the summation is taken over partitions of $n$,
$V_\lambda$ are Specht modules for $S_n$, and $L_\lambda$ are
some distinct irreducible representations of ${\mathfrak{gl}}(V)$ or zero.
\end{theorem}

\subsection{Schur-Weyl duality for $GL(V)$}

The Schur-Weyl duality for the Lie algebra ${\mathfrak{gl}}(V)$ implies a
similar statement for the group $GL(V)$.

\begin{proposition}
The image of $GL(V)$ in $\End(V^{\otimes n})$ spans $B$.
\end{proposition}

\begin{proof}
Denote the span of $g^{\otimes n}$, $g\in
GL(V)$, by $B'$. Let $b\in \End V$ be any element.

We claim that $B'$ contains $b^{\otimes n}$.
Indeed, for all values of $t$ but finitely many,
$t\cdot \Id+b$ is invertible, so $(t\cdot \Id+b)^{\otimes n}$
belongs to $B'$. This implies that this is true for all $t$,
in particular for $t=0$,
since $(t\cdot \Id+b)^{\otimes n}$ is a polynomial in $t$.

The rest follows from Lemma \ref{symfu}.
\end{proof}

\begin{corollary} As a representation of $S_n\times GL(V)$,
$V^{\otimes n}$ decomposes as $\oplus_\lambda V_\lambda\otimes
L_\lambda$, where $L_\lambda=\Hom_{S_n}(V_\lambda,V^{\otimes n})$
are distinct irreducible representations of
$GL(V)$ or zero.
\end{corollary}

\begin{example}
If $\lambda=(n)$ then $L_\lambda=S^nV$, and if $\lambda=(1^n)$
($n$ copies of $1$) then $L_\lambda=\wedge^n V$.
It was shown in Problem \ref{3:1} that these representations
are indeed irreducible (except that $\wedge^n V$ is zero if
$n>\dim V$).
\end{example}

\subsection{Schur polynomials}

Let $\lambda=(\lambda_1,...,\lambda_p)$ be a partition of $n$,
and $N\ge p$. Let
$$
D_\lambda(x)=\sum_{s\in S_N}(-1)^s\prod_{j=1}^N
x_{s(j)}^{\lambda_j+N-j}=\det(x_i^{\lambda_j+N-j}).
$$
Define the polynomials
$$
S_\lambda(x):=\frac{D_\lambda(x)}{D_0(x)}
$$
(clearly $D_0(x)$ is just $\Delta(x)$). It is easy to see that
these are indeed polynomials, as $D_\lambda$ is antisymmetric
and therefore must be divisible by $\Delta$. The polynomials
$S_\lambda$ are called the {\bf Schur polynomials}.

\begin{proposition}\label{Schur}
$$
\prod_m (x_1^m+...+x_N^m)^{i_m}=\sum_{\lambda: p\le N}
\chi_\lambda(C_{\bold i})S_\lambda(x).
$$
\end{proposition}

\begin{proof}
The identity follows from the Frobenius character formula
and the antisymmetry of
$$
\Delta(x)\prod_m
(x_1^m+...+x_N^m)^{i_m}.
$$
\end{proof}

Certain special values of Schur polynomials are of importance.
Namely, we have

\begin{proposition}\label{eval}
$$
S_\lambda(1,z,z^2,...,z^{N-1})=\prod_{1\le i<j\le N}
\frac{z^{\lambda_i-i}-z^{\lambda_j-j}}{z^{-i}-z^{-j}}
$$
Therefore,
$$
S_\lambda(1,...,1)=\prod_{1\le i<j\le N}
\frac{\lambda_i-\lambda_j+j-i}{j-i}
$$
\end{proposition}

\begin{proof}
The first identity is obtained from the definition
using the Vandermonde determinant. The second identity
follows from the first one by setting $z=1$.
\end{proof}

\subsection{The characters of $L_\lambda$}

Proposition \ref{Schur} allows us to calculate the characters
of the representations $L_\lambda$.

Namely, let $\dim V=N$, $g\in GL(V)$, and
$x_1,...,x_N$ be the eigenvalues of $g$ on $V$.
To compute the character $\chi_{L_\lambda}(g)$, let us calculate
${\rm Tr}_{V^{\otimes n}}(g^{\otimes n}s)$, where $s\in
S_n$. If $s\in C_{\bold i}$, we easily get that this trace equals
$$
\prod_m {\rm Tr}(g^m)^{i_m}=\prod_m H_m(x)^{i_m}.
$$
On the other hand, by the Schur-Weyl duality
$$
{\rm Tr}_{V^{\otimes n}}(g^{\otimes n}s)=\sum_\lambda \chi_\lambda(C_{\bold i}){\rm Tr}_{L_\lambda}(g).
$$
Comparing this to Proposition \ref{Schur} and using linear
independence of columns of the character table of $S_n$, we
obtain

\begin{theorem} (Weyl character formula)
The representation $L_\lambda$ is zero if and only if $N<p$,
where $p$ is the number of parts of $\lambda$. If $N\ge p$,
the character of $L_\lambda$ is the Schur polynomial
$S_\lambda(x)$. Therefore, the dimension of $L_\lambda$ is given
by the formula
$$
\dim L_\lambda=\prod_{1\le i<j\le N}\frac{\lambda_i-\lambda_j+j-i}{j-i}
$$
\end{theorem}

This shows that irreducible representations of $GL(V)$ which occur in $V^{\otimes
n}$ for some $n$ are labeled by Young diagrams with any number of
squares but at most $N=\dim V$ rows.

\begin{proposition}\label{adddet}
The representation $L_{\lambda+1^N}$ (where $1^N=(1,1,...,1)\in
\Bbb Z^N$) is isomorphic
to $L_\lambda\otimes \wedge^NV$.
\end{proposition}

\begin{proof} Indeed,
$L_\lambda\otimes \wedge^N V\subset V^{\otimes n}\otimes \wedge^N
V\subset V^{\otimes n+N}$, and the only component of $V^{\otimes n+N}$
that has the same character as $L_\lambda\otimes \wedge^NV$ is
$L_{\lambda+1^N}$. This implies the statement.
\end{proof}

\subsection{Polynomial representations of $GL(V)$}

\begin{definition}
We say that a finite dimensional representation $Y$ of $GL(V)$ is
{\bf polynomial} (or {\bf algebraic}, or {\bf rational})
if its matrix elements are polynomial functions
of the entries of $g,g^{-1}$, $g\in GL(V)$ (i.e., belong to
$k[g_{ij}][1/\det (g)]$).
\end{definition}

For example, $V^{\otimes n}$ and hence
all $L_\lambda$ are polynomial. Also define $L_{\lambda-r\cdot
1^N}:=L_\lambda\otimes (\wedge^NV^*)^{\otimes r}$
(this definition makes sense by Proposition \ref{adddet}).
This is also a polynomial representation.
Thus we have attached a unique irreducible polynomial
representation $L_\lambda$ of $GL(V)=GL_N$ to any sequence
$(\lambda_1,...,\lambda_N)$ of integers (not necessarily
positive) such that $\lambda_1\ge...\ge \lambda_N$.
This sequence is called the {\bf highest weight} of $L_\lambda$.

\begin{theorem}
(i) Every finite dimensional polynomial representation of $GL(V)$ is
completely reducible, and decomposes into summands of the form
$L_\lambda$ (which are pairwise non-isomorphic).

(ii) (the Peter-Weyl theorem for $GL(V)$). Let $R$ be the algebra of polynomial functions on $GL(V)$. 
Then as a representation of $GL(V)\times GL(V)$ (with action $(\rho(g,h)\phi)(x)=\phi(g^{-1}xh)$, 
$g,h,x\in GL(V)$, $\phi\in R$), $R$ decomposes as 
$$
R=\oplus_\lambda L_\lambda^*\otimes L_\lambda,
$$
where the summation runs over all $\lambda$. 
\end{theorem}

\begin{proof}
(i) Let $Y$ be a polynomial representation of $GL(V)$.
We have an
embedding $\xi: Y\to Y\otimes R$ given by $(u,\xi(v))(g):=u(gv)$, $u\in V^*$.
It is easy to see that $\xi$ is a homomorphism of representations
(where the action of $GL(V)$ on the first component of
$Y\otimes R$ is trivial). Thus, it suffices to prove the theorem
for a subrepresentation $Y\subset R^m$. Now, every element of $R$
is a polynomial of $g_{ij}$ times a nonpositive power of
$\det(g)$. Thus, $R$ is a quotient of a direct sum of representations of the form
$S^r(V\otimes V^*)\otimes (\wedge^NV^*)^{\otimes s}$.
So we may assume that
$Y$ is contained in a quotient of a (finite) direct sum of such
representations. As $V^*=\wedge^{N-1}V\otimes
\wedge^NV^*$, $Y$ is contained in a direct sum of representations
of the form $V^{\otimes n}\otimes (\wedge^NV^*)^{\otimes s}$,
and we are done.

(ii) Let $Y$ be a polynomial representation of $GL(V)$, and 
let us regard $R$ as a representation of $GL(V)$ via $(\rho(h)\phi)(x)=\phi(xh)$. 
Then ${\rm Hom}_{GL(V)}(Y,R)$ is the space of polynomial functions on $GL(V)$ 
with values in $Y^*$, which are $GL(V)$-equivariant. 
This space is naturally identified with $Y^*$. Taking into account the proof of (i), we deduce that 
$R$ has the required decomposition, which is compatible with the second action of $GL(V)$
(by left multiplications). This implies the statement. 
\end{proof}

Note that the Peter-Weyl theorem generalizes Maschke's theorem for finite group, one 
of whose forms states that the space of complex functions ${\rm Fun}(G,\Bbb C)$ 
on a finite group $G$ as a representation of $G\times G$ decomposes as $\oplus_{V\in {\rm Irrep(G)}}V^*\otimes V$. 

\begin{remark}
Since the Lie algebra ${\mathfrak{sl}}(V)$ of traceless
operators on $V$ is a quotient of ${\mathfrak{gl}}(V)$ by scalars, 
the above results extend in a straightforward manner
to representations of the Lie algebra ${\mathfrak{sl}}(V)$. Similarly, 
the results for $GL(V)$ extend to the case of the group
$SL(V)$ of operators with determinant $1$. The only difference is that in this case
the representations $L_\lambda$ and $L_{\lambda+1^m}$ are
isomorphic, so the irreducible representations are parametrized
by integer sequences $\lambda_1\ge...\ge \lambda_N$ up to a simultaneous
shift by a constant.

In particular, one can show that any finite dimensional representation of ${\mathfrak{sl}}(V)$
is completely reducible, and any irreducible one is of the form
$L_\lambda$ (we will not do this here). For $\dim V=2$ one then recovers the
representation theory of ${\mathfrak{sl}}(2)$ studied in Problem \ref{4:1}.
\end{remark}

\subsection{Problems}

\begin{problem}
(a) Show that the $S_n$-representation $V_\lambda':=\Bbb C[S_n]b_\lambda a_\lambda$
is isomorphic to $V_\lambda$.

Hint. Define $S_n$-homomorphisms 
$f: V_\lambda\to V_\lambda'$ and $g: V_\lambda'\to V_\lambda$ 
by the formulas $f(x)=xa_\lambda$ and $g(y)=yb_\lambda$, and show
that they are inverse to each other up to a nonzero scalar. 

(b) Let $\phi: \CC[S_n]\to \CC[S_n]$ be the automorphism sending
$s$ to $(-1)^ss$ for any permutation $s$. Show that $\phi$ maps
any representation $V$ of $S_n$ to $V\otimes \CC_-$.
Show also that $\phi(\CC[S_n]a)=\CC[S_n]\phi(a)$, for $a\in
\CC[S_n]$. Use (a) to deduce that $V_\lambda\otimes
\CC_-=V_{\lambda^*}$, where $\lambda^*$ is the conjugate
partition to $\lambda$, obtained by reflecting the Young diagram
of $\lambda$.
\end{problem}

\begin{problem}
Let $R_{k,N}$ be the algebra of polynomials
on the space of $k$-tuples of complex $N$ by $N$ matrices $X_1,...,X_k$,
invariant under simultaneous conjugation. An example of an
element of $R_{k,N}$ is the function $T_w:={\rm Tr} (w(X_1,...,X_k))$,
where $w$ is any finite word on a $k$-letter alphabet. Show that
$R_{k,N}$ is generated by the elements $T_w$.

Hint. Consider invariant functions that are of degree $d_i$ in each $X_i$, 
and realize this space as a tensor product $\otimes_i S^{d_i}(V\otimes V^*)$. 
Then embed this tensor product into $(V\otimes V^*)^{\otimes N}=\End(V)^{\otimes n}$, and
use the Schur-Weyl duality to get the result. 
\end{problem}

\subsection {Representations of $GL_2(\mathbb F_q)$}

\subsubsection{Conjugacy classes in $GL_2(\mathbb F_q)$}
Let $\mathbb F_q$ be a finite field of size $q$ of characteristic
other than $2$, and 
$G=GL_2(\mathbb F_q)$.
 Then $$\abs{G}=(q^2-1)(q^2-q),$$
since the first column of an invertible 2 by 2 matrix
 must be non-zero and the second column may
not be a multiple of the first one. Factoring, $$\abs{GL_2(\mathbb
F_q)}=q(q+1)(q-1)^2.$$ The goal of this section is to describe the
irreducible representations of $G$.\\
To begin, let us find the conjugacy classes in $GL_2(\mathbb
F_q)$.
$$
\begin{array} {l|l|l}
\parbox{5cm} {Representatives} & \parbox{6cm} {Number of elements in a conjugacy
class} & \parbox{4cm} {Number of classes} \\ &&\\ \hline
\parbox{5cm} {Scalar $\bigl( \begin{smallmatrix} x&0\\ 0&x
\end{smallmatrix} \bigr)$} &
\parbox{6cm} {$1$ (this is a central element)} &
\parbox{4cm} {$q-1$ (one for every non-zero $x$)} \\
&&\\
\parbox{5cm} {Parabolic $\bigl(\begin{smallmatrix} x&1\\ 0&x
\end{smallmatrix} \bigr)$} &
\parbox{6cm} {$q^2-1$ (elements that commute with this one are of the
form $\bigl( \begin{smallmatrix} t&u\\ 0&t \end{smallmatrix}
\bigr), \; t \ne 0$)} &
\parbox{4cm} {$q-1$ (one for every non-zero $x$)} \\
&&\\
\parbox{5cm} {Hyperbolic $\bigl( \begin{smallmatrix} x&0\\ 0&y
\end{smallmatrix} \bigr), \; y \ne x$} &
\parbox{6cm} {$q^2+q$ (elements that commute with this one
are of the form $\bigl( \begin{smallmatrix} t&0\\ 0&u
\end{smallmatrix} \bigr), \; t,u \ne 0$)} &
\parbox{4cm} {$\frac{1}{2} (q-1)
(q-2)$ ($x,y \ne 0$ and $x \ne y$)} \\
&&\\
\parbox{5cm} {Elliptic $\bigl( \begin{smallmatrix} x&\eps y\\ y&x
\end{smallmatrix} \bigr), x \in \mathbb F_q, \; y \in \mathbb
F_q^\times, \; \eps \in \mathbb F_q \setminus \mathbb
F_q^2$ (characteristic polynomial over $\mathbb F_q$ is
irreducible)} &
\parbox{6cm} {$q^2-q$ (the reason will be described below)} &
\parbox{4cm} {$\frac{1}{2} q (q-1)$ (matrices with $y$ and $-y$ are
conjugate)}\\
\end{array}
$$

More on the conjugacy class of elliptic matrices: these are the
matrices whose characteristic polynomial is irreducible over
$\mathbb{F}_q$ and which therefore don't have eigenvalues in
$\Bbb F_q$. Let $A$
be such a matrix, and consider a quadratic extension of
$\mathbb{F}_q$, $$\mathbb{F}_q(\sqrt{\eps}), \eps \in
\mathbb F_q \setminus \mathbb F_q^2.$$ Over this field, $A$ will
have eigenvalues $$\alpha=\alpha_1 + \sqrt{\eps}\alpha_2$$ and
$$\ov{\alpha}=\alpha_1 - \sqrt{\eps}\alpha_2,$$ with
corresponding eigenvectors $$v, \; \ov{v} \quad (Av = \alpha v, \;
A\ov{v}=\ov{\alpha}\ov{v}).$$\\
Choose a basis $$\lbrace e_1=v+\ov{v}, \;
e_2=\sqrt{\eps}(v-\ov{v}) \rbrace.$$ In this basis, the matrix
A will have the form $$\begin{pmatrix} \alpha_1 &
\eps\alpha_2\\ \alpha_2&\alpha_1 \end{pmatrix},$$ justifying
the description of representative elements of
this conjugacy class.\\
In the basis $\lbrace v, \ov{v} \rbrace$, matrices that commute
with $A$ will have the form $$\begin{pmatrix} \lambda&0\\
0&\ov{\lambda} \end{pmatrix},$$ for all $$\lambda \in \mathbb
F_{q^2}^{\times},$$
so the number of such matrices is $q^2-1$.

\subsubsection{1-dimensional representations}
First, we describe the 1-dimensional representations of $G$.
\begin{proposition}
$[G,G]=SL_2(\mathbb F_q)$.
\end{proposition}
\begin{proof}
Clearly, $$\det(xyx^{-1} y^{-1})=1,$$ so $$[G,G] \subseteq SL_2(\mathbb
F_q).$$ To show the converse, it suffices to show that
the matrices
$$\begin{pmatrix}
1&1\\0&1
\end{pmatrix},\begin{pmatrix}
a&0\\0&a^{-1}
\end{pmatrix},\begin{pmatrix}
1&0\\1&1
\end{pmatrix}$$
are commutators (as such matrices generate
$SL_2(\mathbb F_q).$)
Clearly, by using transposition, it suffices to show that only the first
two matrices are commutators.
But it is easy to see that the matrix $$\begin{pmatrix} 1&1\\0&1
\end{pmatrix}$$ is the commutator of
the matrices
$$
A=
\begin{pmatrix}
1 & 1/2\\
0 & 1
\end{pmatrix}
, \;  B=
\begin{pmatrix}
1 & 0\\
0 & -1
\end{pmatrix},
$$
while the matrix $$\begin{pmatrix} a&0\\0&a^{-1}
\end{pmatrix}$$ is the commutator of
the matrices
$$
A=
\begin{pmatrix}
a & 0\\
0 & 1
\end{pmatrix}
, \; B=
\begin{pmatrix}
0 & 1\\
1 & 0
\end{pmatrix},
$$
This completes the proof.
\end{proof}

Therefore,
$$G/[G,G] \cong \mathbb F_q^{\times} \text{ via } g \rightarrow
\det(g).$$ The one-dimensional representations of $G$ thus have
the form $$\rho(g) = \xi \bigl(\det(g) \bigr),$$ where $\xi$ is a
homomorphism $$\xi: \mathbb F_q^{\times} \rightarrow
\CC^{\times};$$ so there are
$q-1$ such representations, denoted $\CC_{\xi}.$

\subsubsection{Principal series representations}
Let $$B \subset G, \; B =
\{\begin{pmatrix}*&*\\0&*\end{pmatrix}\}$$ (the set of upper
triangular matrices); then $$\abs{B}=(q-1)^2 q,$$
$$[B,B]=U=\{\begin{pmatrix}1&*\\0&1\end{pmatrix}\},$$ and $$B/[B,B]
\cong \mathbb F_q^{\times} \times \mathbb F_q^{\times}$$ (the
isomorphism maps an element of $B$ to its two
diagonal entries).\\
Let $$\lambda: B \rightarrow \CC^\times$$ be a homomorphism defined by
$$\lambda
\begin{pmatrix}a&b\\0&c\end{pmatrix}=\lambda_1(a)\lambda_2(c)\text{,for
some pair of homomorphisms }\lambda_1, \lambda_2: \mathbb F_q^\times
\rightarrow \CC^\times.$$ Define $$V_{\lambda_1, \lambda_2} =
\Ind_B^G
\CC_\lambda,$$ where $\CC_\lambda$ is the 1-dimensional
representation of $B$ in which $B$ acts by $\lambda.$ We have
$$\dim(V_{\lambda_1, \lambda_2}) = \frac{\abs{G}}{\abs{B}} = q+1.$$

\begin{theorem}
\begin{enumerate}
\item $\lambda_1 \ne \lambda_2 \Rightarrow V_{\lambda_1,
\lambda_2}$ is irreducible. \item $\lambda_1 = \lambda_2 = \mu
\Rightarrow V_{\lambda_1, \lambda_2} = \CC_\mu \oplus W_\mu,$
where $W_\mu$ is a $q$-dimensional irreducible representation of
$G$.
\item $W_\mu \cong W_\nu$ if and only if $\mu = \nu;$ $V_{ \lambda_1,
\lambda_2 } \cong V_{\lambda_1', \lambda_2'}$ if and only if $\{ \lambda_1,
\lambda_2 \} = \{ \lambda_1^{'}, \lambda_2^{'} \}$ (in the second
case, $\lambda_1 \ne \lambda_2, \lambda_1^{'} \ne \lambda_2^{'}$).
\end{enumerate}
\end{theorem}

\begin{proof}
 From the Mackey formula, we have
$$tr_{V_{\lambda_1, \lambda_2}}(g) = \frac{1}{\abs{B}} \sum_{a \in G, \, aga^{-1} \in B}
\lambda(aga^{-1}).$$ If $$g=\begin{pmatrix} x&0\\0&x
\end{pmatrix},$$ the expression on the right evaluates to
$$\lambda(g) \frac{\abs{G}}{\abs{B}} = \lambda_1(x) \lambda_2(x) \bigl(q+1\bigr).$$
If $$g=\begin{pmatrix} x&1\\0&x
\end{pmatrix},$$ the expression evaluates to
$$\lambda(g) \cdot 1,$$ since here $$aga^{-1} \in B \Rightarrow a \in
B.$$ If $$g=\begin{pmatrix} x&0\\0&y \end{pmatrix},$$ the
expression evaluates to $$\bigl(
\lambda_1(x)\lambda_2(y)+\lambda_1(y)\lambda_2(x) \bigr) \cdot
1,$$ since here $$aga^{-1} \in B \Rightarrow a \in B \text{ or $a$
is an element of $B$ multiplied by the transposition matrix.}$$
If $$g=\begin{pmatrix}
x&\eps y\\y&x \end{pmatrix},\ x\ne y$$ the expression on the right
evaluates to 0 because matrices of this type don't have
eigenvalues over $\mathbb F_q$
(and thus cannot be conjugated into $B$).
 From the definition, $\lambda_i(x) (i=1,2)$ is a root of unity, so

\begin{multline*}
\abs{G} \langle \chi_{V_{\lambda_1, \lambda_2}},
\chi_{V_{\lambda_1,
\lambda_2}} \rangle = (q+1)^2(q-1)+(q^2-1)(q-1)\\
+2(q^2+q)\frac{(q-1)(q-2)}{2}+(q^2+q)\sum_{x
\ne y}\lambda_1(x)\lambda_2(y) \ov{\lambda_1(y)\lambda_2(x)}.
\end{multline*}

The last two summands come from the expansion $$\abs{a+b}^2 =
\abs{a}^2 + \abs{b}^2 + a \ov{b} + \ov{a} b.$$ If $$\lambda_1 =
\lambda_2 = \mu,$$ the last term is equal to
$$(q^2+q)(q-2)(q-1),$$ and the total in this case is
$$
(q+1)(q-1)[(q+1)+(q-1)+2q(q-2)] = (q+1)(q-1)2q(q-1) = 2\abs{G},
$$
so
$$\langle \chi_{V_{\lambda_1, \lambda_2}}, \chi_{V_{\lambda_1, \lambda_2}} \rangle = 2.$$
Clearly, $$\CC_\mu \subseteq \Ind_B^G \CC_{\mu,\mu},$$ since
$$
\Hom_G(\CC_\mu, \Ind_B^G \CC_{\mu,\mu}) = \Hom_B(\CC_\mu, \CC_\mu)
= \CC \; \text{(Theorem \ref{Frobenius reciprocity}).}
$$
Therefore, $\Ind_B^G \CC_{\mu,\mu} = \CC_{\mu} \oplus W_\mu;$
$W_\mu$ is irreducible; and the character of $W_\mu$ is different
for distinct values of $\mu,$ proving that $W_\mu$ are distinct.

If $\lambda_1 \ne \lambda_2,$ let $z=xy^{-1},$ then the last term
of the summation is
$$
(q^2+q) \sum_{x \ne y} \lambda_1(z) \ov{\lambda_2(z)} = (q^2+q)
\sum_{x; z\ne1} \frac{\lambda_1}{\lambda_2}(z) = (q^2+q)(q-1)
\sum_{z \ne 1} \frac{\lambda_1}{\lambda_2}(z).
$$
Since
$$
\sum_{z \in \mathbb F_q^\times} \frac{\lambda_1}{\lambda_2}(z) = 0,
$$
because the sum of all roots of unity of a given order $m>1$ is
zero, the last term becomes
$$
-(q^2+q)(q-1) \sum_{z \ne 1} \frac{\lambda_1}{\lambda_2}(1) =
-(q^2+q)(q-1).
$$
The difference between this case and the case of $\lambda_1 =
\lambda_2$ is equal to
$$
-(q^2+q)[(q-2)(q-1)+(q-1)] = \abs{G},
$$
so this is an irreducible representation by Lemma \ref{virtu}.

To prove the third assertion of the theorem, we look at the
characters on hyperbolic elements and note that the function
$$
\lambda_1(x)\lambda_2(y) + \lambda_1(y)\lambda_2(x)
$$
determines $\lambda_1, \lambda_2$ up to permutation.

\end{proof}

\subsubsection{Complementary series representations}
Let $\mathbb F_{q^2} \supset \mathbb F_q$ be a quadratic extension
$\mathbb F_q(\sqrt{\eps}), \eps \in \mathbb F_q \setminus \mathbb
F_q^2.$ We regard this as a 2-dimensional vector space over
$\mathbb F_q;$ then $G$ is the group of linear
transformations of $\mathbb F_{q^2}$ over $\mathbb F_q.$ Let $K
\subset G$ be the cyclic group of multiplications
by elements of $\mathbb F_{q^2}^\times,$
$$
K=\{ \begin{pmatrix} x&\eps y\\y&x \end{pmatrix} \}, \;
\abs{K}=q^2-1.
$$
For $\nu: K \rightarrow \CC^{\times}$ a homomorphism, let
$$
Y_\nu = \Ind_K^G {\Bbb C}_\nu.
$$
This representation, of course, is very reducible. Let us compute
its character, using the Mackey formula. We get
$$
\chi \begin{pmatrix} x&0\\0&x \end{pmatrix} = q(q-1) \nu(x);
$$
$$
\chi(A)=0 \text{ for $A$ parabolic or hyperbolic;}
$$
$$
\chi \begin{pmatrix} x&\eps y\\y&x \end{pmatrix} =
\nu\begin{pmatrix} x&\eps y\\y&x \end{pmatrix} +
\nu\begin{pmatrix} x&\eps y\\y&x \end{pmatrix}^q.
$$
The last assertion holds because if we regard the matrix as an element
of $\mathbb F_{q^2}$, conjugation is an automorphism of $\mathbb
F_{q^2}$
over $\mathbb F_{q}$,
but the only nontrivial automorphism of $\mathbb F_{q^2}$ over $\mathbb F_q$ is
the $q^\text{th}$ power map.

We thus have
$$
\Ind_K^G {\Bbb C}_{\nu^q} \cong \Ind_K^G {\Bbb C}_\nu
$$
because they have the same character. Therefore, for $\nu^q \ne \nu$ we get
$\frac{1}{2}q(q-1)$ representations.

Next, we look at the following tensor product:
$$
W_1 \otimes V_{\alpha, 1},
$$
where $1$ is the trivial character and $W_1$ is defined as in the previous section.
The character of this representation is
$$
\chi \begin{pmatrix} x&0\\0&x \end{pmatrix} = q(q+1) \alpha(x);
$$
$$
\chi(A)=0 \text{ for $A$ parabolic or elliptic;}
$$
$$
\chi \begin{pmatrix} x&0\\0&y \end{pmatrix} =
\alpha(x) + \alpha(y).
$$

Thus the "virtual representation"
$$
W_{1} \otimes V_{\alpha, 1} - V_{\alpha, 1} - \Ind_K^G {\Bbb C}_\nu,
$$
where $\alpha$ is the restriction of $\nu$ to scalars, has character
$$
\chi \begin{pmatrix} x&0\\0&x \end{pmatrix} = (q-1) \alpha(x);
$$
$$
\chi \begin{pmatrix} x&1\\0&x \end{pmatrix} = -\alpha(x);
$$
$$
\chi \begin{pmatrix} x&0\\0&y \end{pmatrix} = 0;
$$
$$
\chi \begin{pmatrix} x&\eps y\\y&x \end{pmatrix} =
-\nu \begin{pmatrix} x&\eps y\\y&x \end{pmatrix}
-\nu^q \begin{pmatrix} x&\eps y\\y&x \end{pmatrix}.
$$

In all that follows, we will have $\nu^q \ne \nu.$

The following two lemmas will establish that the inner product of this character with
itself is equal to 1, that its value at 1 is positive.
As we know from Lemma \ref{virtu}, these two properties
imply that it is the character of an irreducible representation of $G$.

\begin{lemma}
Let $\chi$ be the character of the "virtual representation" defined above. Then
$$\langle \chi, \chi \rangle = 1$$ and $$\chi(1)>0.$$
\end{lemma}

\begin{proof}
$$\chi(1)=q(q+1)-(q+1)-q(q-1)=q-1>0.$$
We now compute the inner product $\langle \chi, \chi \rangle.$
Since $\alpha$ is a root of unity, this will be equal to
$$
\frac{1}{(q-1)^2q(q+1)} \bigl[(q-1)\cdot(q-1)^2\cdot 1 + (q-1)\cdot1\cdot(q^2-1) +
\frac{q(q-1)}{2}\cdot\sum_{\zeta\text{ elliptic}} (\nu(\zeta)+\nu^q(\zeta))
\ov{(\nu(\zeta)+\nu^q(\zeta))} \bigr]
$$
Because $\nu$ is also a root of unity, the last term of the expression evaluates to
$$
\sum_{\zeta\text{ elliptic}}(2+ \nu^{q-1}(\zeta) + \nu^{1-q}(\zeta)).$$
Let's evaluate the last summand.

Since $\mathbb F_{q^2}^{\times}$ is cyclic and $\nu^q \ne \nu$,
$$
\sum_{\zeta \in \mathbb F_{q^2}^{\times}} \nu^{q-1}(\zeta) =
\sum_{\zeta \in \mathbb F_{q^2}^{\times}} \nu^{1-q}(\zeta) = 0.
$$
Therefore,
$$
\sum_{\zeta\text{ elliptic}} (\nu^{q-1}(\zeta) + \nu^{1-q}(\zeta)) =
 - \sum_{\zeta \in \mathbb F_q^{\times}} (\nu^{q-1}(\zeta) + \nu^{1-q}(\zeta)) =
 - 2(q-1) =
$$
since $\mathbb F_q^{\times}$ is cyclic of order $q-1.$
Therefore,
$$
\langle \chi, \chi \rangle = \frac{1}{(q-1)^2q(q+1)} \bigl((q-1)\cdot(q-1)^2\cdot 1
+ (q-1)\cdot1\cdot(q^2-1) + \frac{q(q-1)}{2}\cdot(2(q^2-q)-2(q-1)) \bigr) = 1.
$$
\end{proof}

We have now shown that for any $\nu$ with $\nu^q \ne \nu$ the representation
$Y_\nu$ with the same character as
$$
W_{1} \otimes V_{\alpha, 1} - V_{\alpha, 1} - \Ind_K^G {\Bbb C}_\nu
$$
exists and is irreducible. These characters are distinct for distinct
pairs $(\alpha,\;\nu)$
(up to switch $\nu\to \nu^q$), so there are
$\frac{q(q-1)}{2}$ such representations, each of dimension $q-1$.

We have thus found $q-1$ 1-dimensional representations of $G,$ $\frac{q(q-1)}{2}$
principal series representations, and $\frac{q(q-1)}{2}$ complementary series
representations, for a total of $q^2-1$ representations, i.e., the number of conjugacy
classes in $G.$
This implies that we have in fact found all irreducible representations
of $GL_2(\Bbb F_q)$.

\subsection{Artin's theorem}

\begin{theorem}\label{art}
Let $X$ be a conjugation-invariant system of subgroups of a finite group $G$.
Then two conditions are equivalent:

(i) Any element of $G$ belongs to a subgroup $H\in X$.

(ii) The character of any irreducible representation of $G$
belongs to the $\Bbb Q$-span of characters of induced
representations $\Ind_H^GV$, where $H\in X$ and $V$ is an
irreducible representation of $H$.
\end{theorem}

{\bf Remark.} Statement (ii) of Theorem \ref{art} is equivalent
to the same statement with $\Bbb Q$-span replaced by $\Bbb
C$-span. Indeed, consider the matrix whose columns consist of  
the coefficients of the decomposition of $\Ind_H^GV$ 
(for various $H,V$) with respect to the irreducible
representations of $G$. Then both statements are equivalent to
the condition that the rows of this matrix are linearly
independent. 

\begin{proof} Proof that (ii) implies (i). Assume that $g\in G$
does not belong to any of the subgroups $H\in X$. Then, since $X$ is
conjugation invariant, it cannot be conjugated into such a
subgroup. Hence by the Mackey formula, $\chi_{\Ind_H^G(V)}(g)=0$
for all $H\in X$ and $V$. So by (ii), for any irreducible
representation $W$ of $G$, $\chi_W(g)=0$. But irreducible
characters span the space of class functions, so any class
function vanishes on $g$, which is a contradiction.

Proof that (i) implies (ii).
Let $U$ be a virtual representation of $G$ over $\Bbb C$ (i.e., a linear
combination of irreducible representations with nonzero integer
coefficients) such that $(\chi_U,\chi_{\Ind_H^GV})=0$ for all $H,V$.
So by Frobenius reciprocity, $(\chi_{U|_H},\chi_V)=0$.
This means that $\chi_U$ vanishes on $H$ for any $H\in X$. Hence
by (i), $\chi_U$ is identically zero. This implies (ii) (because
of the above remark).
\end{proof}

\begin{corollary}
Any irreducible character of a finite group is a rational linear
combination of induced characters from its cyclic subgroups.
\end{corollary}

\subsection{Representations of semidirect products}

Let $G,A$ be groups and $\phi: G\to {\rm Aut}(A)$
be a homomorphism. For $a\in A$, denote $\phi(g)a$ by $g(a)$.
The semidirect product $G\ltimes A$ is defined to be the product
$A\times G$ with multiplication law
$$
(a_1,g_1)(a_2,g_2)=(a_1g_1(a_2),g_1g_2).
$$
Clearly, $G$ and $A$ are subgroups of $G\ltimes A$ in a natural
way.

We would like to study irreducible complex representations of
$G\ltimes A$. For simplicity, let us do it when $A$ is abelian.

In this case, irreducible representations of $A$ are
1-dimensional and form the character group $A^\vee$, which
carries an action of $G$. Let $O$ be an orbit of this action,
$x\in O$ a chosen element, and $G_x$ the stabilizer of $x$ in
$G$. Let $U$ be an irreducible representation of $G_x$.
Then we define a representation $V_{(O,U)}$ of $G\ltimes A$ as
follows.

As a representation of $G$, we set
$$
V_{(O,x,U)}=\Ind_{G_x}^GU=
\lbrace{f: G\to U|f(hg)=hf(g), h\in G_x\rbrace}.
$$
Next, we introduce an additional action of $A$ on this space by
$(af)(g)=x(g(a))f(g)$. Then it's easy to check that these
two actions combine into an action of $G\ltimes A$.
Also, it is clear that this representation does not really depend
on the choice of $x$, in the following sense. Let $x,y\in O$, and
$g\in G$ be such that $gx=y$, and let $g(U)$ be the
representation of $G_y$ obtained from the representation $U$ of
$G_x$ by the action of $g$. Then $V_{(O,x,U)}$ is (naturally)
isomorphic to $V_{(O,y,g(U))}$. Thus we will denote $V_{(O,x,U)}$
by $V_{(O,U)}$ (remembering, however, that $x$ has been fixed). 

\begin{theorem}\label{semidi}
(i) The representations $V_{(O,U)}$ are irreducible.

(ii) They are pairwise nonisomorphic.

(iii) They form a complete set of irreducible representations of
$G\ltimes A$.

(iv) The character of $V=V_{(O,U)}$ is given by the Mackey-type formula
$$
\chi_V(a,g)=\frac{1}{|G_x|}\sum_{h\in G: hgh^{-1}\in G_x}x(h(a))
\chi_U(hgh^{-1}).
$$
\end{theorem}

\begin{proof}
(i) Let us decompose $V=V_{(O,U)}$ as an $A$-module. Then we get
$$
V=\oplus_{y\in O}V_y,
$$
where $V_y=\lbrace{v\in V_{(O,U)}|av=(y,a)v, a\in A\rbrace}$.
(Equivalently, $V_y=\lbrace v\in V_{(O,U)}|v(g)=0\text{ unless
}gy=x\rbrace $). So if $W\subset V$ is a subrepresentation, 
then $W=\oplus_{y\in O} W_y$,
where $W_y\subset V_y$. Now, $V_y$ is a representation of $G_y$,
which goes to $U$ under any isomorphism $G_y\to G_x$ determined
by $g\in G$ mapping $x$ to $y$. Hence, $V_y$ is irreducible
over $G_y$, so $W_y=0$ or $W_y=V_y$ for each $y$. Also, if
$hy=z$ then $hW_y=W_z$, so either $W_y=0$ for all $y$ or
$W_y=V_y$ for all $y$, as desired.

(ii) The orbit $O$ is determined by the $A$-module structure of
$V$, and the representation $U$ by the structure of $V_x$ as a
$G_x$-module.

(iii) We have
$$
\sum_{U,O}\dim V_{(U,O)}^2=\sum_{U,O}|O|^2 (\dim U)^2=
$$
$$
\sum_O |O|^2 |G_x|=\sum_O |O| |G/G_x| |G_x|=|G|\sum_O
|O|=|G||A^\vee|=|G\ltimes A|.
$$

(iv) The proof is essentially the same as that of the Mackey
formula.
\end{proof}

{\bf Exercise.} Redo Problems \ref{2:5}(a), \ref{2:6}, \ref{3:4} 
using Theorem \ref{semidi}.

{\bf Exercise.} Deduce parts (i)-(iii) of Theorem \ref{semidi}
from part (iv).  

\newpage \section{Quiver Representations}

\subsection{Problems}
\begin{problem}\label{5:1}
{\bf Field embeddings.}
Recall that $k(y_1,...,y_m)$ denotes the
field of rational functions of $y_1,...,y_m$ over a field $k$.
Let $f: k[x_1,...,x_n]\to k(y_1,...,y_m)$ be an injective
$k$-algebra homomorphism. Show that $m\ge n$. (Look at the growth of
dimensions of the spaces $W_N$ of polynomials of degree $N$ in
$x_i$ and their images under $f$ as $N\to \infty$).
Deduce that if  $f: k(x_1,...,x_n)\to k(y_1,...,y_m)$ is a field
embedding, then $m\ge n$.
\end{problem}

\begin{problem}\label{5:2}
{\bf Some algebraic geometry.}

Let $k$ be an algebraically closed field, and $G=GL_n(k)$.
Let $V$ be a polynomial representation of $G$.
Show that if $G$ has finitely many orbits on $V$ then ${\rm
dim}(V)\le n^2$. Namely:

(a) Let $x_1,...,x_N$ be linear coordinates on $V$.
Let us say that a subset $X$ of $V$ is Zariski dense if
any polynomial $f(x_1,...,x_N)$ which vanishes on $X$
is zero (coefficientwise). Show that if $G$ has finitely many
orbits on $V$ then $G$ has at least one Zariski dense orbit on $V$.

(b) Use (a) to construct a field embedding $k(x_1,...,x_N)\to
k(g_{pq})$, then use Problem \ref{5:1}.

(c) generalize the result of this problem to the case when
$G=GL_{n_1}(k)\times...\times GL_{n_m}(k)$.
\end{problem}

\begin{problem}\label{5:3} {\bf Dynkin diagrams.}

Let $\Gamma$ be a graph, i.e., a finite set of points (vertices)
connected with a certain number of edges (we allow multiple edges).
We assume that $\Gamma$
is connected (any vertex can be connected to any other by a path of edges)
and has no self-loops (edges from a vertex to itself).
Suppose the vertices of $\Gamma$ are labeled by integers
$1,...,N$. Then one can assign
to $\Gamma$ an $N\times N$ matrix $R_\Gamma=(r_{ij})$, where $r_{ij}$
is the number of edges connecting vertices $i$ and $j$.
This matrix is obviously symmetric, and is called the adjacency
matrix. Define the matrix $A_\Gamma=2I-R_\Gamma$, where $I$ is the identity
matrix.

{\bf Main definition:} $\Gamma$ is said to be a Dynkin diagram if
the quadratic from on ${\Bbb R}^N$
with matrix $A_\Gamma$ is positive definite.

Dynkin diagrams appear in many areas of mathematics
(singularity theory, Lie algebras, representation theory,
algebraic geometry, mathematical physics, etc.)
In this problem you will get a complete classification of Dynkin diagrams.
Namely, you will prove

{\bf Theorem.} $\Gamma$ is a Dynkin diagram
if and only if it is one on the following graphs:

\begin{itemize}

\item $A_n$ : {\large \vspace{-.55cm} $$\stackrel{}{\circ}\hspace{-.2cm}\sn\hspace{-.2cm}\stackrel{}{\circ} \dots \stackrel{\text{}}{\circ}\hspace{-.29cm}\sn\hspace{-.18cm}\stackrel{\text{}}{\circ}$$}

\item $D_n$: {\large \vspace{-.95cm} $$\stackrel{}{\circ}\hspace{-.2cm}\sn\hspace{-.2cm}\stackrel{}{\circ} \dots \stackrel{\text{}}{\circ}\hspace{-.29cm}\sn\hspace{-.23cm}\stackrel{\text{}}{\circ}$$
\vspace{-0.85cm}$$\hspace{1.03cm}|$$
\vspace{-0.9cm}$$\hspace{1.05cm}\circ\text{}$$}
\\
\item $E_6$  : {\large \vspace{-.55cm} $$\stackrel{}{\circ}\hspace{-.2cm}\sn\hspace{-.2cm}\stackrel{}{\circ} \hspace{-.2cm}\sn\hspace{-.2cm} \stackrel{\text{}}{\circ}\hspace{-.2cm}\sn\hspace{-.2cm}\stackrel{}{\circ}\hspace{-.2cm}\sn \hspace{-.2cm}\stackrel{}{\circ}$$
\vspace{-.87cm}$$\hspace{.04cm}|$$
\vspace{-.97cm}$$\hspace{-.02cm}\circ\text{}$$}

\item $E_7$  : {\large \vspace{-.55cm} $$\hspace{.5cm}\stackrel{}{\circ}\hspace{-.2cm}\sn\hspace{-.2cm}\stackrel{}{\circ} \hspace{-.2cm}\sn\hspace{-.2cm} \stackrel{\text{}}{\circ}\hspace{-.2cm}\sn\hspace{-.2cm}\stackrel{}{\circ}\hspace{-.2cm}\sn \hspace{-.2cm}\stackrel{}{\circ}\hspace{-.2cm}\sn\hspace{-.2cm}\stackrel{}{\circ}$$
\vspace{-0.9cm}$$\hspace{1.02cm}|$$
\vspace{-0.9cm}$$\hspace{1.02cm}\circ\text{}$$}

\item $E_8$  : {\large \vspace{-1.cm} $$\hspace{1cm}\stackrel{}{\circ}\hspace{-.2cm}\sn\hspace{-.2cm}\stackrel{}{\circ} \hspace{-.2cm}\sn\hspace{-.2cm} \stackrel{\text{}}{\circ}\hspace{-.2cm}\sn\hspace{-.2cm}\stackrel{}{\circ}\hspace{-.2cm}\sn \hspace{-.2cm}\stackrel{}{\circ}\hspace{-.2cm}\sn\hspace{-.2cm}\stackrel{}{\circ} \hspace{-.2cm}\sn\hspace{-.2cm}\stackrel{}{\circ}$$
\vspace{-1.00cm}$$\hspace{2.05cm}|$$
\vspace{-0.9cm}$$\hspace{2.05cm}\circ\text{}$$}

\end{itemize}

(a) Compute the determinant of $A_\Gamma$ where $\Gamma=A_N,D_N$.
(Use the row decomposition rule, and write down a recursive equation
for it). Deduce by Sylvester criterion\footnote{Recall the Sylvester criterion: a symmetric real matrix is positive definite
if and only if all its upper left corner principal minors are positive.} that
$A_N,D_N$ are Dynkin diagrams.\footnote{The Sylvester criterion
says that a symmetric bilinear form $(,)$ on
$\Bbb R^N$ is positive definite if and only if
for any $k\le N$, $\det_{1\le i,j\le k}(e_i,e_j)>0$.}

(b) Compute the determinants of $A_\Gamma$
for $E_6,E_7,E_8$ (use row decomposition
and reduce to (a)). Show they are Dynkin diagrams.

(c) Show that if $\Gamma$ is a Dynkin diagram, it cannot have cycles.
For this, show that $det(A_\Gamma)=0$ for a graph $\Gamma$ below
\footnote{Please ignore the numerical labels; they will be
relevant for Problem \ref{6:3} below.}

\includegraphics{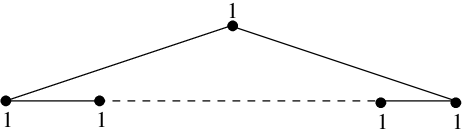}

(show that the sum of rows is 0). Thus $\Gamma$ has to be a tree.

(d) Show that if $\Gamma$ is a Dynkin diagram, it cannot have
vertices with 4 or more incoming edges, and
that $\Gamma$ can have no more than one vertex
with 3 incoming edges.
For this, show that $det(A_\Gamma)=0$ for a graph $\Gamma$ below:

\includegraphics{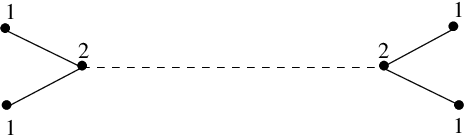}

(e) Show that $det(A_\Gamma)=0$ for all graphs $\Gamma$ below:

\includegraphics{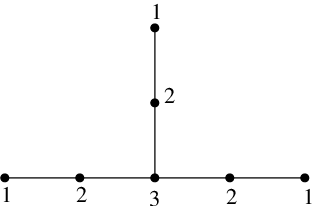}

\includegraphics{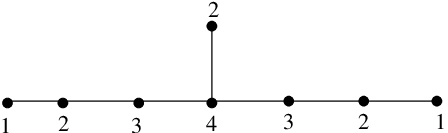}

\includegraphics{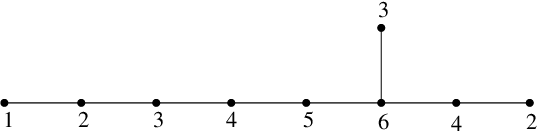}

(f) Deduce from (a)-(e) the classification theorem
for Dynkin diagrams.

(g) A (simply laced)
affine Dynkin diagram is a connected graph without self-loops such
that the quadratic form defined by $A_\Gamma$ is positive
semidefinite. Classify affine Dynkin diagrams.
(Show that they are exactly the forbidden diagrams
from (c)-(e)).
\end{problem}

\begin{problem}\label{5:4} Let $Q$ be a quiver with set of vertices $D$.
We say that $Q$ is of finite type
if it has finitely many indecomposable representations.
Let $b_{ij}$ be the number of edges from $i$ to $j$ in $Q$
($i,j\in D$).

There is the following remarkable theorem, proved by P.
Gabriel in early seventies.

{\bf Theorem.} A connected quiver $Q$ is of finite type if and only if
the corresponding unoriented graph (i.e., with directions of arrows forgotten)
is a Dynkin diagram.

In this problem you will prove the ``only if'' direction of this theorem
(i.e., why other quivers are NOT of finite type).

(a) Show that if $Q$ is of finite type then
for any rational numbers $x_i\ge 0$ which are not simultaneously zero,
one has $q(x_1,...,x_N)>0$, where
$$
q(x_1,...,x_N):=\sum_{i\in D} x_i^2-\frac{1}{2}\sum_{i,j\in D}b_{ij}x_ix_j.
$$

Hint. It suffices to check the result for integers: $x_i=n_i$.
First assume that $n_i\ge 0$,
and consider the space $W$ of representations $V$ of $Q$
such that ${\rm dim}V_i=n_i$. Show that the group
$\prod_i GL_{n_i}(k)$ acts with finitely many orbits
on $W\oplus k$, and use Problem \ref{5:2} to derive the
inequality. Then deduce the result in the case when $n_i$ are
arbitrary integers.

(b) Deduce that $q$ is a positive definite quadratic form.

Hint. Use the fact that $\Bbb Q$ is dense in $\Bbb R$.

(c) Show that a quiver of finite type can have no self-loops. Then,
using Problem \ref{5:3}, deduce the theorem.
\end{problem}

\begin{problem}\label{6:3}
Let $G\ne 1$ be a finite subgroup of $SU(2)$, and
$V$ be the 2-dimensional representation of $G$ coming
from its embedding into $SU(2)$. Let $V_i$, $i\in I$, be all the irreducible
representations of $G$. Let $r_{ij}$ be the multiplicity of $V_i$
in $V\otimes V_j$.
\end{problem}

(a) Show that $r_{ij}=r_{ji}$.

(b) The McKay graph of $G$,
$M(G)$, is the graph whose vertices are labeled by $i\in I$,
and $i$ is connected to $j$ by $r_{ij}$ edges.
Show that $M(G)$ is connected. (Use Problem \ref{6:1})

(c) Show that $M(G)$ is an affine Dynkin graph (one of the
``forbidden'' graphs in Problem \ref{5:3}).
For this, show that the matrix $a_{ij}=2\delta_{ij}-r_{ij}$
is positive semidefinite but not definite, and use Problem \ref{5:3}.

Hint. Let $f=\sum x_i\chi_{V_i}$, where $\chi_{V_i}$
be the characters of $V_i$. Show directly that
$((2-\chi_V)f,f)\ge 0$. When is it equal to $0$?
Next, show that $M(G)$ has no self-loops, by using that if $G$
is not cyclic then $G$ contains the central element $-Id\in SU(2)$.

(d) Which groups from Problem \ref{6:2} correspond to which diagrams?

(e) Using the McKay graph, find the dimensions of irreducible
representations of all finite $G\subset SU(2)$
(namely, show that they are the numbers labeling the vertices of
the affine Dynkin diagrams on our pictures).
Compare with the results on subgroups of $SO(3)$ we
obtained in Problem \ref{6:2}.

\subsection{Indecomposable representations of the quivers $A_1, A_2, A_3$}

\CompileMatrices

We have seen that a central question about representations of quivers
is whether a certain connected quiver has only finitely many indecomposable
representations. In the previous subsection it is shown
that only those quivers whose
underlying undirected graph is a Dynkin diagram may have this
property. To see if they actually do have this property, we first explicitly decompose representations of certain easy quivers.

\begin{remark}
  By an object of the type $\xymatrix@1{1 \ar[r] & 0}$ we mean a
map from a one-dimensional vector space to the zero
space. Similarly, an object of the type $\xymatrix@1{0 \ar[r] &
1}$ is a map from the zero space into a
one-dimensional space.
The object $\xymatrix@1{1 \ar[r] & 1}$ means an isomorphism from
a one-dimensional to another one-dimensional space. The numbers in such diagrams always mean the dimension of the attached spaces and the maps are the canonical maps (unless specified otherwise)
\end{remark}

\begin{example}[$A_1$]
The quiver $A_1$ consists of a single vertex and has no
edges. Since a representation of this quiver is just a single
vector space, the only indecomposable representation is the
ground field (=a one-dimensional space).
\end{example}

\begin{example}[$A_2$]
 \label{example2}
  The quiver $A_2$ consists of two vertices connected by a single edge.
  $$\xymatrix{\bullet \ar[r] & \bullet}$$
A representation of this quiver consists of two vector spaces $V, W$ and an operator $A:V \to W$.
  $$\xymatrix{\bullet \save[]+<0cm,-0.25cm>*{V}\restore \ar[r]^A & \bullet  \save[]+<0cm,-0.25cm>*{W}\restore}$$
To decompose this representation, we first let $V'$ be a complement to the kernel of $A$ in $V$ and let $W'$ be a complement to the image of $A$ in $W$. Then we can decompose the representation as follows
$$\xymatrix{\bullet \save[]+<0cm,-0.25cm>*{V}\restore \ar[r]^A & \bullet  \save[]+<0cm,-0.25cm>*{W}\restore} \;=\;
\xymatrix{\bullet \save[]+<0cm,-0.25cm>*{\ker A}\restore \ar[r]^0 & \bullet  \save[]+<0cm,-0.25cm>*{0}\restore} \;\oplus\;
\xymatrix{\bullet \save[]+<0cm,-0.25cm>*{V'}\restore \ar[r] \ar@{}@<-0.75ex>[r]^{\scriptstyle A \atop \textstyle \thicksim} & \bullet  \save[]+<0cm,-0.25cm>*{\Image A}\restore} \;\oplus\;
\xymatrix{\bullet \save[]+<0cm,-0.25cm>*{0}\restore \ar[r]^0 & \bullet  \save[]+<0cm,-0.25cm>*{W'}\restore}$$
The first summand is a multiple of the object $\xymatrix@1{1 \ar[r] & 0}$, the second a multiple of  $\xymatrix@1{1 \ar[r] & 1}$, the third of $\xymatrix@1{0 \ar[r] & 1}$.
We see that the quiver $A_2$ has three indecomposable representations, namely
$$\xymatrix@1{1 \ar[r] & 0}, \quad  \xymatrix@1{1 \ar[r] & 1}\quad \mbox{and} \quad \xymatrix@1{0 \ar[r] & 1}.$$
\end{example}

\begin{example}[$A_3$]
  The quiver $A_3$ consists of three vertices and two connections between them. So we have to choose between two possible orientations.
 $$\xymatrix@1{\bullet \ar[r] & \bullet \ar[r] & \bullet} \quad \mbox{or} \quad \xymatrix@1{\bullet \ar[r] & \bullet & \ar[l] \bullet}$$
  \begin{enumerate}
  \item We first look at the orientation
    $$\xymatrix@1{\bullet \ar[r] & \bullet \ar[r] & \bullet}.$$
    Then a representation of this quiver looks like
    $$\xymatrix@1@+=1.5 cm{\bullet \save[]+<0cm,-0.25cm>*{V}\restore \ar[r]^A & \bullet \save[]+<0cm,-0.25cm>*{W}\restore \ar[r]^B & \bullet \save[]+<0cm,-0.25cm>*{Y}\restore}.$$
    Like in Example \ref{example2} we first split away
    $$\xymatrix@1@+=1.5 cm{\bullet \save[]+<0cm,-0.25cm>*{\ker A}\restore \ar[r]^0 & \bullet \save[]+<0cm,-0.25cm>*{0}\restore \ar[r]^0 & \bullet \save[]+<0cm,-0.25cm>*{0}\restore}.$$
    This object is a multiple of $\xymatrix@1{1 \ar[r] & 0 \ar[r]
& 0}$. Next, let $Y'$ be a complement of $\Image B$ in $Y$. Then we can also split away
    $$\xymatrix@1@+=1.5 cm{\bullet \save[]+<0cm,-0.25cm>*{0}\restore \ar[r]^0 & \bullet \save[]+<0cm,-0.25cm>*{0}\restore \ar[r]^0 & \bullet \save[]+<0cm,-0.25cm>*{Y'}\restore}$$
    which is a multiple of the object $\xymatrix@1{0 \ar[r] & 0 \ar[r] & 1}$. This results in a situation where the map $A$ is injective and the map $B$ is surjective (we rename the spaces to simplify notation):
    $$\xymatrix@1@+=1.5 cm{*+<0.25 cm>{\bullet} \save[]+<0cm,-0.25cm>*{V}\restore \ar@{^{(}->}[r]^A & \bullet \save[]+<0cm,-0.25cm>*{W}\restore \ar@{->>}[r]^B & \bullet \save[]+<0cm,-0.25cm>*{Y}\restore}.$$
    Next, let $X = \ker (B \circ A)$ and let $X'$ be a complement of $X$ in $V$. Let $W'$ be a complement of $A(X)$ in $W$ such that $A(X') \subset W'$. Then we get
    $$\xymatrix@1@+=1.5 cm{*+<0.25 cm>{\bullet} \save[]+<0cm,-0.25cm>*{V}\restore \ar@{^{(}->}[r]^A & \bullet \save[]+<0cm,-0.25cm>*{W}\restore \ar@{->>}[r]^B & \bullet \save[]+<0cm,-0.25cm>*{Y}\restore} \;=\;
    \xymatrix@1@+=1.5 cm{*+<0.25 cm>{\bullet} \save[]+<0cm,-0.25cm>*{X}\restore \ar[r]^A & \bullet \save[]+<0cm,-0.28cm>*{A(X)}\restore \ar[r]^B & \bullet \save[]+<0cm,-0.25cm>*{0}\restore} \;\oplus\;
    \xymatrix@1@+=1.5 cm{*+<0.25 cm>{\bullet} \save[]+<0cm,-0.25cm>*{X'}\restore \ar@{^{(}->}[r]^A & \bullet \save[]+<0cm,-0.25cm>*{W'}\restore \ar@{->>}[r]^B & \bullet \save[]+<0cm,-0.25cm>*{Y}\restore}$$
    The first of these summands is a multiple of $\xymatrix@1{1 \ar[r] \ar@{}@<-0.1cm>[r]^{\textstyle\thicksim} & 1 \ar[r] & 0}$. Looking at the second summand, we now have a situation where $A$ is injective, $B$ is surjective and furthermore $\ker (B\circ A) = 0$. To simplify notation, we redefine
$$V = X', \; W = W'.$$
    Next we let $X = \Image(B\circ A)$ and let $X'$ be a complement of $X$ in $Y$. Furthermore, let $W' = B^{-1}(X')$. Then $W'$ is a complement of $A(V)$ in $W$. This yields the decomposition
    $$\xymatrix@1@+=1.5 cm{*+<0.25 cm>{\bullet} \save[]+<0cm,-0.25cm>*{V}\restore \ar@{^{(}->}[r]^A & \bullet \save[]+<0cm,-0.25cm>*{W}\restore \ar@{->>}[r]^B & \bullet \save[]+<0cm,-0.25cm>*{Y}\restore} \;=\;
    \xymatrix@1@+=1.5 cm{*+<0.25 cm>{\bullet} \save[]+<0cm,-0.25cm>*{V}\restore \ar@<-0.8ex>@{}[r]^{\textstyle\thicksim} \ar@{}@<0.3ex>[r]^A \ar[r] & \bullet \save[]+<0cm,-0.28cm>*{A(V)}\restore  \ar@<-0.8ex>@{}[r]^{\textstyle\thicksim} \ar@{}@<0.3ex>[r]^B \ar[r] & \bullet \save[]+<0cm,-0.25cm>*{X}\restore} \;\oplus\;
    \xymatrix@1@+=1.5 cm{*+<0.25 cm>{\bullet} \save[]+<0cm,-0.25cm>*{0}\restore \ar[r] & \bullet \save[]+<0cm,-0.25cm>*{W'}\restore \ar@{->>}[r]^B & \bullet \save[]+<0cm,-0.25cm>*{X'}\restore}$$
    Here, the first summand is a multiple of $\xymatrix@1{1 \ar[r] \ar@{}@<-0.1cm>[r]^{\textstyle\thicksim} & 1 \ar[r]\ar@{}@<-0.1cm>[r]^{\textstyle\thicksim} & 1}$. By splitting away the kernel of $B$, the second summand can be decomposed into multiples of $\xymatrix@1{0 \ar[r] & 1 \ar[r]\ar@{}@<-0.1cm>[r]^{\textstyle\thicksim} & 1}$ and $\xymatrix@1{0 \ar[r] & 1 \ar[r] & 0}.$
    So, on the whole, this quiver has six indecomposable representations:
$$\xymatrix@1{1 \ar[r] & 0 \ar[r] & 0}, \quad \xymatrix@1{0 \ar[r] & 0 \ar[r] & 1}, \quad \xymatrix@1{1 \ar[r] \ar@{}@<-0.1cm>[r]^{\textstyle\thicksim} & 1 \ar[r] & 0},$$
$$\xymatrix@1{1 \ar[r] \ar@{}@<-0.1cm>[r]^{\textstyle\thicksim} & 1 \ar[r]\ar@{}@<-0.1cm>[r]^{\textstyle\thicksim} & 1}, \quad \xymatrix@1{0 \ar[r] & 1 \ar[r]\ar@{}@<-0.1cm>[r]^{\textstyle\thicksim} & 1}, \quad \xymatrix@1{0 \ar[r] & 1 \ar[r] & 0}$$
\item Now we look at the orientation
\label{example32}
$$\quad \xymatrix@1{\bullet \ar[r] & \bullet & \ar[l] \bullet}.$$
Very similarly to the other orientation, we can split away objects of the type
$$\xymatrix@1{1 \ar[r] & 0 & 0 \ar[l]}, \quad \xymatrix@1{0 \ar[r] & 0 & 1 \ar[l]}$$
which results in a situation where both $A$ and $B$ are injective:
    $$\xymatrix@1@+=1.5 cm{*+<0.25cm>{\bullet} \save[]+<0cm,-0.25cm>*{V}\restore \ar@{^{(}->}[r]^A & \bullet \save[]+<0cm,-0.25cm>*{W}\restore \ar@{<-^{)}}[r]^B & *+<0.25cm>{\bullet} \save[]+<0cm,-0.25cm>*{Y}\restore}.$$
    By identifying $V$ and $Y$ as subspaces of $W$, this leads to
the problem of classifying pairs of subspaces of a given space
$W$ up to isomorphism (the \textbf{pair of subspaces
problem}). To do so, we first choose a complement $W'$ of $V\cap
Y$ in $W$, and set $V'=W'\cap V$, $Y'=W'\cap Y$.
Then we can decompose the representation as follows:
    $$\xymatrix@1@+=1.5 cm{*+<0.25cm>{\bullet} \save[]+<0cm,-0.25cm>*{V}\restore \ar@{^{(}->}[r] & \bullet \save[]+<0cm,-0.25cm>*{W}\restore \ar@{<-^{)}}[r] & *+<0.25cm>{\bullet} \save[]+<0cm,-0.25cm>*{Y}\restore} \;=\;
    \xymatrix@1@+=1.5 cm{*+<0.25cm>{\bullet} \save[]+<0cm,-0.25cm>*{V'}\restore \ar@{^{(}->}[r] & \bullet \save[]+<0cm,-0.25cm>*{W'}\restore \ar@{<-^{)}}[r] & *+<0.25cm>{\bullet} \save[]+<0cm,-0.25cm>*{Y'}\restore} \;\oplus\;
    \xymatrix@1@+=1.5 cm{*+<0.25cm>{\bullet} \save[]+<0cm,-0.25cm>*{V\cap Y}\restore \ar[r] \ar@<0.4ex>@{}[r] \ar@{}@<-0.1cm>[r]^{\textstyle\thicksim} & \bullet \save[]+<0cm,-0.25cm>*{V\cap Y}\restore &  *+<0.25cm>{\bullet} \ar[l] \ar@<-2.75ex>@{}[l] \ar@{}@<-0.25cm>[l]^{\textstyle\thicksim} \save[]+<0cm,-0.25cm>*{V\cap Y}\restore}.$$
    The second summand is a multiple of the object $\xymatrix@1{1 \ar[r] \ar@<-0.75ex>@{}[r]^{\textstyle \thicksim} & 1 & 1 \ar[l] \ar@<+0.75ex>@{}[l]_{\textstyle\thicksim}}$. We go on decomposing the first summand. Again, to simplify notation, we let
    $$V = V', \;W = W', \;Y = Y'.$$
    We can now assume that $V \cap Y = 0$. Next, let $W'$ be a complement of $V \oplus Y$ in $W$. Then we get
$$\xymatrix@1@+=1.5 cm{*+<0.25cm>{\bullet} \save[]+<0cm,-0.25cm>*{V}\restore \ar@{^{(}->}[r] & \bullet \save[]+<0cm,-0.25cm>*{W}\restore \ar@{<-^{)}}[r] & *+<0.25cm>{\bullet} \save[]+<0cm,-0.25cm>*{Y}\restore} \;=\;
\xymatrix@1@+=1.5 cm{*+<0.25cm>{\bullet} \save[]+<0cm,-0.25cm>*{V}\restore \ar@{^{(}->}[r] & \bullet \save[]+<0cm,-0.25cm>*{V \oplus Y}\restore \ar@{<-^{)}}[r] & *+<0.25cm>{\bullet} \save[]+<0cm,-0.25cm>*{Y}\restore} \;\oplus\;
\xymatrix@1@+=1.5 cm{*+<0.25cm>{\bullet} \save[]+<0cm,-0.25cm>*{0}\restore \ar[r] & \bullet \save[]+<0cm,-0.25cm>*{W'}\restore  & *+<0.25cm>{\bullet} \save[]+<0cm,-0.25cm>*{0}\restore  \ar[l]}$$
The second of these summands is a multiple of the indecomposable object $\xymatrix@1{0 \ar[r] & 1 & 0 \ar[l]}$. The first summand can be further decomposed as follows:
$$\xymatrix@1@+=1.5 cm{*+<0.25cm>{\bullet} \save[]+<0cm,-0.25cm>*{V}\restore \ar@{^{(}->}[r] & \bullet \save[]+<0cm,-0.25cm>*{V \oplus Y}\restore \ar@{<-^{)}}[r] & *+<0.25cm>{\bullet} \save[]+<0cm,-0.25cm>*{Y}\restore} \;=\;
\xymatrix@1@+=1.5 cm{*+<0.25cm>{\bullet} \save[]+<0cm,-0.25cm>*{V}\restore \ar[r] \ar@{}@<0.2ex>[r] \ar@{}@<-0.75ex>[r]^{\textstyle \thicksim}& \bullet \save[]+<0cm,-0.25cm>*{V}\restore  & *+<0.25cm>{\bullet} \save[]+<0cm,-0.25cm>*{0}\restore \ar[l]} \;\oplus\;
\xymatrix@1@+=1.5 cm{*+<0.25cm>{\bullet} \save[]+<0cm,-0.25cm>*{0}\restore \ar[r] & \bullet \save[]+<0cm,-0.25cm>*{Y}\restore  & *+<0.25cm>{\bullet} \save[]+<0cm,-0.25cm>*{Y}\restore \ar[l] \ar@{}@<-2.65ex>[l] \ar@{}@<-1.35ex>[l]^{\textstyle \thicksim}}$$
These summands are multiples of
$$\xymatrix@1{1 \ar[r] & 1 & 0 \ar[l]}, \quad \xymatrix@1{0 \ar[r] & 1 & 1 \ar[l]}$$
So - like in the other orientation - we get 6 indecomposable representations of $A_3$:
$$\xymatrix@1{1 \ar[r] & 0 & 0 \ar[l]}, \quad \xymatrix@1{0 \ar[r] & 0 & 1 \ar[l]}, \quad \xymatrix@1{1 \ar[r]^{\textstyle \thicksim} & 1 & 1 \ar[l]_{\textstyle\thicksim}},$$
$$\xymatrix@1{0 \ar[r] & 1 & 0 \ar[l]}, \quad \xymatrix@1{1 \ar[r] & 1 & 0 \ar[l]}, \quad \xymatrix@1{0 \ar[r] & 1 & 1 \ar[l]}$$
  \end{enumerate}
\end{example}

\subsection{Indecomposable representations of the quiver $D_4$}
As a last - slightly more complicated - example we consider the quiver $D_4$.
\begin{example}[$D_4$]
  We restrict ourselves to the orientation
  $$\xymatrix{\bullet \ar[r] & \bullet & \bullet \ar[l] \\ & \bullet \ar[u]}.$$
  So a representation of this quiver looks like
  $$\xymatrix@+=1cm{\bullet  \save[]+<0cm,-0.3cm>*{V_1}\restore \ar[r]^{A_1} & \bullet \save[]+<0cm,0.3cm>*{V}\restore & \bullet \save[]+<0cm,-0.3cm>*{V_3}\restore \ar[l]_{A_3} \\ & \bullet \save[]+<0cm,-0.3cm>*{V_2}\restore \ar[u]_{A_2}}$$
  The first thing we can do is - as usual - split away the kernels of the maps $A_1, A_2, A_3$. More precisely, we split away the representations
  $$\xymatrix@+=1cm{\bullet  \save[]+<0cm,-0.3cm>*{\ker A_1}\restore \ar[r]^{0} & \bullet \save[]+<0cm,0.3cm>*{0}\restore & \bullet \save[]+<0cm,-0.3cm>*{0}\restore \ar[l] \\ & \bullet \save[]+<0cm,-0.3cm>*{0}\restore \ar[u]} \quad\quad
\xymatrix@+=1cm{\bullet  \save[]+<0cm,-0.3cm>*{0}\restore \ar[r] & \bullet \save[]+<0cm,0.3cm>*{0}\restore & \bullet \save[]+<0cm,-0.3cm>*{0}\restore \ar[l] \\ & \bullet \save[]+<0cm,-0.3cm>*{\ker A_2}\restore \ar[u]_{0}} \quad\quad
\xymatrix@+=1cm{\bullet  \save[]+<0cm,-0.3cm>*{0}\restore \ar[r] & \bullet \save[]+<0cm,0.3cm>*{0}\restore & \bullet \save[]+<0cm,-0.3cm>*{\ker A_3}\restore \ar[l]_{0} \\ & \bullet \save[]+<0cm,-0.3cm>*{0}\restore \ar[u]}$$
These representations are multiples of the indecomposable objects
  $$\xymatrix@+=1cm{\bullet  \save[]+<0cm,-0.3cm>*{1}\restore \ar[r]^{0} & \bullet \save[]+<0cm,0.3cm>*{0}\restore & \bullet \save[]+<0cm,-0.3cm>*{0}\restore \ar[l] \\ & \bullet \save[]+<0cm,-0.3cm>*{0}\restore \ar[u]} \quad\quad
\xymatrix@+=1cm{\bullet  \save[]+<0cm,-0.3cm>*{0}\restore \ar[r] & \bullet \save[]+<0cm,0.3cm>*{0}\restore & \bullet \save[]+<0cm,-0.3cm>*{0}\restore \ar[l] \\ & \bullet \save[]+<0cm,-0.3cm>*{1}\restore \ar[u]_{0}} \quad\quad
\xymatrix@+=1cm{\bullet  \save[]+<0cm,-0.3cm>*{0}\restore \ar[r] & \bullet \save[]+<0cm,0.3cm>*{0}\restore & \bullet \save[]+<0cm,-0.3cm>*{1}\restore \ar[l]_{0} \\ & \bullet \save[]+<0cm,-0.3cm>*{0}\restore \ar[u]}$$
So we get to a situation where all of the maps $A_1, A_2, A_3$ are injective.
  $$\xymatrix@+=1cm{\bullet  \save[]+<0cm,-0.3cm>*{V_1}\restore \ar@{^{(}->}[r]^{A_1} & \bullet \save[]+<0cm,0.3cm>*{V}\restore & \bullet \save[]+<0cm,-0.3cm>*{V_3}\restore \ar@{_{(}->}[l]_{A_3} \\ & \bullet \save[]+<0cm,-0.3cm>*{V_2}\restore \ar@{_{(}->}[u]_{A_2}}$$
As in \ref{example32}, we can then identify the spaces $V_1, V_2, V_3$ with subspaces of $V$. So we get to the \textbf{triple of subspaces problem} of classifying a triple of subspaces of a given space $V$.

The next step is to split away a multiple of
$$\xymatrix@+=1cm{\bullet  \save[]+<0cm,-0.3cm>*{0}\restore \ar[r] & \bullet \save[]+<0cm,0.3cm>*{1}\restore & \bullet \save[]+<0cm,-0.3cm>*{0}\restore \ar[l] \\ & \bullet \save[]+<0cm,-0.3cm>*{0}\restore \ar[u]}$$
to reach a situation where $$V_1 + V_2 + V_3 = V.$$
By letting $Y = V_1 \,\cap\, V_2 \,\cap\, V_3$, choosing a
complement $V'$ of $Y$ in $V$, and setting $V_i'=V'\cap V_i$,
$i=1,2,3$, we can decompose this representation into
  $$\vcenter{\xymatrix@+=1cm{\bullet  \save[]+<0cm,-0.3cm>*{V'_1}\restore \ar@{^{(}->}[r] & \bullet \save[]+<0cm,0.3cm>*{V'}\restore & \bullet \save[]+<0cm,-0.3cm>*{V'_3}\restore \ar@{_{(}->}[l] \\ & \bullet \save[]+<0cm,-0.3cm>*{V'_2}\restore \ar@{_{(}->}[u]}} \quad \oplus \quad
\vcenter{\xymatrix@+=1cm{\bullet \save[]+<0cm,-0.3cm>*{Y}\restore \ar[r]^{\textstyle \thicksim} & \bullet \save[]+<0cm,0.3cm>*{Y}\restore \ar@<+1ex>@{}[d]^<<<{}="b" & \bullet \save[]+<0cm,-0.3cm>*{Y}\restore \ar[l]_{\textstyle \thicksim} \\ & \bullet \save[]+<0cm,-0.3cm>*{Y}\restore \ar[u] \ar@<-1ex>@{}[u]_<<<{}="a" \ar@{~} "a";"b" } }$$
The last summand is a multiple of the indecomposable representation
$$\vcenter{\xymatrix@+=1cm{\bullet  \save[]+<0cm,-0.3cm>*{1}\restore \ar[r]^{\textstyle \thicksim} & \bullet \save[]+<0cm,0.3cm>*{1}\restore \ar@<+1ex>@{}[d]^<<<{}="b" & \bullet \save[]+<0cm,-0.3cm>*{1}\restore \ar[l]_{\textstyle \thicksim} \\ & \bullet \save[]+<0cm,-0.3cm>*{1}\restore \ar[u] \ar@<-1ex>@{}[u]_<<<{}="a" \ar@<-1ex>@{}[u]_<<<{}="a" \ar@{~} "a";"b" } }$$
So - considering the first summand and renaming the spaces to simplify notation - we are in a situation where
$$V = V_1 + V_2 + V_3, \quad V_1 \cap V_2 \cap V_3 = 0.$$
As a next step, we let $Y= V_1 \cap V_2$ and we choose a
complement $V'$ of $Y$ in $V$ such that $V_3 \subset V'$,
and set $V'_1=V'\cap V_1,
V'_2=V'\cap V_2$. This yields the decomposition
$$\vcenter{\xymatrix@+=1cm{\bullet  \save[]+<0cm,-0.3cm>*{V_1}\restore \ar@{^{(}->}[r] & \bullet \save[]+<0cm,0.3cm>*{V}\restore & \bullet \save[]+<0cm,-0.3cm>*{V_3}\restore \ar@{_{(}->}[l] \\ & \bullet \save[]+<0cm,-0.3cm>*{V_2}\restore \ar@{_{(}->}[u]}} \quad=\quad
\vcenter{\xymatrix@+=1cm{\bullet  \save[]+<0cm,-0.3cm>*{V'_1}\restore \ar@{^{(}->}[r] & \bullet \save[]+<0cm,0.3cm>*{V'}\restore & \bullet \save[]+<0cm,-0.3cm>*{V_3}\restore \ar@{_{(}->}[l] \\ & \bullet \save[]+<0cm,-0.3cm>*{V'_2}\restore \ar@{_{(}->}[u]}} \quad\oplus\quad
\vcenter{\xymatrix@+=1cm{\bullet  \save[]+<0cm,-0.3cm>*{Y}\restore \ar[r]^{\textstyle \thicksim} & \bullet \save[]+<0cm,0.3cm>*{Y}\restore \ar@<+1ex>@{}[d]^<<<{}="b" & \bullet \save[]+<0cm,-0.3cm>*{0}\restore \ar[l] \\ & \bullet \save[]+<0cm,-0.3cm>*{Y}\restore \ar[u] \ar@<-1ex>@{}[u]_<<<{}="a" \ar@<-1ex>@{}[u]_<<<{}="a" \ar@{~} "a";"b" } }
$$
The second summand is a multiple of the indecomposable object
$$\xymatrix@+=1cm{\bullet  \save[]+<0cm,-0.3cm>*{1}\restore \ar[r]^{\textstyle \thicksim} & \bullet \save[]+<0cm,0.3cm>*{1}\restore \ar@<+1ex>@{}[d]^<<<{}="b" & \bullet \save[]+<0cm,-0.3cm>*{0}\restore \ar[l] \\ & \bullet \save[]+<0cm,-0.3cm>*{1}\restore \ar[u] \ar@<-1ex>@{}[u]_<<<{}="a" \ar@<-1ex>@{}[u]_<<<{}="a" \ar@{~} "a";"b" }.$$
In the resulting situation we have $V_1 \cap V_2 = 0$. Similarly we can split away multiples of
$$\vcenter{\xymatrix@+=1cm{\bullet  \save[]+<0cm,-0.3cm>*{1}\restore \ar[r]^{\textstyle \thicksim} & \bullet \save[]+<0cm,0.3cm>*{1}\restore & \bullet \save[]+<0cm,-0.3cm>*{1}\restore \ar[l]_{\textstyle \thicksim} \\ & \bullet \save[]+<0cm,-0.3cm>*{0}\restore \ar[u]} } \quad\mbox{and}\quad
\vcenter{\xymatrix@+=1cm{\bullet  \save[]+<0cm,-0.3cm>*{0}\restore \ar[r] & \bullet \save[]+<0cm,0.3cm>*{1}\restore \ar@<+1ex>@{}[d]^<<<{}="b" & \bullet \save[]+<0cm,-0.3cm>*{1}\restore \ar[l]_{\textstyle \thicksim} \\ & \bullet \save[]+<0cm,-0.3cm>*{1}\restore \ar[u] \ar@<-1ex>@{}[u]_<<<{}="a" \ar@<-1ex>@{}[u]_<<<{}="a" \ar@{~} "a";"b" } }$$
to reach a situation where the spaces $V_1, V_2, V_3$ do not intersect pairwise
$$V_1 \cap V_2 \;=\; V_1 \cap V_3 \;=\; V_2 \cap V_3 \;=\; 0.$$
If $V_1 \nsubseteq V_2 \oplus V_3$ we let $Y = V_1 \cap \left( V_2 \oplus V_3 \right)$. We let $V'_1$ be a complement of $Y$ in $V_1$. Since then $V'_1 \cap (V_2 \oplus V_3) = 0$, we can select a complement $V'$ of $V'_1$ in $V$ which contains $V_2 \oplus V_3$. This gives us the decomposition
$$\vcenter{\xymatrix@+=1cm{\bullet  \save[]+<0cm,-0.3cm>*{V_1}\restore \ar@{^{(}->}[r] & \bullet \save[]+<0cm,0.3cm>*{V}\restore & \bullet \save[]+<0cm,-0.3cm>*{V_3}\restore \ar@{_{(}->}[l] \\ & \bullet \save[]+<0cm,-0.3cm>*{V_2}\restore \ar@{_{(}->}[u]}} \quad=\quad
\vcenter{\xymatrix@+=1cm{\bullet  \save[]+<0cm,-0.3cm>*{V'_1}\restore \ar[r]^{\textstyle \thicksim} & \bullet \save[]+<0cm,0.3cm>*{V'_1}\restore & \bullet \save[]+<0cm,-0.3cm>*{0}\restore \ar[l] \\ & \bullet \save[]+<0cm,-0.3cm>*{0}\restore \ar[u]} } \quad\oplus\quad
\vcenter{\xymatrix@+=1cm{\bullet  \save[]+<0cm,-0.3cm>*{Y}\restore \ar@{^{(}->}[r] & \bullet \save[]+<0cm,0.3cm>*{V'}\restore & \bullet \save[]+<0cm,-0.3cm>*{V_3}\restore \ar@{_{(}->}[l] \\ & \bullet \save[]+<0cm,-0.3cm>*{V_2}\restore \ar@{_{(}->}[u]}}$$
The first of these summands is a multiple of
$$\vcenter{\xymatrix@+=1cm{\bullet  \save[]+<0cm,-0.3cm>*{1}\restore \ar[r]^{\textstyle \thicksim} & \bullet \save[]+<0cm,0.3cm>*{1}\restore \ar@<+1ex>@{}[d]^<<<{}="b" & \bullet \save[]+<0cm,-0.3cm>*{0}\restore \ar[l] \\ & \bullet \save[]+<0cm,-0.3cm>*{0}\restore \ar[u] } }$$
By splitting these away we get to a situation where $V_1 \subseteq V_2 \oplus V_3$. Similarly, we can split away objects of the type
$$\vcenter{\xymatrix@+=1cm{\bullet \save[]+<0cm,-0.3cm>*{0}\restore \ar[r] & \bullet \save[]+<0cm,0.3cm>*{1}\restore \ar@<+1ex>@{}[d]^<<<{}="b" & \bullet \save[]+<0cm,-0.3cm>*{0}\restore \ar[l] \\ & \bullet \save[]+<0cm,-0.3cm>*{1}\restore \ar[u] \ar@<-1ex>@{}[u]_<<<{}="a" \ar@<-1ex>@{}[u]_<<<{}="a" \ar@{~} "a";"b" }} \quad\mbox{and}\quad
\vcenter{\xymatrix@+=1cm{\bullet  \save[]+<0cm,-0.3cm>*{0}\restore \ar[r] & \bullet \save[]+<0cm,0.3cm>*{1}\restore \ar@<+1ex>@{}[d]^<<<{}="b" & \bullet \save[]+<0cm,-0.3cm>*{1}\restore \ar[l]_{\textstyle \thicksim} \\ & \bullet \save[]+<0cm,-0.3cm>*{0}\restore \ar[u] } }$$
to reach a situation in which the following conditions hold
\begin{enumerate}
\item $V_1 + V_2 + V_3 = V.$
\item $V_1 \cap V_2 = 0, \quad V_1 \cap V_3 = 0, \quad V_2 \cap V_3 = 0.$
\item $V_1 \subseteq V_2 \oplus V_3, \quad V_2 \subseteq V_1 \oplus V_3, \quad V_3 \subseteq V_1 \oplus V_2.$
\end{enumerate}
But this implies that
$$V_1 \oplus V_2 = V_1 \oplus V_3 = V_2 \oplus V_3 = V.$$
So we get
$$\dim V_1 = \dim V_2 = \dim V_3 = n$$
and
$$\dim V = 2n.$$
Since $V_3 \subseteq V_1 \oplus V_2$ we can write every element of $V_3$ in the form
$$x \in V_3, \quad x = (x_1, x_2), \; x_1 \in V_1, \, x_2 \in V_2.$$
We then can define the projections
$$B_1: V_3 \to V_1, \quad (x_1,x_2) \mapsto x_1,$$
$$B_2: V_3 \to V_2, \quad (x_1,x_2) \mapsto x_2.$$
Since $V_3\cap V_1=0, V_3 \cap V_2=0$, these maps have to be injective and therefore are isomorphisms. We then define the isomorphism
$$A = B_2 \circ B_1^{-1}: V_1\to V_2.$$
Let $e_1, \ldots, e_n$ be a basis for $V_1$. Then we get
$$V_1 = \CC\,e_1 \oplus \CC\,e_2 \oplus \dots \oplus \CC\,e_n$$
$$V_2 = \CC\,Ae_1 \oplus \CC\,Ae_2 \oplus \dots \oplus \CC\,Ae_n$$
$$V_3 = \CC\,(e_1+Ae_1) \oplus \CC\,(e_2+Ae_2) \oplus \dots
\oplus \CC\,(e_n+Ae_n).$$
So we can think of $V_3$ as the graph of an isomorphism $A:V_1 \to V_2$.  From this we obtain the decomposition
$$\vcenter{\xymatrix@+=1cm{\bullet  \save[]+<0cm,-0.3cm>*{V_1}\restore \ar@{^{(}->}[r] & \bullet \save[]+<0cm,0.3cm>*{V}\restore & \bullet \save[]+<0cm,-0.3cm>*{V_3}\restore \ar@{_{(}->}[l] \\ & \bullet \save[]+<0cm,-0.3cm>*{V_2}\restore \ar@{_{(}->}[u]}} \quad=\quad \bigoplus_{j=1}^n \quad
\vcenter{ \xymatrix@+=1cm{ \bullet \save[]+<0cm,-0.3cm>*{\CC(1,0)}\restore \ar@{^{(}->}[r] & \bullet \save[]+<0cm,0.3cm>*{\CC^2}\restore & \bullet \save[]+<0cm,-0.3cm>*{\CC(1,1)}\restore \ar@{_{(}->}[l] \\ & \bullet \save[]+<0cm,-0.3cm>*{\CC(0,1)}\restore \ar@{_{(}->}[u] }}$$
These correspond to the indecomposable object
  $$\xymatrix@+=1cm{\bullet \save[]+<0cm,-0.3cm>*{1}\restore \ar[r] & \bullet \save[]+<0cm,0.3cm>*{2}\restore & \bullet \save[]+<0cm,-0.3cm>*{1}\restore \ar[l] \\ & \bullet \save[]+<0cm,-0.3cm>*{1}\restore \ar[u]}$$
Thus the quiver $D_4$ with the selected orientation has 12 indecomposable objects.
If one were to explicitly decompose representations for the other possible orientations, one would also find 12 indecomposable objects.
\end{example}
It appears as if the number of indecomposable representations does not depend on the orientation of the edges, and indeed - Gabriel's theorem will generalize this observation.

\subsection{Roots}
 From now on, let $\Gamma$ be a fixed graph of type $A_n, D_n, E_6, E_7, E_8$. We denote the adjacency matrix of $\Gamma$ by $R_\Gamma$.
\begin{definition}[Cartan Matrix]
 We define the Cartan matrix as
$$A_\Gamma = 2 \Id - R_\Gamma.$$
\end{definition}
On the lattice $\ZZ^n$ (or the space $\RR^n$) we then define an inner product
$$B(x,y) = x^{T} A_\Gamma y$$
corresponding to the graph $\Gamma$.
\begin{lemma}
  \begin{enumerate}
  \item $B$ is positive definite.
  \item $B(x,x)$ takes only even values for $x \in \ZZ^n$.
  \end{enumerate}
\end{lemma}
\begin{proof}
  \begin{enumerate}
  \item This follows by definition, since $\Gamma$ is a Dynkin diagram.
  \item By the definition of the Cartan matrix we get
    $$B(x, x) = x^{T} A_\Gamma x = \sum_{i,j} x_i\, a_{ij}\, x_j = 2 \sum_i x_i^2 + \sum_{i,j,\;i \neq j} x_i\, a_{ij}\, x_j = 2 \sum_i x_i^2 + 2 \cdot \sum_{i<j} a_{ij}\, x_ix_j$$
    which is even.
  \end{enumerate}
\end{proof}
\begin{definition}
  A root with respect to a certain positive
inner product is a shortest (with respect to this inner product),
nonzero vector in $\ZZ^n$.
\end{definition}
So for the inner product $B$, a root is a nonzero vector $x \in \ZZ^n$ such that
$$B(x,x) = 2.$$
\begin{remark}
  There can be only finitely many roots, since all of them have to lie in some ball.
\end{remark}
\begin{definition}
  We call vectors of the form
  $$\alpha_i = (0, \ldots, \overbrace{1}^{i-\mathrm{th}}, \ldots, 0)$$
  \textbf{simple roots}.
\end{definition}
The $\alpha_i$ naturally form a basis of the lattice $\ZZ^n$.
\begin{lemma}\label{posneg}
  Let $\alpha$ be a root, $\alpha = \sum_{i=1}^n k_i \alpha_i$. Then either $k_i \geq 0$ for all $i$ or $k_i \leq 0$ for all $i$.
\end{lemma}
\begin{proof}
  Assume the contrary, i.e., $k_i>0$, $k_j<0$. Without loss of generality, we can also assume that $k_s=0$ for all $s$ between $i$ and $j$. We can identify the indices $i,j$ with vertices of the graph $\Gamma$.
$$\xymatrix{\bullet \ar@{-}[r] & \bullet \save[]+<0cm,-0.3cm>*{i}\restore  \ar@{-}[r]^{\epsilon} & \bullet \save[]+<0cm,+0.3cm>*{i'}\restore \ar@{-}[r] & \bullet \ar@{-}[r] & \bullet  \save[]+<0cm,-0.3cm>*{j}\restore \ar@{-}[r] & \bullet \ar@{-}[r] & \bullet \\ & & \bullet \ar@{-}[u] }$$
Next, let $\epsilon$ be the edge connecting $i$ with the next vertex towards $j$ and $i'$ be the vertex on the other end of $\epsilon$. We then let $\Gamma_1, \Gamma_2$ be the graphs obtained from $\Gamma$ by removing $\epsilon$. Since $\Gamma$ is supposed to be a Dynkin diagram - and therefore has no cycles or loops - both $\Gamma_1$ and $\Gamma_2$ will be connected graphs, which are not connected to each other.
$$\framebox{\xymatrix{\bullet \ar@{-}[r] & \bullet \save[]+<0cm,-0.3cm>*{i}\restore}}_{\;\Gamma_1} \quad\quad
\framebox{\xymatrix{\bullet \ar@{-}[r] & \bullet \ar@{-}[r] & \bullet  \save[]+<0cm,-0.3cm>*{j}\restore \ar@{-}[r] & \bullet \ar@{-}[r] & \bullet \\ \bullet \ar@{-}[u]}}_{\;\Gamma_2}$$
Then we have $i \in \Gamma_1, \, j \in \Gamma_2$. We define
$$\beta = \sum_{m \in \Gamma_1} k_m \alpha_m, \quad \gamma = \sum_{m \in \Gamma_2} k_m \alpha_m.$$
With this choice we get
$$\alpha = \beta + \gamma.$$
Since $k_i >0, k_j < 0$ we know that $\beta \neq 0, \gamma \neq 0$ and therefore
$$B(\beta, \beta) \geq 2, \quad B(\gamma, \gamma) \geq 2.$$
Furthermore,
$$B(\beta, \gamma) = - k_i k_{i'},$$
since $\Gamma_1, \Gamma_2$ are only connected at $\epsilon$. But this has to be a nonnegative number, since $k_i > 0$ and $k_{i'} \leq 0$. This yields
$$B(\alpha, \alpha) = B(\beta + \gamma, \beta + \gamma) = \underbrace{B(\beta, \beta)}_{\geq 2} + 2 \underbrace{B(\beta, \gamma)}_{\geq 0} + \underbrace{B(\gamma, \gamma)}_{\geq 2} \geq 4.$$
But this is a contradiction, since $\alpha$ was assumed to be a root.
\end{proof}
\begin{definition}
  We call a root $\alpha  = \sum_i k_i \alpha_i$ a positive root
if all $k_i \geq 0$. A root for which $k_i \leq 0$ for all $i$ is called a negative root.
\end{definition}
\begin{remark}
Lemma \ref{posneg} states that every root is either positive or negative.
\end{remark}

\begin{example}
  \begin{enumerate}
  \item Let $\Gamma$ be of the type $A_{N-1}$. Then the lattice $L= \ZZ^{N-1}$ can be realized as a subgroup of the lattice $\ZZ^N$ by letting $L \subseteq \ZZ^N$ be the subgroup of all vectors $(x_1, \ldots, x_N)$ such that
$$\sum_i x_i = 0.$$
The vectors
\begin{eqnarray*}
\alpha_1 &=& (1, -1, 0, \ldots, 0)\\
\alpha_2 &=& (0, 1, -1, 0, \ldots, 0)\\
&\vdots&\\
\alpha_{N-1} &=& (0, \ldots, 0, 1, -1)
\end{eqnarray*}
naturally form a basis of $L$.
Furthermore, the standard inner product $$(x,y) = \sum x_iy_i$$ on $\ZZ^N$ restricts to the inner product $B$ given by $\Gamma$ on $L$, since it takes the same values on the basis vectors:
$$(\alpha_i, \alpha_i) = 2$$
$$(\alpha_i, \alpha_j) = \left\{
  \begin{array}[h]{ll}
   -1 & i,j \;\mbox{adjacent} \\
   0 & \mbox{otherwise}
  \end{array}
\right.$$
  This means that vectors of the form
  $$(0, \ldots, 0, 1, 0, \ldots, 0, -1, 0, \ldots, 0) = \alpha_i + \alpha_{i+1} + \dots + \alpha_{j-1}$$
  and
  $$(0, \ldots, 0, -1, 0, \ldots, 0, 1, 0, \ldots, 0) = -(\alpha_i + \alpha_{i+1} + \dots + \alpha_{j-1})$$
  are the roots of $L$. Therefore the number of positive roots in $L$ equals
  $$\frac{N(N-1)}{2}.$$
\item As a fact we also state the number of positive roots in the other Dynkin diagrams:
\center{
  \begin{tabular}[h]{ll}
    $D_N$ & $N(N-1)$ \\
    $E_6$ & 36 roots \\
    $E_7$ & 63 roots \\
    $E_8$ & 120 roots
  \end{tabular}}
  \end{enumerate}
\end{example}
\begin{definition}
  Let $\alpha \in \ZZ^n$ be a positive root. The reflection $s_\alpha$ is defined by the formula
  $$s_\alpha(v) = v - B(v,\alpha) \alpha.$$
\end{definition}
We denote $s_{\alpha_i}$ by $s_i$ and call these \textbf{simple reflections}.
\begin{remark}
  As a linear operator of $\RR^n$, $s_\alpha$ fixes any vector orthogonal to $\alpha$ and
  $$s_\alpha(\alpha) = -\alpha$$
  Therefore $s_\alpha$ is the reflection at the
hyperplane orthogonal to $\alpha$, and in particular fixes $B$.
  The $s_i$ generate a subgroup $W \subseteq O(\RR^n)$, which is
called {\it the Weyl group} of $\Gamma$.
Since for every $w \in W$, $w(\alpha_i)$ is a root, and since there are only finitely many roots, $W$ has to be finite.
\end{remark}

\subsection{Gabriel's theorem}
\begin{definition}
  Let $Q$ be a quiver with any labeling $1, \ldots, n$ of the vertices. Let $V = (V_1, \ldots, V_n)$ be a representation of $Q$. We then call
$$d(V) = (\dim V_1, \ldots, \dim V_n)$$
the dimension vector of this representation.
\end{definition}
We are now able to formulate Gabriel's theorem using roots.
\begin{theorem}[Gabriel's theorem]
  Let $Q$ be a quiver of type $A_n, D_n, E_6, E_7, E_8$. Then $Q$ has finitely many indecomposable representations. Namely, the dimension vector of any indecomposable representation is a positive root (with respect to $B_\Gamma$) and for any positive root $\alpha$ there is exactly one indecomposable representation with dimension vector $\alpha$.
\end{theorem}

\subsection{Reflection Functors}
\begin{definition}
  Let $Q$ be any quiver. We call a vertex $i \in Q$ a sink if all edges connected to $i$ point towards $i$.
  $$\xymatrix{\ar[r] & \bullet \save[]+<0cm,+0.3cm>*{i}\restore & \ar[l] \\ & \ar[u]}$$
  We call a vertex $i \in Q$ a source if all edges connected to $i$ point away from $i$.
  $$\xymatrix{ & \bullet \save[]+<0cm,+0.3cm>*{i}\restore \ar[l]\ar[r]\ar[d]&  \\ &}$$
\end{definition}
\begin{definition}
  Let $Q$ be any quiver and $i \in Q$ be a sink (a source). Then we let $\overline{Q_i}$ be the quiver obtained from $Q$ by reversing all arrows pointing into (pointing out of) $i$.
\end{definition}
We are now able to define the reflection functors (also called \textit{Coxeter functors}).
\begin{definition}
  Let $Q$ be a quiver, $i \in Q$ be a sink. Let $V$ be a representation of $Q$. Then we define the reflection functor
  $$F_i^+ : \Rep Q \to \Rep \overline{Q_i}$$
  by the rule
  $$F^+_i (V)_k = V_k \quad \mbox{if}\; k \neq i$$
  $$F^+_i (V)_i = \ker\left(\varphi: \bigoplus_{j \to i} V_j \to V_i\right).$$
  Also, all maps stay the same but those now pointing out of $i$;
these are replaced by compositions of the inclusion of
$\ker\varphi$ into $\oplus V_j$ with the projections $\oplus
V_j\to V_k$.
\end{definition}
\begin{definition}
  Let $Q$ be a quiver, $i \in Q$ be a source. Let $V$ be a representation of $Q$. Let $\psi$ be the canonical map
  $$\psi : V_i \to \bigoplus_{i \to j} V_j.$$
  Then we define the reflection functor
  $$F^-_i : \Rep Q \to \Rep \overline{Q_i}$$
  by the rule
  $$F^-_i (V)_k = V_k \quad \mbox{if}\; k \neq i$$
  $$F^-_i (V)_i = \mathrm{Coker}\left( \psi \right) = \left(\bigoplus_{i \to j} V_j\right)/\Image \psi.$$
  Again, all maps stay the same but those now pointing into $i$;
these are replaced by the compositions of the inclusions $V_k\to
\oplus_{i\to j}V_j$ with the natural map
$\oplus V_j\to \oplus V_j/{\rm Im}\psi$.
\end{definition}
\begin{proposition}
  Let $Q$ be a quiver, $V$ an indecomposable representation of $Q$.  \label{prop2}
  \begin{enumerate}
  \item Let $i \in Q$ be a sink. Then either $\dim V_i = 1, \, \dim V_j = 0$ for $j \neq i$ \textbf{or}
    $$\varphi: \bigoplus_{j \to i} V_j \to V_i$$
    is surjective.
  \item Let $i \in Q$ be a source. Then either $\dim V_i = 1, \, \dim V_j = 0$ for $j \neq i$ \textbf{or}
    $$\psi: V_i \to \bigoplus_{i \to j} V_j$$
    is injective.
  \end{enumerate}
\end{proposition}
\begin{proof}
  \begin{enumerate}
  \item Choose a complement $W$ of $\Image \varphi$. Then we get
    $$V \quad=\quad
\vcenter{\xymatrix@-=0.5cm{\bullet  \save[]+<0cm,-0.3cm>*{0}\restore \ar[r] & \bullet \save[]+<0cm,0.3cm>*{W}\restore & \bullet \save[]+<0cm,-0.3cm>*{0}\restore \ar[l] \\ & \bullet \save[]+<0cm,-0.3cm>*{0}\restore \ar[u]}} \quad\oplus\quad
V'$$
Since $V$ is indecomposable, one of these summands has to be
zero. If the first summand is zero, then $\varphi$ has to be
surjective. If the second summand is zero, then the first one has to be of the desired form, because else we could write it as a direct sum of several objects of the type
$$\vcenter{\xymatrix@-=0.5cm{\bullet  \save[]+<0cm,-0.3cm>*{0}\restore \ar[r] & \bullet \save[]+<0cm,0.3cm>*{1}\restore & \bullet \save[]+<0cm,-0.3cm>*{0}\restore \ar[l] \\ & \bullet \save[]+<0cm,-0.3cm>*{0}\restore \ar[u]}}$$
which is impossible, since $V$ was supposed to be indecomposable.
\item Follows similarly by splitting away the kernel of $\psi$.
  \end{enumerate}
\end{proof}
\begin{proposition}
  Let $Q$ be a quiver, $V$ be a representation of $Q$. \label{prop3}
  \begin{enumerate}
  \item If
    $$\varphi: \bigoplus_{j \to i} V_j \to V_i$$
    is surjective, then
    $$F_i^-F_i^+ V = V.$$
  \item If
    $$\psi: V_i \to \bigoplus_{i \to j} V_j$$
    is injective, then
    $$F_i^+F_i^- V = V.$$
  \end{enumerate}
\end{proposition}
\begin{proof}
 In the following proof, we will always mean by $i \to j$ that $i$ points into $j$ in the original quiver $Q$.
 We only establish the first statement and we also restrict ourselves to showing that the spaces of $V$ and $F_i^-F_i^+V$ are the same. It is enough to do so for the $i$-th space.
 Let
 $$\varphi : \bigoplus_{j \to i} V_j \to V_i$$
 be surjective and let
 $$K = \ker \varphi.$$
 When applying $F_i^+$, the space $V_i$ gets replaced by $K$.
 Furthermore, let
 $$\psi: K \to \bigoplus_{j \to i} V_j.$$
 After applying $F_i^-$, $K$ gets replaced by
 $$K' = \left( \bigoplus_{j \to i} V_j \right)/(\Image \psi).$$
 But
 $$\Image \psi = K$$
 and therefore
$$
K' = \left( \bigoplus_{j \to i} V_j \right) / \left( \ker(\varphi:
\bigoplus_{j \to i} V_j \to V_i) \right) = \Image (\varphi: \bigoplus_{j \to
i} V_j \to V_i)
$$
 by the homomorphism theorem.
Since $\varphi$ was assumed to be surjective, we get
 $$K' = V_i.$$
\end{proof}
\begin{proposition}
  Let $Q$ be a quiver, and $V$ be an indecomposable representation of
$Q$. Then $F_i^+V$ and $F_i^-V$ (whenever defined) are either indecomposable or 0. \label{prop4}
\end{proposition}
\begin{proof}
  We prove the proposition for $F_i^+V$ - the case $F_i^-V$
follows similarly. By Proposition \ref{prop2} it follows that either
  $$\varphi: \bigoplus_{j \to i} V_j \to V_i$$
  is surjective or $\dim V_i = 1, \dim V_j = 0, \quad j \neq i$. In the last case
  $$F^+_iV = 0.$$
  So we can assume that $\varphi$ is surjective. In this case,
assume that $F_i^+V$ is decomposable as
  $$F_i^+ V = X \oplus Y$$
  with $X, Y \neq 0$.  But $F_i^+V$ is injective at $i$, since
the maps are canonical projections, whose direct sum is the
tautological embedding. Therefore $X$ and $Y$ also have to be injective at $i$ and hence (by \ref{prop3})
  $$F_i^+F_i^- X = X, \quad F_i^+F_i^- Y = Y$$
  In particular
  $$F_i^- X \neq 0, \quad F_i^- Y \neq 0.$$
  Therefore
  $$V = F_i^-F_i^+V = F_i^-X \oplus F_i^-Y$$
  which is a contradiction, since $V$ was assumed to be indecomposable. So we can infer that
  $$F_i^+ V$$
  is indecomposable.
\end{proof}
\begin{proposition}
  Let $Q$ be a quiver and $V$ a representation of $Q$.
  \begin{enumerate}
  \item Let $i \in Q$ be a sink and let $V$ be surjective at $i$. Then
    $$d(F_i^+ V) = s_i(d(V)).$$
  \item Let $i \in Q$ be a source and let $V$ be injective at $i$. Then
    $$d(F_i^- V) = s_i(d(V)).$$
  \end{enumerate}
\end{proposition}
\begin{proof}
  We only prove the first statement, the second one follows similarly. Let $i \in Q$ be a sink and let
  $$\varphi: \bigoplus_{j \to i} V_j \to V_i$$
  be surjective. Let $K = \ker \varphi$. Then
  $$\dim K = \sum_{j \to i} \dim V_j - \dim V_i.$$
  Therefore we get
  $$\left(d(F_i^+V) - d(V)\right)_i = \sum_{j \to i} \dim V_j - 2 \dim V_i = - B\left(d(V), \alpha_i\right)$$
  and
  $$\left(d(F_i^+V) - d(V)\right)_j = 0, \quad j \neq i.$$
  This implies
  $$d(F_i^+V) - d(V) = -B\left(d(V),\alpha_i\right)\alpha_i$$
  $$\Leftrightarrow \quad d(F_i^+V) \;=\; d(V) - B\left(d(V),\alpha_i\right)\alpha_i \;=\; s_i\left(d(V)\right).$$
\end{proof}

\subsection{Coxeter elements}
\begin{definition}
  Let $Q$ be a quiver and let $\Gamma$ be the underlying
graph. Fix any labeling $1, \ldots, n$ 
of the vertices of $\Gamma$. Then the Coxeter element $c$ of $Q$ corresponding to this labeling is defined as
  $$c = s_1s_2\dots s_n.$$
\end{definition}

\begin{lemma}
  Let \label{lemma7}
  $$\beta = \sum_i k_i \alpha_i$$
  with $k_i \geq 0$ for all $i$ but not all $k_i=0$. Then there is $N \in \NN$, such that
  $$c^N\beta$$
  has at least one strictly negative coefficient.
\end{lemma}
\begin{proof}
  $c$ belongs to a finite group $W$. So there is $M \in \NN$, such that
  $$c^M = 1.$$
  We claim that
  $$1 + c + c^2 + \dots + c^{M-1} = 0$$
  as operators on $\RR^n$. This implies what we need, since $\beta$ has at least one strictly positive coefficient, so one of the elements
  $$c\beta, c^2\beta, \ldots, c^{M-1}\beta$$
  must have at least one strictly negative one. Furthermore, it is enough to show that 1 is not an eigenvalue for $c$, since
  $$(1 + c + c^2 + \dots + c^{M-1})v = w \neq 0$$
  $$\Rightarrow \quad c w = c \left(1 + c + c^2 + \dots + c^{M-1}\right) v = (c + c^2 + c^3 + \dots + c^{M-1} + 1) v = w.$$
  Assume the contrary, i.e., 1 is a eigenvalue
of $c$ and let $v$ be a corresponding eigenvector.
  $$cv = v \quad \Rightarrow \quad s_1\dots s_n v = v$$
  $$\Leftrightarrow \quad s_2 \dots s_n v = s_1 v.$$
  But since $s_i$ only changes the $i$-th coordinate of $v$, we get
  $$s_1 v = v \quad \mbox{and} \quad s_2 \dots s_n v = v.$$
  Repeating the same procedure, we get
  $$s_i v = v$$
  for all $i$. But this means
  $$B(v, \alpha_i) = 0.$$
  for all $i$, and since $B$ is nondegenerate, we get $v = 0$. But this is a contradiction, since $v$ is an eigenvector.
\end{proof}

\subsection{Proof of Gabriel's theorem}
Let $V$ be an indecomposable representation of $Q$. We introduce
a fixed labeling $1, \ldots n$ on $Q$, such that $i<j$ if one can
reach $j$ from $i$. This is possible, since we can assign the
highest label to any sink, remove this sink from the quiver,
assign the next highest label to a sink of the remaining quiver
and so on. This way we create a labeling of the desired kind.

We now consider the sequence
$$V^{(0)} = V, \; V^{(1)} = F^+_n V, \; V^{(2)} = F^+_{n-1} F^+_n V, \ldots$$
This sequence is well defined because of the selected labeling:
$n$ has to be a sink of $Q$, $n-1$ has to be a sink of
$\overline{Q_n}$ (where $\overline{Q}_n$ is obtained from $Q$ by
reversing all the arrows at the vertex $r$) and so on. Furthermore, we note that $V^{(n)}$ is a representation of $Q$ again, since every arrow has been reversed twice (since we applied a reflection functor to every vertex). This implies that we can define
$$V^{(n+1)} = F^+_n V^{(n)}, \ldots$$
and continue the sequence to infinity.
\begin{theorem}
  There is $m \in \NN$, such that \label{theorem7}
  $$d\left(V^{(m)}\right) = \alpha_p$$
  for some $p$.
\end{theorem}
\begin{proof}
  If $V^{(i)}$ is surjective at the appropriate vertex $k$, then
  $$d\left(V^{(i+1)}\right) = d\left(F_k^+ V^{(i)}\right) = s_k d\left(V^{(i)}\right).$$
  This implies, that if $V^{(0)}, \ldots, V^{(i-1)}$ are surjective at the appropriate vertices, then
  $$d\left(V^{(i)}\right) =
\dots s_{n-1}s_n d(V).$$
  By Lemma \ref{lemma7} this cannot continue indefinitely - since $d\left(V^{(i)}\right)$ may not have any negative entries.
  Let $i$ be smallest number such that $V^{(i)}$ is not
surjective at the appropriate vertex. By Proposition
\ref{prop4} it is indecomposable. So, by Proposition \ref{prop2}, we get
  $$d(V^{(i)}) = \alpha_p$$
  for some $p$.
\end{proof}
We are now able to prove Gabriel's theorem.
Namely, we get the following corollaries.
\begin{corollary}
  Let $Q$ be a quiver, $V$ be any indecomposable representation. Then $d(V)$ is a positive root.
\end{corollary}
\begin{proof}
  By Theorem \ref{theorem7}
  $$s_{i_1} \dots s_{i_m}\left(d(V)\right) = \alpha_p.$$
  Since the $s_i$ preserve $B$, we get
  $$B(d(V),d(V)) = B(\alpha_p, \alpha_p) = 2.$$
\end{proof}
\begin{corollary}
  Let $V, V'$ be indecomposable representations of $Q$ such that $d(V) = d(V')$. Then $V$ and $V'$ are isomorphic.
\end{corollary}
\begin{proof}
  Let $i$ be such that
  $$d\left(V^{(i)}\right) = \alpha_p.$$
  Then we also get $d\left(V'^{(i)}\right) = \alpha_p$. So
  $$V'^{(i)} = V^{(i)} =: V^i.$$
  Furthermore we have
  $$V^{(i)} = F^+_k \dots F_{n-1}^+ F_n^+ V^{(0)}$$
  $$V'^{(i)} = F^+_k \dots F_{n-1}^+ F_n^+ V'^{(0)}.$$
  But both $V^{(i-1)}, \ldots, V^{(0)}$ and $V'^{(i-1)}, \ldots, V'^{(0)}$ have to be surjective at the appropriate vertices. This implies
  $$F^-_{n} F^-_{n-1} \dots F^-_k  V^i =
  \left\{\begin{array}[h]{lll}
      F^-_{n} F^-_{n-1} \dots F^-_k F^+_k \dots F_{n-1}^+ F_n^+ V^{(0)} &= V^{(0)} &= V \\
      F^-_{n} F^-_{n-1} \dots F^-_k F^+_k \dots F_{n-1}^+ F_n^+ V'^{(0)} &= V'^{(0)} &= V'
  \end{array}\right.$$
\end{proof}
These two corollaries show that there are only finitely many indecomposable representations (since there are only finitely many roots) and that the dimension vector of each of them is a positive root. The last statement of Gabriel's theorem follows from
\begin{corollary}
  For every positive root $\alpha$, there is an indecomposable
representation $V$ with
  $$d(V) = \alpha.$$
\end{corollary}
\begin{proof}
  Consider the sequence
  $$s_n\alpha, s_{n-1}s_n\alpha, \ldots$$
  Consider the first element of this sequence which is a negative
root (this has to happen by Lemma \ref{lemma7}) and look at one
step before that, calling
this element $\beta$. So $\beta$ is a positive root and $s_i\beta$ is a negative root for some $i$.
  But since the $s_i$ only change one coordinate, we get
  $$\beta = \alpha_i$$
  and
  $$(s_q \dots s_{n-1}s_n) \alpha = \alpha_i.$$
  We let $\CC_{(i)}$ be the representation having dimension vector $\alpha_i$. Then we define
  $$V = F^-_nF^-_{n-1} \dots F^-_q \CC_{(i)}.$$
  This is an indecomposable representation and
  $$d(V) = \alpha.$$
\end{proof}

\begin{example}
Let us demonstrate by example how reflection functors work.
Consider the quiver $D_4$ with the orientation of all arrows
towards the node (which is labeled by $4$). Start with the
1-dimensional representation $V_{\alpha_4}$ sitting at the 4-th vertex.
Apply to $V_{\alpha_4}$ the functor $F_3^-F_2^-F_1^-$. This yields
$$
F_1^-F_2^-F_3^-V_{\alpha_4}=V_{\alpha_1+\alpha_2+\alpha_3+\alpha_4}.
$$
Now applying $F_4^-$ we get
$$
F_4^-F_1^-F_2^-F_3^-V_{\alpha_4}=V_{\alpha_1+\alpha_2+\alpha_3+2\alpha_4}.
$$
Note that this is exactly the inclusion of 3 lines into the
plane, which is the most complicated indecomposable
representation of the $D_4$ quiver.
\end{example}

\subsection{Problems}

\begin{problem} Let $Q_n$ be the cyclic quiver of length $n$,
i.e., $n$ vertices connected by $n$ oriented edges forming a
cycle. Obviously, the classification of
indecomposable representations of $Q_1$ is given by the Jordan
normal form theorem. Obtain a similar classification
of indecomposable representations of $Q_2$. In other words,
classify pairs of linear operators $A: V\to W$
and $B:W\to V$ up to isomorphism. Namely:

(a) Consider the following pairs (for $n\ge 1$):

1) $E_{n,\lambda}$: $V=W=\Bbb C^n$, $A$ is the Jordan block of
size $n$ with eigenvalue $\lambda$, $B=1$ ($\lambda\in \Bbb C$).

2) $E_{n,\infty}$: is obtained from $E_{n,0}$ by exchanging $V$
with $W$ and $A$ with $B$.

3) $H_n$: $V=\Bbb C^n$ with basis $v_i$, $W=\Bbb C^{n-1}$ with
basis $w_i$, $Av_i=w_i$, $Bw_i=v_{i+1}$ for $i<n$, and $Av_n=0$.

4) $K_n$ is obtained from $H_n$ by exchanging $V$
with $W$ and $A$ with $B$.

Show that these are indecomposable and pairwise nonisomorphic.

(b) Show that if $E$ is a representation of $Q_2$ such that
$AB$ is not nilpotent, then $E=E'\oplus E''$, where
$E''=E_{n,\lambda}$ for some $\lambda\ne 0$.

(c) Consider the case when $AB$ is nilpotent, and consider
the operator $X$ on $V\oplus W$ given by $X(v,w)=(Bw,Av)$.
Show that $X$ is nilpotent, and admits a basis consisting
of chains (i.e., sequences $u,Xu,X^2u,...X^{l-1}u$ where $X^lu=0$)
which are compatible with the direct sum decomposition
(i.e., for every chain $u\in V$ or $u\in W$). Deduce that
(1)-(4) are the only indecomposable representations of $Q_2$.

(d)(harder!) generalize this classification to the Kronecker quiver, which
has two vertices $1$ and $2$ and two edges both going from $1$ to
$2$.

(e)(still harder!) can you generalize this classification to $Q_n$,
$n>2$, with any orientation?
\end{problem}

\begin{problem}
Let $L\subset \frac{1}{2}\Bbb Z^8$ be the lattice
of vectors where the coordinates are either all integers or all
half-integers (but not integers), and the sum of all coordinates
is an even integer.

(a) Let $\alpha_i=e_i-e_{i+1}$, $i=1,...,6$, $\alpha_7=e_6+e_7$,
$\alpha_8=-1/2\sum_{i=1}^8 e_i$. Show that $\alpha_i$ are a basis
of $L$ (over $\Bbb Z$).

(b) Show that roots in $L$ (under the usual inner
product) form a root system of type $E_8$
(compute the inner products of $\alpha_i$).

(c) Show that the $E_7$ and $E_6$ lattices can be obtained as the
sets of vectors in the $E_8$ lattice $L$ where the
first two, respectively three, coordinates (in the basis $e_i$)
are equal.

(d) Show that $E_6,E_7,E_8$ have 72,126,240 roots, respectively
(enumerate types of roots in terms of the presentations in the basis
$e_i$, and count the roots of each type).
\end{problem}

\begin{problem}
Let $V_\alpha$ be the indecomposable representation of a Dynkin
quiver $Q$ which corresponds to a positive root $\alpha$.
For instance, if $\alpha_i$ is a simple root, then
$V_{\alpha_i}$ has a 1-dimensional space at $i$ and 0 everywhere
else.

(a) Show that if $i$ is a source then $\Ext(V,V_{\alpha_i})=0$
for any representation $V$ of $Q$, and if $i$ is a sink, then
$\Ext(V_{\alpha_i},V)=0$.

(b) Given an orientation of the quiver, find a Jordan-H\"older
series of $V_\alpha$ for that orientation.
\end{problem}

\newpage \section{Introduction to categories}

\subsection{The definition of a category}

We have now seen many examples of representation theories
and of operations with representations (direct
sum, tensor product, induction, restriction, reflection functors,
etc.) A context in which one can systematically talk about this
is provided by {\bf Category Theory}.

Category theory was founded by Saunders MacLane and Samuel Eilenberg
around 1940. It is a fairly abstract theory which seemingly has no
content, for which reason it was christened ``abstract
nonsense''. Nevertheless, it is a very flexible and powerful language,
which has become totally indispensable in many areas of mathematics,
such as algebraic geometry, topology, representation theory, and
many others.

We will now give a very short introduction to Category theory,
highlighting its relevance to the topics in representation theory
we have discussed. For a serious acquaintance with category
theory, the reader should use the classical book \cite{McL}.

\begin{definition}
A category ${\mathcal C}$ is the following data:

(i) a class of objects $Ob({\mathcal C})$;

(ii) for every objects $X,Y\in Ob({\mathcal C})$, the class $\Hom_{\mathcal C}(X,Y)=\Hom(X,Y)$ of
morphisms (or arrows) from $X,Y$ (for $f\in \Hom(X,Y)$, one
may write $f: X\to Y$);

(iii) For any objects $X,Y,Z\in Ob({\mathcal C})$, a composition map
$\Hom(Y,Z)\times \Hom(X,Y)\to \Hom(X,Z)$, $(f,g)\mapsto f\circ g$,

which satisfy the following axioms:

1. The composition is associative, i.e., $(f\circ g)\circ h=f\circ
(g\circ h)$;

2. For each $X\in Ob({\mathcal C})$, there is a morphism $1_X\in \Hom(X,X)$,
called the unit morphism, such that $1_X\circ f=f$ and $g\circ
1_X=g$ for any $f,g$ for which compositions make sense.
\end{definition}

{\bf Remark.} We will write $X\in {\mathcal C}$ instead of $X\in
Ob({\mathcal C})$.

\begin{example}
1. The category ${\bf Sets}$ of sets (morphisms are arbitrary
maps).

2. The categories ${\bf Groups}$, ${\bf Rings}$
(morphisms are homomorphisms).

3. The category ${\bf Vect}_k$ of vector spaces over a field $k$
(morphisms are linear maps).

4. The category ${\rm Rep}(A)$ of representations of an algebra $A$
(morphisms are homomorphisms of representations).

5. The category of topological spaces (morphisms are
continuous maps).

6. The homotopy category of topological spaces (morphisms are homotopy
classes of continuous maps).
\end{example}

{\bf Important remark.} Unfortunately, one cannot simplify this
definition by replacing the word ``class'' by the much more
familiar word ``set''. Indeed, this would rule out the important Example 1, as
it is well known that there is no set of all sets, and working
with such a set leads to contradictions. The precise definition of
a class and the precise distinction between a class and a set is
the subject of set theory, and cannot be discussed here. Luckily,
for most practical purposes (in particular, in these notes),
this distinction is not essential.

We also mention that in many examples, including examples 1-6,
the word ``class'' in (ii) can be replaced by ``set''.
Categories with this property (that $\Hom(X,Y)$ is a set for any
$X,Y$) are called locally small; many categories that we
encounter are of this kind.

Sometimes the collection $\Hom(X,Y)$ of morphisms from $X$ to $Y$ 
in a given locally small category ${\mathcal C}$ is not just a set but has some additional structure (say, the structure of an abelian group, 
or a vector space over some field). In this case 
one says that ${\mathcal C}$ is {\bf enriched} over another
category ${\mathcal D}$ (which is a {\it monoidal} category, 
i.e., has a product operation 
and a unit object under this product, e.g. the category 
of abelian groups or vector spaces with the tensor product operation). 
This means that for each 
$X,Y\in {\mathcal C}$, $\Hom(X,Y)$ is an object of ${\mathcal
D}$, and the composition $\Hom(Y,Z)\times \Hom(X,Y)\to \Hom(X,Z)$
is a morphism in ${\mathcal{D}}$. 
E.g., if ${\mathcal D}$ is the category of vector spaces, this means that 
the composition is bilinear, i.e. gives rise to a linear map 
$\Hom(Y,Z)\otimes \Hom(X,Y)\to \Hom(X,Z)$.
For a more detailed discussion of this, we refer the reader to \cite{McL}. 

{\bf Example.} The category $\Rep(A)$ of representations of a $k$-algebra $A$ is enriched over the category of 
$k$-vector spaces. 

\begin{definition}
A full subcategory of a category ${\mathcal C}$ is a category ${\mathcal C}'$ whose
objects are a subclass of objects of ${\mathcal C}$, and
$\Hom_{{\mathcal C}'}(X,Y)=\Hom_{{\mathcal C}}(X,Y)$.
\end{definition}

{\bf Example.} The category ${\bf AbelianGroups}$ is a full
subcategory of the category ${\bf Groups}$.

\subsection{Functors}

We would like to define arrows between categories. Such arrows
are called {\bf functors}.

\begin{definition}
A functor $F: {\mathcal C}\to {\mathcal D}$ between categories ${\mathcal C}$ and ${\mathcal D}$ is

(i) a map $F: Ob({\mathcal C})\to Ob({\mathcal D})$;

(ii) for each $X,Y\in {\mathcal C}$, a map 
$F=F_{X,Y}: \Hom(X,Y)\to \Hom(F(X),F(Y))$ which preserves compositions
and identity morphisms.
\end{definition}

Note that functors can be composed in an obvious way.
Also, any category has the
identity functor.

\begin{example}
1. A (locally small)
category ${\mathcal C}$ with one object $X$
is the same thing as a monoid. A functor between such
categories is a homomorphism of monoids.

2. Forgetful functors ${\bf Groups}\to {\bf Sets}$,
${\bf Rings}\to {\bf AbelianGroups}$.

3. The opposite category of a given category is the same category
with the order of arrows and compositions reversed. Then
$V\mapsto V^*$ is a functor ${\bf Vect}_k\mapsto {\bf Vect}_k^{op}$.

4. The Hom functors: If ${\mathcal C}$ is a locally small category then
we have the functor ${\mathcal C}\to {\bf Sets}$ given by
$Y\mapsto \Hom(X,Y)$ and ${\mathcal C}^{op}\to {\bf Sets}$
given by $Y\mapsto \Hom(Y,X)$.

5. The assignment $X\mapsto {\rm Fun}(X,\Bbb Z)$ is a functor
${\bf Sets}\to {\bf Rings}^{op}$.

6. Let $Q$ be a quiver. Consider the category ${\mathcal C}(Q)$ whose
objects are the vertices and morphisms are oriented paths between
them. Then functors from ${\mathcal C}(Q)$ to ${\bf Vect}_k$ are
representations of $Q$ over $k$.

7. Let $K\subset G$ be groups. Then we have the induction functor
${\rm Ind}_K^G: \Rep(K)\to \Rep(G)$, and ${\rm Res}^G_K:
\Rep(G)\to \Rep(K)$.

8. We have an obvious notion of the Cartesian product of
categories (obtained by taking the Cartesian
products of the classes of objects and morphisms of the
factors). The functors of direct sum and tensor product
are then functors ${\bf Vect}_k\times {\bf Vect}_k\to {\bf
Vect}_k$. Also the operations $V\mapsto V^{\otimes n}$,
$V\mapsto S^nV$, $V\mapsto \wedge^n V$ are functors on ${\bf
Vect}_k$. More generally, if $\pi$ is a representation
of $S_n$, we have functors $V\mapsto \Hom_{S_n}(\pi,V^{\otimes n})$.
Such functors (for irreducible $\pi$) are called the Schur
functors. They are labeled by Young diagrams.

9. The reflection functors $F_i^\pm: \Rep(Q)\to \Rep(\bar Q_i)$ are
functors between representation categories of quivers.
\end{example}

\subsection{Morphisms of functors}

One of the important features of functors between categories
which distinguishes them from usual maps or functions is that the
functors between two given categories themselves form a category,
i.e., one can define a nontrivial notion of a morphism between two
functors.

\begin{definition}
Let ${\mathcal C},{\mathcal D}$ be categories and $F,G: {\mathcal C}\to {\mathcal D}$ be functors between
them. A morphism $a: F\to G$ (also called a natural
transformation or a functorial morphism)
is a collection of morphisms $a_X: F(X)\to G(X)$
labeled by the objects $X$ of ${\mathcal C}$, which is {\it functorial} in
$X$, i.e., for any morphism $f: X\to Y$ (for $X,Y\in {\mathcal C}$) one has
$a_Y \circ F(f)=G(f)\circ a_X$.
\end{definition}

A morphism $a: F\to G$ is an isomorphism if there is another
morphism $a^{-1}: G\to F$ such that $a\circ a^{-1}$ and
$a^{-1}\circ a$ are the identities. The set of morphisms
from $F$ to $G$ is denoted by $\Hom(F,G)$.

\begin{example}
1. Let ${\bf FVect}_k$ be the category of finite dimensional vector
spaces over $k$. Then the functors ${\rm id}$ and $**$ on this
category are isomorphic. The isomorphism is defined by the
standard maps $a_V:V\to V^{**}$ given by $a_V(u)(f)=f(u)$, $u\in V$, $f\in
V^*$. But these two functors are not isomorphic on the category
of all vector spaces ${\bf Vect}_k$, since for an infinite
dimensional vector space $V$, $V$ is not isomorphic to $V^{**}$.

2. Let ${\bf FVect}_k'$ be the category of finite dimensional
$k$-vector spaces, where the morphisms are the isomorphisms.
We have a functor $F$ from this category to itself sending any space
$V$ to $V^*$ and any morphism $a$ to $(a^*)^{-1}$.
This functor satisfies the property that $V$ is isomorphic to
$F(V)$ for any $V$, but it is not isomorphic to the identity
functor. This is because the isomorphism $V\to F(V)=V^*$
cannot be chosen to be compatible with the action of $GL(V)$, as $V$
is not isomorphic to $V^*$ as a representation of $GL(V)$.

3. Let $A$ be an algebra over a field $k$,
and $F: A-{\bf mod}\to {\bf Vect}_k$ be the forgetful functor.
Then as follows from Problem \ref{1:3},
${\rm End}F=\Hom(F,F)=A$.

4. The set of endomorphisms of the identity functor
on the category $A-{\bf mod}$ is the center of $A$ (check it!).
\end{example}

\subsection{Equivalence of categories}

When two algebraic or geometric objects are isomorphic, it is
usually not a good idea to say that they are equal
(i.e., literally the same). The reason is that such objects are
usually equal in many different ways, i.e., there are many ways
to pick an isomorphism, but by saying that the objects are equal
we are misleading the reader or listener into thinking that we
are providing a certain choice of the identification, which we
actually do not do. A vivid example of this is a finite
dimensional vector space $V$ and its dual space $V^*$.

For this reason in category theory, one
most of the time tries to avoid saying that two objects or
two functors are equal. In particular, this applies to the
definition of isomorphism of categories.

Namely, the naive notion of isomorphism of categories is defined
in the obvious way: a functor $F:{\mathcal C}\to {\mathcal D}$ is an isomorphism
if there exists $F^{-1}: {\mathcal D}\to {\mathcal C}$ such that $F\circ F^{-1}$
and $F^{-1}\circ F$ are equal to the identity functors.
But this definition is not very useful. We might suspect so since
we have used the word ``equal'' for objects of a category
(namely, functors) which we are not supposed to
do. And in fact here is an example of two categories which
are ``the same for all practical purposes'' but are not isomorphic;
it demonstrates the deficiency of our definition.

Namely, let ${\mathcal C}_1$ be the simplest possible category:
$Ob({\mathcal C_1})$ consists of
one object $X$, with $\Hom(X,X)=\lbrace{1_X\rbrace}$.
Also, let ${\mathcal C}_2$ have two objects $X,Y$ and 4 morphisms:
$1_X,1_Y,a: X\to Y$ and $b: Y\to X$. So we must have $a\circ
b=1_Y$, $b\circ a=1_X$.

It is easy to check that for any category ${\mathcal D}$, there is a natural
bijection between the collections of
isomorphism classes of functors ${\mathcal C}_1\to {\mathcal D}$ and
${\mathcal C}_2\to {\mathcal D}$ (both are identified with the collection of
isomorphism classes of objects of ${\mathcal D}$). This is what we mean
by saying that ${\mathcal C}_1$ and ${\mathcal C}_2$ are ``the same for all practical
purposes''. Nevertheless they are not isomorphic, since ${\mathcal C}_1$ has
one object, and ${\mathcal C}_2$ has two objects (even though these two
objects are isomorphic to each other).

This shows that we should adopt a more flexible and less
restrictive notion of isomorphism of categories.
This is accomplished by the definition of an {\bf equivalence of
categories.}

\begin{definition}
A functor $F:{\mathcal C}\to {\mathcal D}$ is an equivalence of categories
if there exists $F': {\mathcal D}\to {\mathcal C}$ such that $F\circ F'$
and $F'\circ F$ are {\bf isomorphic} to the identity functors.
\end{definition}

In this situation, $F'$ is said to be {\it a quasi-inverse} to $F$.

In particular, the above categories ${\mathcal C}_1$ and ${\mathcal C}_2$ are
equivalent (check it!).

Also, the category ${\bf FSet}$ of finite
sets is equivalent to the category whose objects are nonnegative
integers, and morphisms are given by
$\Hom(m,n)={\rm Maps}(\lbrace{1,...,m\rbrace},
\lbrace{1,...,n\rbrace})$. Are these categories isomorphic?
The answer to this question depends on whether you believe that
there is only one finite set with a given number of elements, or that
there are many of those. It seems better to think that there are
many (without asking ``how many''), so that isomorphic sets need
not be literally equal, but this is really a matter of choice.
In any case, this is not really a reasonable question;
the answer to this question is irrelevant for any
practical purpose, and thinking about it will give you nothing but
a headache.

\subsection{Representable functors}

A fundamental notion in category theory is that of a {\bf
representable functor}. Namely, let ${\mathcal C}$ be a (locally small) category, and $F:
{\mathcal C}\to {\bf Sets}$ be a functor. We say that $F$ is {\bf
representable} if there exists an object $X\in {\mathcal C}$ such that $F$
is isomorphic to the functor $\Hom(X,?)$. More precisely, if we are given such an object $X$, 
together with an isomorphism $\xi: F\cong \Hom(X,?)$, we say that the functor $F$ is {\bf represented by} $X$ (using $\xi$).  

In a similar way, one can talk about representable functors from
${\mathcal C}^{op}$ to ${\bf Sets}$. Namely, one calls such a functor
representable if it is of the form $\Hom(?,X)$ for some object
$X\in {\mathcal C}$, up to an isomorphism.

Not every functor is representable, but if a representing object
$X$ exists, then it is unique. Namely, we have the following
lemma.

\begin{lemma} (The Yoneda Lemma)
If a functor $F$ is represented by an object $X$, then $X$ is unique up to a unique isomorphism.
I.e., if $X,Y$ are two objects in ${\mathcal C}$, then for any isomorphism of functors
$\phi: \Hom(X,?)\to \Hom(Y,?)$ there is a unique isomorphism $a_\phi:
X\to Y$ inducing $\phi$.
\end{lemma}

\begin{proof} (Sketch) One sets $a_\phi=\phi_Y^{-1}(1_Y)$, and
shows that it is invertible by constructing the inverse, which is
$a_\phi^{-1}=\phi_X(1_X)$. It remains to show that the composition
both ways is the identity, which we will omit here.
This establishes the existence of $a_\phi$. Its uniqueness
is verified in a straightforward manner.
\end{proof}

{\bf Remark.} In a similar way, if a category ${\mathcal C}$ is enriched over 
another category ${\mathcal D}$ (say, the category of abelian groups or vector spaces), 
one can define the notion of a representable functor from ${\mathcal C}$ to ${\mathcal D}$. 

\begin{example}
Let $A$ be an algebra. Then the forgetful functor to vector spaces on
the category of left $A$-modules is representable, and the representing object is the
free rank $1$ module (=the regular representation) $M=A$.
But if $A$ is infinite dimensional, and we restrict attention to
the category of finite dimensional modules, then the forgetful
functor, in general, is not representable (this is so, for
example, if $A$ is the algebra of complex functions on $\Bbb Z$ 
which are zero at all points but finitely many).
\end{example}

\subsection{Adjoint functors}

Another fundamental notion in category theory is the notion of {\bf
adjoint functors}.

\begin{definition}
Functors $F: {\mathcal C}\to {\mathcal D}$ and $G: {\mathcal D}\to {\mathcal C}$ are said to be a pair of
adjoint functors if for any $X\in {\mathcal C}$, $Y\in {\mathcal D}$ we are given an
isomorphism $\xi_{XY}: \Hom_{\mathcal C}(F(X),Y)\to \Hom_{\mathcal D}(X,G(Y))$ which is
functorial in $X$ and $Y$; in other words, if we are given an
isomorphism of functors $\Hom(F(?),?)\to \Hom(?,G(?))$
(${\mathcal C}\times {\mathcal D}\to {\bf Sets}$). In this situation, we say that
$F$ is left adjoint to $G$ and $G$ is right adjoint to $F$.
\end{definition}

Not every functor has a left or right adjoint, but
if it does, it is unique
and can be constructed canonically (i.e., if we somehow found
two such functors, then there is a canonical isomorphism between
them). This follows easily from the Yoneda lemma, as if
$F,G$ are a pair of adjoint functors then $F(X)$ represents the
functor $Y\mapsto \Hom(X,G(Y))$, and $G(Y)$ represents the functor
$X\mapsto \Hom(F(X),Y)$.

\begin{remark} The terminology ``left and right adjoint functors'' 
is motivated by the analogy between categories
and inner product spaces. More specifically, 
we have the following useful dictionary between category theory and linear algebra, 
which helps understand better many notions of category theory. 

\begin{tabular}{|l|l|}
\hline
\multicolumn{2}{|c|}{Dictionary between category theory and linear algebra} \\
\hline
 Category ${\mathcal C}$ & Vector space $V$ with a nondegenerate inner product \\
 The set of morphisms $\Hom(X,Y)$ & Inner product $(x,y)$ on $V$ (maybe nonsymmetric) \\
Opposite category ${\mathcal C}^{op}$ & Same space $V$ with reversed inner product \\
The category ${\bf Sets}$ & The ground field $k$ \\
Full subcategory in ${\mathcal C}$ & Nondegenerate subspace in $V$\\
Functor $F: {\mathcal C}\to {\mathcal D}$ & Linear operator $f: V\to W$ \\
Functor $F: {\mathcal C}\to {\bf Sets}$ & Linear functional $f\in V^*=\Hom(V,k)$ \\
Representable functor & Linear functional $f\in V^*$ given by $f(v)=(u,v)$, $u\in V$ \\
Yoneda lemma & Nondegeneracy of the inner product (on both sides) \\
Not all functors are representable & If $\dim V=\infty$, not $\forall f\in V^*$, $f(v)=(u,v)$\\
Left and right adjoint functors & Left and right adjoint operators \\
Adjoint functors don't always exist & Adjoint operators may not exist if $\dim V=\infty$\\ 
If they do, they are unique & If they do, they are unique\\
Left and right adjoints may not coincide & The inner product may be nonsymmetric \\
\hline
\end{tabular}

\end{remark}

\begin{example}
1. Let $V$ be a finite dimensional representation
of a group $G$ or a Lie algebra ${\mathfrak{g}}$. Then the left and right adjoint to
the functor $V\otimes$ on the category of representations of $G$
is the functor $V^*\otimes$.

2. The functor ${\rm Res}_K^G$ is left adjoint to ${\rm Ind}_K^G$.
This is nothing but the statement of the Frobenius reciprocity.
\end{example}

3. Let ${\bf Assoc}_k$ be the category of associative unital
algebras, and ${\bf Lie}_k$ the category of Lie algebras
over some field $k$. We have a functor $L: {\bf Assoc}_k\to
{\bf Lie}_k$, which attaches to an associative algebra the same
space regarded as a Lie algebra, with bracket $[a,b]=ab-ba$.
Then the functor $L$ has a left adjoint, which is the functor $U$
of taking the universal enveloping algebra of a Lie algebra.

4. We have the functor $GL_1: {\bf Assoc}_k\to {\bf Groups}$,
given by $A\mapsto GL_1(A)=A^\times$. This functor has a left
adjoint, which is the functor $G\mapsto k[G]$, the group algebra of
$G$.

5. The left adjoint to the forgetful functor
${\bf Assoc}_k\to {\bf Vect}_k$ is the functor of tensor algebra:
$V\mapsto TV$. Also, if we denote by ${\bf Comm}_k$ the category of
commutative algebras, then the left adjoint to the forgetful functor
${\bf Comm}_k\to {\bf Vect}_k$ is the functor of the symmetric algebra:
$V\mapsto SV$.

One can give many more examples, spanning many fields.
These examples show that adjoint functors are ubiquitous in
mathematics.

\subsection{Abelian categories}

The type of categories that most often appears in representation
theory is {\bf abelian categories}. The standard definition of an abelian
category is rather long, so we will not give it here, referring the reader 
to the textbook \cite{Fr}; rather, we
will use as the definition what is really the statement of
the Freyd-Mitchell theorem:

\begin{definition}\label{ab}
An abelian category is a category (enriched over the category of abelian groups), which 
is equivalent to a full subcategory ${\mathcal C}$ of the category
$A$-mod of left modules over a ring $A$, closed under taking finite direct sums,
as well as kernels, cokernels, and images of morphisms.
\end{definition}

We see from this definition that in an abelian category,
$\Hom(X,Y)$ is an abelian group for each $X,Y$, compositions are
group homomorphisms with respect to each argument, there is the
zero object, the notion of an injective morphism (monomorphism)
and surjective morphism (epimorphism), and every
morphism has a kernel, a cokernel, and an image.

\begin{example}\label{amod}
The category of modules over an algebra $A$ and the category of
finite dimensional modules over $A$ are abelian categories.
\end{example}

\begin{remark}
The good thing about Definition \ref{ab} is that it
allows us to visualize objects, morphisms, kernels, and cokernels
in terms of classical algebra. But the definition also has a big
drawback, which is that even if ${\mathcal C}$ is the whole category
$A$-mod, the ring $A$ is not
determined by ${\mathcal C}$. In particular, two
different rings can have equivalent categories
of modules (such rings are called {\bf Morita equivalent}).
Actually, it is worse than that:
for many important abelian categories there is no natural (or
even manageable) ring $A$ at all.
 This is why people prefer to use the standard
definition, which is free from this drawback, even though it is
more abstract.
\end{remark}

We say that an abelian category ${\mathcal C}$ is $k$-{\bf linear} if
the groups $\Hom_{\mathcal C}(X,Y)$ are equipped with a structure of a
vector space over $k$, and composition maps are $k$-linear in
each argument. In particular, the categories in Example
\ref{amod} are $k$-linear.

\subsection{Exact functors}

\begin{definition}
A sequence of objects and morphisms
$$
X_0\to X_1\to...\to X_{n+1}
$$
in an abelian category is said to be {\bf a complex}
if the composition of any two consecutive arrows is zero.
The {\bf cohomology} of this complex is
$H^i=\Ker(d_i)/{\rm Im}(d_{i-1})$, where $d_i: X_{i}\to X_{i+1}$
(thus the cohomology is defined for $1\le i\le n$).
The complex is said to be {\bf exact in the $i$-th term} if $H^i=0$, and is
said to be {\bf an exact sequence} if it is exact in all terms.
A {\bf short exact sequence} is an exact sequence of the form
$$
0\to X\to Y\to Z\to 0.
$$
\end{definition}

Clearly, $0\to X\to Y\to Z\to 0$ is a short exact sequence if and only if
$X\to Y$ is injective, $Y\to Z$ is surjective, and the
induced map $Y/X\to Z$ is an isomorphism.

\begin{definition}
A functor $F$ between two abelian categories is {\bf additive} if
it induces homomorphisms on $\Hom$ groups. Also, for $k$-linear
categories one says that $F$ is $k$-linear if it induces
$k$-linear maps between $\Hom$ spaces.
\end{definition}

It is easy to show that if $F$ is an additive functor, then
$F(X\oplus Y)$ is canonically isomorphic to $F(X)\oplus F(Y)$.

\begin{example} The functors $\Ind_K^G$, ${\rm Res}_K^G$,
$\Hom_G(V,?)$ in the theory of group representations over a field
$k$ are additive and $k$-linear.
\end{example}

\begin{definition}
An additive functor $F: {\mathcal C}\to {\mathcal D}$ between abelian categories
is {\bf left exact} if for any exact sequence
$$
0\to X\to Y\to Z,
$$
the sequence
$$
0\to F(X)\to F(Y)\to F(Z)
$$
is exact. $F$ is {\bf right exact}
if for any exact sequence
$$
X\to Y\to Z\to 0,
$$
the sequence
$$
F(X)\to F(Y)\to F(Z)\to 0
$$
is exact. $F$ is {\bf exact} if it is both left and right exact.
\end{definition}

\begin{definition} An abelian category ${\mathcal C}$ is {\bf
semisimple} if any
short exact sequence in this category splits, i.e., is isomorphic to a sequence
$$
0\to X\to X\oplus Y\to Y\to 0
$$
(where the maps are obvious).
\end{definition}

\begin{example}
The category of representations of a finite group $G$ over a
field of characteristic not dividing $|G|$ (or 0) is semisimple.
\end{example}

Note that in a semisimple category, any additive functor is
automatically exact on both sides.

\begin{example}
(i) The functors $\Ind_K^G$, ${\rm Res}_K^G$ are exact.

(ii) The functor $\Hom(X,?)$ is left exact, but not necessarily
right exact. To see that it need not be right exact, it suffices
to consider the exact sequence
$$
0\to \Bbb Z\to \Bbb Z\to \Bbb Z/2\Bbb Z\to 0,
$$
and apply the functor $\Hom(\Bbb Z/2\Bbb Z,?)$.

(iii) The functor $X\otimes_A$ for a right $A$-module $X$
(on the category of left $A$-modules) is right exact, but
not necessarily left exact. To see this, it suffices to tensor multiply
the above exact sequence by $\Bbb Z/2\Bbb Z$.
\end{example}

{\bf Exercise.} 
Show that if $(F,G)$ is a pair of adjoint additive functors 
between abelian categories, then $F$ is right exact and $G$ is
left exact. 

{\bf Exercise.}
(a) Let $Q$ be a quiver and $i\in Q$ a source. Let $V$ be a
representation of $Q$, and $W$ a representation of
$\overline{Q_i}$ (the quiver obtained from $Q$ by reversing
arrows at the vertex $i$). Prove that there is a natural isomorphism
between $\mathrm{Hom}\left(F_i^-V,W\right)$ and
$\mathrm{Hom}\left(V,F_i^+W\right)$. In other words, the functor
$F_i^+$ is right adjoint to $F_i^-$. 

(b) Deduce that the functor $F_i^+$ is left exact, and $F_i^-$ is
right exact. 

\newpage \section{Structure of finite dimensional algebras}

In this section we return to studying the structure of finite
dimensional algebras. Throughout the section, we work over an
algebraically closed field $k$ (of any characteristic).

\subsection{Projective modules}

Let $A$ be an algebra, and $P$ be a left $A$-module.

\begin{theorem}\label{proje}
The following properties of $P$ are equivalent:

(i) If $\alpha: M\to N$ is a surjective morphism, and
$\nu: P\to N$ any morphism, then there exists a morphism $\mu: P\to M$ such
that $\alpha\circ \mu=\nu$.

(ii) Any surjective morphism $\alpha: M\to P$ splits, i.e., there
exists $\mu: P\to M$ such that $\alpha\circ \mu={\rm id}$.

(iii) There exists another $A$-module $Q$ such that $P\oplus Q$
is a free $A$-module, i.e., a direct sum of copies of $A$.

(iv) The functor $\Hom_A(P,?)$ on the category of $A$-modules
is exact.
\end{theorem}

\begin{proof}
To prove that (i) implies (ii), take $N=P$. To prove that (ii)
implies (iii), take $M$ to be free (this can always be done since
any module is a quotient of a free module). To prove that (iii)
implies (iv), note that the functor $\Hom_A(P,?)$ is exact if $P$
is free (as $\Hom_A(A,N)=N$), so the statement follows, as if the
direct sum of two complexes is exact, then each of them is
exact. To prove that (iv) implies (i), let $K$ be the kernel of
the map $\alpha$, and apply the exact functor $\Hom_A(P,?)$ to
the exact sequence
$$
0\to K\to M\to N\to 0.
$$
\end{proof}

\begin{definition} A module satisfying any of the conditions
(i)-(iv) of Theorem \ref{proje} is said to be {\bf projective}.
\end{definition}

\subsection{Lifting of idempotents}

Let $A$ be a ring, and $I\subset A$ a nilpotent ideal.

\begin{proposition}\label{lif}
Let $e_0\in A/I$ be an idempotent, i.e., $e_0^2=e_0$.
There exists an idempotent $e\in A$ which is a lift of $e_0$
(i.e., it projects to $e_0$ under the reduction modulo $I$).
This idempotent is unique up to conjugation by an element of
$1+I$.
\end{proposition}

\begin{proof}
Let us first establish the statement in the case when
$I^2=0$. Note that in this case $I$ is a left and right module over $A/I$.
Let $e_*$ be any lift of $e_0$ to $A$.
Then $e_*^2-e_*=a\in I$, and $e_0a=ae_0$.
We look for $e$ in the form $e=e_*+b$, $b\in I$.
The equation for $b$ is $e_0b+be_0-b=a$.

Set $b=(2e_0-1)a$. Then
$$
e_0b+be_0-b=2e_0a-(2e_0-1)a=a,
$$
so $e$ is an idempotent. To classify other solutions, set
$e'=e+c$. For $e'$ to be an idempotent, we must have
$ec+ce-c=0$. This is equivalent to saying that $ece=0$ and
$(1-e)c(1-e)=0$, so $c=ec(1-e)+(1-e)ce=[e,[e,c]]$. Hence
$e'=(1+[c,e])e(1+[c,e])^{-1}$.

Now, in the general case, we prove by induction in $k$ that there
exists a lift $e_k$ of $e_0$ to $A/I^{k+1}$, and it is unique up
to conjugation by an element of $1+I^k$ (this is sufficient
as $I$ is nilpotent). Assume it is true for $k=m-1$, and let us
prove it for $k=m$. So we have an idempotent $e_{m-1}\in A/I^m$,
and we have to lift it to $A/I^{m+1}$. But $(I^m)^2=0$ in
$A/I^{m+1}$, so we are done.
\end{proof}

\begin{definition}
A complete system of orthogonal idempotents
in a unital algebra $B$ is a collection of elements
$e_1,...,e_n\in B$ such that
$e_ie_j=\delta_{ij}e_i$, and $\sum_{i=1}^n e_i=1$.
\end{definition}

\begin{corollary}\label{lif1}
Let $e_{01},...,e_{0m}$ be a complete system of orthogonal idempotents in
$A/I$. Then there exists a complete system of orthogonal idempotents
$e_1,...,e_m$ ($e_ie_j=\delta_{ij}e_i$, $\sum e_i=1$)
in $A$ which lifts $e_{01},...,e_{0m}$.
\end{corollary}

\begin{proof}
The proof is by induction in $m$. For $m=2$ this follows from
Proposition \ref{lif}. For $m>2$, we lift $e_{01}$ to $e_1$ using Proposition
\ref{lif}, and then apply the induction assumption to the algebra
$(1-e_1)A(1-e_1)$.
\end{proof}

\subsection{Projective covers}

Obviously, every finitely generated projective module over a finite dimensional 
algebra $A$ is a direct sum of indecomposable projective modules, so
to understand finitely generated projective modules over $A$, it suffices to
classify indecomposable ones.

Let $A$ be a finite dimensional algebra, with simple modules
$M_1,...,M_n$.

\begin{theorem}
(i) For each $i=1,...,n$ there exists a unique indecomposable
finitely generated projective module $P_i$ such that
$\dim\Hom(P_i,M_j)=\delta_{ij}$.

(ii) $A=\oplus_{i=1}^n (\dim M_i)P_i$.

(iii) any indecomposable finitely generated projective module over $A$ is isomorphic
to $P_i$ for some $i$.
\end{theorem}

\begin{proof}
Recall that $A/\Rad(A)=\oplus_{i=1}^n \End(M_i)$, and
$\Rad(A)$ is a nilpotent ideal. Pick a basis of $M_i$, and let
$e_{ij}^0=E_{jj}^i$, the rank 1 projectors projecting to the basis
vectors of this basis ($j=1,...,\dim M_i$). Then $e_{ij}^0$ are orthogonal
idempotents in $A/\Rad(A)$. So by Corollary \ref{lif1} we can
lift them to orthogonal
idempotents $e_{ij}$ in $A$. Now define $P_{ij}=Ae_{ij}$.
Then $A=\oplus_i \oplus_{j=1}^{\dim M_i}P_{ij}$, so $P_{ij}$ are
projective. Also, we have $\Hom(P_{ij},M_k)=e_{ij}M_k$,
so $\dim\Hom(P_{ij},M_k)=\delta_{ik}$. Finally, $P_{ij}$ is
independent of $j$ up to an isomorphism,
as $e_{ij}$ for fixed $i$ are conjugate under
$A^\times$ by Proposition \ref{lif}; thus we will denote $P_{ij}$
by $P_i$.

We claim that $P_i$ is indecomposable. Indeed, if $P_i=Q_1\oplus
Q_2$, then $\Hom(Q_l,M_j)=0$ for all
$j$ either for $l=1$ or for $l=2$, so either $Q_1=0$ or $Q_2=0$.

Also, there can be no other indecomposable finitely generated projective modules, since
any such module has to occur in the
decomposition of $A$. The theorem is proved.
\end{proof}

\end{document}